\documentclass{amsart}
\usepackage{amssymb,graphicx}
\theoremstyle{plain}
\newtheorem{Thm}{Theorem}
\newtheorem{Prop}[Thm]{Proposition}
\newtheorem{Lem}[Thm]{Lemma}

\newcommand{\z}{\textstyle}
\newcommand{\ia}{{\rm a}}\newcommand{\iA}{{\rm A}}\newcommand{\ib}{{\rm b}}\newcommand{\iB}{{\rm B}}
\newcommand{\iE}{{\rm E}}\newcommand{\iG}{{\rm G}}

\newcommand{\C}{\mathbb{C}}\newcommand{\M}{\mathbb{M}}\newcommand{\N}{\mathbb{N}}
\newcommand{\Q}{\mathbb{Q}}\newcommand{\R}{\mathbb{R}}\newcommand{\Z}{\mathbb{Z}}
\newcommand{\cE}{\mathcal{E}}
\newcommand{\cH}{\mathcal{H}}\newcommand{\cM}{\mathcal{M}}
\newcommand{\cO}{\mathcal{O}}\newcommand{\cS}{\mathcal{S}}

\newcommand{\BM}{\begin{smallmatrix}}\newcommand{\EM}{\end{smallmatrix}}
\newcommand{\BC}{\begin{cases}}\newcommand{\EC}{\end{cases}}
\newcommand{\VS}[1]{\z\sum\limits_{#1}}\newcommand{\VP}[1]{\z\prod\limits_{#1}}
\newcommand{\VT}[1]{\z\bigoplus\limits_{#1}}
\newcommand{\6}{\;\;\;\;\;\;}
\newcommand{\ang}[1]{\langle #1\rangle}\newcommand{\TL}{\vartriangleleft}
\newcommand{\bto}{\xrightarrow{\sim}}

\begin{document}
\title{graded rings of modular forms (1)}
\author{Suda Tomohiko}
\email{t@mshk1201.com}
\maketitle

For each integer $k\ge0$ and a congruence subgroup $\varGamma\subset{\rm SL}_2(\Z)$, let $\cM(\varGamma)_k$ be the $\C$-vector space of all modular forms of weight $k$ with respect to $\varGamma$ and $\cS(\varGamma)_k$ be the subspace of all cusp forms. Let $\N=\{1,2,\cdots\}$ usually and the semigroup
\[\M=\{0\}\cup\N.\]

We study the graded ring
\[\cM(\varGamma)_\M=\VT{k\in\M}\cM(\varGamma)_k\]
and the homogeneous ideal
\[\cS(\varGamma)_\M=\VT{k\in\M}\cS(\varGamma)_k,\]
especially for the cases $\varGamma$ is the principal congruence subgroup
\[\Gamma(N)=\big\{(\BM a&b\\c&d\EM)\in{\rm SL}_2(\Z)\,\big|\,(\BM a&b\\c&d\EM)\equiv(\BM 1&0\\0&1\EM)\mod N\big\}.\]
We will treat the cases $N$ is 2-power, 3-power or 5 in this paper with many identities among modoular forms.

When $(\BM 1&n\\0&1\EM)\in\varGamma$, we regard $\cM(\varGamma)\subset\C[[q^{\frac1n}]]$ via the Fourier expansion, where
\[q^{\frac1n}=e^{\frac{2\pi iz}n}\6(z\in\cH).\]
Note $\cS(\varGamma)\subset\C[[q^{\frac1n}]]q^{\frac1n}$.

Rational weight modular form are introduced, so graded ring and ideal
\[\cM(\varGamma)_{\frac1n\M}=\VT{\kappa\in\frac1n\M}\cM(\varGamma)_\kappa\]
\[\cS(\varGamma)_{\frac1n\M}=\VT{\kappa\in\frac1n\M}\cS(\varGamma)_\kappa\]
where $\frac1n\M=\{\frac kn \,|\, k\in\M\}$ are also studied.

\newpage

\section{Preparations}

\subsection{Dirichlet characters}

Generally for two sets $X$ and $Y$, let ${\rm map}(X,Y)$ be the maps from $X$ to $Y$. As for two groups $X$ and $Y$, let ${\rm Hom}(X,Y)$ be the homomorphisms from $X$ to $Y$.

We denote the integres $\Z$ and the complex numbers $\C$.

For $N\in\N$ and $X=(\Z/N\Z)^\times$ it is well known that ${\rm Hom}(X,\C^\times)\simeq X$ but it is not canonically (cf. \cite[Proposition 4.3.1]{DS}). So, we will give its elements concletely.

First, generally let ${\tt 1}_X$ be the trivial map in ${\rm Map}(X,\{1\})$.

For $N\in\N$ abbreviate ${\tt 1}_N={\tt 1}_{(\Z/N\Z)^\times}$. For example
\[{\rm Hom}((\Z/2\Z)^\times,\C^\times)=\{{\tt 1}_2\}.\]

Let $(\frac nm)$ is the Kronecker symbol(cf. \cite[p367]{DS}). That is the Legendre symbol if $m$ is odd prime and
\[\z(\frac n2)=\BC 1 &\text{if }n\in8\Z+\{\pm1\}\\-1 &\text{if }n\in8\Z+\{\pm3\} \\ 0 &\text{if }n\in2\Z\EC\]
Moreover define $(\frac n1)=1$ and
\[\z (\frac n{-1})=\BC 1 &\text{if }n\in\M \\ -1 &\text{oth } \EC \6 (\frac n0)=\BC 1 &\text{if }n\in\{\pm1\} \\ -1 &\text{oth} \EC\]\\

For a group $X$ and $x_1,x_2,\cdots,x_n\in X$ let $\ang{x_1,x_2,\cdots,x_n}$ be the subgroup generated by $x_1,x_2,\cdots,x_n$. For example $\Z=\ang1=\ang{-1}$.

Now, put $\chi_4=(\frac{-1}\bullet)$ then
\[{\rm Hom}((\Z/4\Z)^\times,\C^\times)=\ang{\chi_4}.\]

Let the $n$ th root of unity
\[1^{\frac 1n}=e^{\frac{2\pi i}n}\]
for example $1^{\frac12}=-1$, $1^{\frac13}=\frac{-1+\sqrt3i}2$ and $1^{\frac14}=i$.

For $n\ge3$, note $(\Z/2^n\Z)^\times=\ang{-1,5}$ and define $\chi_{2^n}\in{\rm Hom}((\Z/2^n\Z)^\times,\C^\times)$ by
\[\chi_{2^n}(5)=1^{1/2^{n-2}},\6\chi_{2^n}(-1)=1.\]
We see $\chi_8=(\frac\bullet 2)=(\frac2\bullet)$, $\chi_{2^{n+1}}^2=\chi_{2^n}$ and
\[{\rm Hom}((\Z/2^n\Z)^\times,\C^\times)=\ang{\chi_4,\chi_{2^n}}.\]

A number $g$ is a primitive root modulo $N$ if every number $a$ coprime to $N$ is congruent to a power of $g$ modulo $n$, that is $(\Z/N\Z)^\times=\ang g$. For a prime $p\neq2$ note $(\Z/p^r\Z)^\times\simeq\Z/(p^r-p^{r-1})\Z$ and define $\chi_{p^r}\in{\rm Hom}((\Z/p^r\Z)^\times,\C^\times)$ by
\[\chi_{p^r}(g)=1^{1/(p^r-p^{r-1})}\]
where $g$ is the minimum primitive root modulo $N$. For example
\[\chi_3(2)=-1,\6\chi_5(2)=i,\6\chi_7(3)=1^{\frac16},\6\chi_9(2)=1^{\frac16}.\]
We see
\[{\rm Hom}((\Z/p^r\Z)^\times,\C^\times)=\ang{\chi_{p^r}}.\]

Note $\chi_3=(\frac\bullet3)=(\frac{-3}\bullet)$, $\chi_5^2=(\frac\bullet5)$, $\chi_7^3=(\frac\bullet7)$ and $\chi_9^3=\chi_3$.\\

If $\gcd(M,N)=1$ then naturally
\[{\rm Hom}((\Z/MN\Z)^\times,\C^\times)\simeq{\rm Hom}((\Z/M\Z)^\times,\C^\times)\times{\rm Hom}((\Z/N\Z)^\times,\C^\times).\]
For example ${\rm Hom}((\Z/12\Z)^\times,\C^\times)=\ang{\chi_3}\times\ang{\chi_4}$.

Generally for $\chi\in{\rm Map}(X,\C)$ let $\overline\chi(x)=\overline{\chi(x)}$. For example, $\overline{\chi_5}=\chi_5^3$, $\overline{\chi_7}=\chi_7^6$ and $\overline{\chi_9}=\chi_9^3$.

With $G=(\Z/N\Z)^\times$ note the orthogonal relations
\[\z \frac1{|G|}\VS{x\in G}\phi(x)=\delta_{\phi,{\tt 1}}\6\frac1{|G|}\VS{\phi\in{\rm Hom}(G,\C^\times)}\phi(x)=\delta_{x,1}\] where
\[\delta_{x,y}=\delta_x(y)=\BC 1 &\text{if }x=y\\0 &\text{if }x\neq y\EC\]\\

Let $h\in\N$. Every $\chi\in{\rm Hom}((\Z/N\Z)^\times,\C^\times)$ lifts to $\chi\in{\rm Hom}((\Z/Nh\Z)^\times,\C^\times)$ via the natural projection in ${\rm Hom}((\Z/Nh\Z)^\times,(\Z/N\Z)^\times)$. The conductor of $\chi$ is the smallest positive integer $c$ such that $\chi$ arises from ${\rm Hom}((\Z/c\Z)^\times,\C^\times)$. $\chi$ is primitive if when its conductor is $N$.\\

Every $\chi\in{\rm Hom}((\Z/N\Z)^\times,\C^\times)$ extends to $\chi\in{\rm Map}(\Z,\C)$ by
\[n \mapsto \BC \chi(n\text{ mod }N) & \text{if }\gcd(n,N)=1 \\ 0 & {\rm oth}\EC\]
then $\chi$ is completely multiplicative.\\

Generally for two arithmetic functions $f,g\in{\rm Map}(\N,\C)$ the Dirichlet convolution $f*g\in{\rm Map}(\N,\C)$ is defined by
\[(f*g)(n)=\VS{d|n}f(d)g(\frac nd).\]
Note $g*f=f*g$ and $\delta_1*f=f$. Also note if $h\in{\rm Map}(\N,\C)$ is completely multiplicative then $h\cdot(f*g)=(h\cdot f)*(h\cdot g)$.

For two Dirichlet characters $\chi,\psi$, $\psi*\chi$ is denoted $\sigma_0^{\psi,\chi}$ in \cite[p129]{DS}.

For example $({\tt 1}_N*{\tt 1}_\N)(n)$ is the number of divisors of $n$ prime to $N$.

If $\chi\in{\rm Hom}((\Z/N\Z)^\times,\{\pm1\})$ and $\chi(n)=-1$ then
\[\z(\chi*{\tt 1}_\N)(n)=\frac12\Big(\VS{d|n}\chi(d)+\VS{d|n}\chi(\frac nd)\Big)=0.\]\\

Let $\sqrt{\chi_4}\in{\rm Map}((\Z/4\Z)^\times,\C^\times)$ by $1\mapsto1$ and $-1\mapsto i$.

Also let $\xi_8=\overline{\sqrt{\chi_4}}\chi_8\in{\rm Map}((\Z/8\Z)^\times,\C^\times)$ then $\xi_8^2=\chi_4$.

\newpage

\subsection{Congruence groups}

For $N\in\N$ put (cf. \cite[Definition 1.2.1]{DS})
\[\Gamma_0(N)=\big\{(\BM a&b\\c&d\EM)\in{\rm SL}_2(\Z) \,\big|\, c\in N\Z\big\}\]
\[\Gamma^0(N)=\big\{(\BM a&b\\c&d\EM)\in{\rm SL}_2(\Z) \,\big|\, b\in N\Z\big\}\]
and
\[\Gamma_1(N)=\big\{(\BM a&b\\c&d\EM)\in\Gamma_0(N) \,\big|\, d\in N\Z+1\big\}\]
\[\Gamma^1(N)=\big\{(\BM a&b\\c&d\EM)\in\Gamma^0(N) \,\big|\, d\in N\Z+1\big\}\]

We especially focus on the principal congruence subgroup
\[\Gamma(N)=\Gamma_1(N)\cap\Gamma^1(N)\]

A subgroup $\varGamma\subset{\rm SL}_2(\Z)$ is called a congruence subgroup if there exists $N\in\N$ such that $\Gamma(N)\subset\varGamma$. The level of $\varGamma$ is then the smallest such $N$. As for $N=1,2$ note
\[\Gamma(N)=\Gamma_0(N)\cap\Gamma^0(N)\]\\

Let the natural projection $\natural_N:\Z\to\Z/N\Z$ and
\[d_N:(\BM a&b\\c&d\EM)\mapsto\natural_N(d).\]

Since $(\BM a&b\\c&d\EM)(\BM a'&b'\\c'&d'\EM)=(\BM aa'+bc'&ab'+bd'\\ca'+dc'&cb'+dd'\EM)$ it follows
\begin{Prop}$d_{MN}\in{\rm Hom}(\Gamma_0(M)\cap\Gamma^0(N),(\Z/MN\Z)^\times)$.\end{Prop}

For a subgroup $G\subset(\Z/N\Z)^\times$, we define
\[\Gamma_G(N)=\{\gamma\in\Gamma_0(N) \,|\, d_N(\gamma)\in G\}.\]
In particular, $\Gamma_{(\Z/N\Z)^\times}(N)=\Gamma_0(N)$ and $\Gamma_{\{1\}}(N)=\Gamma_1(N)$. We have
\[d_N:\Gamma_G(N)/\Gamma_1(N)\bto G.\]

Let $\varGamma\subset{\rm SL}_2(\Z)$ be a congruence subgroup. We regard
\[{\rm Map}((\Z/N\Z)^\times,\C^\times)\subset{\rm Map}(\varGamma,\C^\times)\]
via $d_N\in{\rm Map}(\varGamma,(\Z/N\Z)^\times)$. Thus
\[d_N\in{\rm Hom}(\varGamma,(\Z/N\Z)^\times) \Longrightarrow {\rm Hom}((\Z/N\Z)^\times,\C^\times)\subset{\rm Hom}(\varGamma,\C^\times).\]\\

Generally, for a group $G$ we put the action ${\TL}:G\curvearrowleft G$ by $g\TL h=h^{-1}gh$. Since
\[(\BM a&b\\c&d\EM)\TL(\BM h&0\\0&1\EM)=(\BM a&b/h\\ch&d\EM)\]
we see
\[\Gamma(N)\TL(\BM N&0\\0&1\EM)=\Gamma_{\ang{N+1}}(N^2).\]
For example $\Gamma(2)\TL(\BM 2&0\\0&1\EM)=\Gamma_0(4)$ and $\Gamma(3)\TL(\BM 3&0\\0&1\EM)=\Gamma_{\ang4}(9)$. Also, since
\[(\BM a&b\\c&d\EM)\TL(\BM 0&-1\\1&0\EM)=(\BM d&-c\\-b&a\EM)\]
we see
\[\Gamma^0(N)\TL(\BM 0&-1\\1&0\EM)=\Gamma_0(N)\]
\[\Gamma^1(N)\TL(\BM 0&-1\\1&0\EM)=\Gamma_1(N)\]

\newpage

For $N\in\N$ let
\[b_N:(\BM a&b\\c&d\EM)\mapsto1^{\frac{bd}N},\6c_N:(\BM a&b\\c&d\EM)\mapsto1^{-\frac{cd}N}.\]

\begin{Lem}\label{b8c8}We have
\[b_8,c_8\in{\rm Hom}(\Gamma_0(2)\cap\Gamma^0(2),\C^\times)\]
\[b_{32}\in{\rm Hom}(\Gamma_0(2)\cap\Gamma^0(4),\C^\times) \6 c_{32}\in{\rm Hom}(\Gamma_0(4)\cap\Gamma^0(2),\C^\times)\]
\end{Lem}
\begin{proof}If $\gamma=(\BM a&2b\\2c&d\EM),\gamma'=(\BM a'&2b'\\2c'&d'\EM)\in\Gamma_0(2)\cap\Gamma^0(2)$ then
\[b_8(\gamma\gamma')=1^{\frac{(ab'+bd')dd'}4}=b_8(\gamma)b_8(\gamma')\]
since $ad=4bc+1$, $d'^2\equiv1 \mod 4$.

If $\gamma=(\BM a&4b\\2c&d\EM),\gamma'=(\BM a'&4b'\\2c'&d'\EM)\in\Gamma_0(2)\cap\Gamma^0(4)$ then
\[b_{32}(\gamma\gamma')=1^{\frac{(ab'+bd')dd'}8}=b_{32}(\gamma)b_{32}(\gamma')\]
since $ad=8bc+1$, $d'^2\equiv1 \mod 8$.
\end{proof}

\begin{Lem}\label{b3c3}We have
\[b_9\in{\rm Hom}(\Gamma^0(3),\C^\times) \6 c_9\in{\rm Hom}(\Gamma_0(3),\C^\times)\]
\end{Lem}
\begin{proof}If $\gamma=(\BM a&3b\\c&d\EM),\gamma'=(\BM a'&3b'\\c'&d'\EM)\in\Gamma^0(3)$ then
\[b_9(\gamma\gamma')=1^{\frac{(ab'+bd')dd'}3}=(-1)^{b'd'+bd}=b_9(\gamma)b_9(\gamma')\]
since $ad=3bc+1$ and $d'^2\equiv1\mod 3$.
\end{proof}

\vspace{0.5cm}

Let $\Q$ be the rational numbers and $\cH=\{\tau\in\C \,|\, \Im \tau>0\}$.

For a congruence group $\varGamma$, denote the modular curve
\[X(\varGamma)=\varGamma\setminus(\cH\cup\Q\cup\{\infty\})\]
and put
\[\varepsilon_2(\varGamma)=\text{the number of elliptic points of period 2 in }X(\varGamma)\]
\[\varepsilon_3(\varGamma)=\text{the number of elliptic points of period 3 in }X(\varGamma)\]
\[\varepsilon_\infty(\varGamma)=\text{the number of cusps of }X(\varGamma)\]
\[d(\varGamma)={\rm deg}\big(X(\varGamma)\to X({\rm SL_2}(\Z))\big)\]
\begin{align*}g(\varGamma)&=\text{the genus of }X(\varGamma)\\
&\z =1+\frac{d(\varGamma)}{12}-\frac{\varepsilon_2(\varGamma)}4-\frac{\varepsilon_3(\varGamma)}3-\frac{\varepsilon_\infty(\varGamma)}2.\end{align*}

In addtion, put
\[\varepsilon_\infty^{\rm reg}(\varGamma)=\text{the number of regular cusps of }X(\varGamma)\]
\[\varepsilon_\infty^{\rm irr}(\varGamma)=\text{the number of irregular cusps of }X(\varGamma)\]

\newpage

\subsection{Power series}

The formal power series ring with a ring $R$ and a variable $x$ are denoted $R[[x]]$.

On $\C[[x]]$ rescaling by $h\in\N$ is defined by
\[\big(\VS{n\in\M}a_nx^n\big)^{\ang h}=\VS{n\in\M}a_nx^{hn}.\]

Define
\[f^{/n/}=\frac{f^{\ang{\frac1n}n}}f,\6 f^{"n"}=\frac{f^{\ang nn}}f\]
for $n\in\N$. For a prime number $p$, we see
\[f^{/p/"p"}=\frac{f^{/p/\ang pp}}{f^{/p/}}=\frac{f^{p^2+1}}{f^{\ang{\frac1p}p}f^{\ang{p}p}}=\frac{f^{"p"\ang{\frac1p}p}}{f^{"p"}}=f^{"p"/p/}.\]
Also, for another prime number $s$
\[f^{/p//s/}=\frac{f^{/p/\ang{\frac1s}s}}{f^{/p/}}=\frac{f^{\ang{\frac1{ps}}ps}f}{f^{\ang{\frac1s}s}f^{\ang{\frac1p}p}}=f^{/s//p/}\]
similarly $f^{/p/"s"}=f^{"s"/p/}$ and $f^{"p""s"}=f^{"s""p"}$.\\

We define $n$-th square for $\phi\in x^r+\C[[x]]x^{r+1}$ by
\[\sqrt[n]\phi\in x^{\frac rn}+\C[[x^{\frac1n}]]x^{\frac{r+1}n},\6\big(\sqrt[n]\phi\big)^n=\phi.\]
For example $\sqrt{1+x}=1+\frac12x-\frac18x^2+\frac1{16}x^3-\frac5{128}x^4+\cdots$.

We also define the complex conjugate $\sigma_4$ by
\[\big(\VS{n\in\M}a_nx^n\big)^{\sigma_4}=\VS{n\in\M}\overline{a_n}x^n.\]\\

Suppose $q=e^{2\pi i\tau}$ and $f\in\cO(\cH)\cap\C[[q]]$.

We see $f^{\ang h}(\tau)=f(h\tau)$.

Note $(f^{\sigma_4})(-\overline \tau)=\overline{f(\tau)}$ since $e^{ri(-x+yi)}=\overline{e^{ri(x+yi)}}$ for $r,x,y\in\R$.\\

\begin{Lem}\label{sumPM}If $A,B\subset\M$ and $A,B\neq\emptyset$ then
\[\#A+\#B-1\le \#(A+B).\]
\end{Lem}
\begin{proof}We may assume $\#B<\infty$. We see
\[\big(\{\min A\}+B\big)\cup\big(A+\{\max B\}\big)\subset A+B.\]\end{proof}

For a subspace $V\subset\C[[x]]$ remark
\[\dim V=\#\big\{i\in\M \,\big|\, V\cap(x^i+\C[[x]]x^{i+1})\neq\emptyset\big\}.\]

If $V,V'\subset\C[[x]]$ are subspaces $\neq\{0\}$ then
\[\dim V+\dim V'-1\leq \dim(VV')\]
by Lemma \ref{sumPM}. As $V=V'$ we see $\dim V\leq\frac{\dim(V\cdot V)+1}2$.

\newpage

\section{Fundamental theories}

\subsection{Integer weight theories}

First, denote the real numbers $\R$ and put
\[{\rm GL}_2^+(\R)=\big\{(\BM a&b\\c&d\EM) \,\big|\, a,b,c,d\in\R,\,ad-bc>0\big\}.\]

For $\gamma=(\BM a&b\\c&d\EM)\in{\rm GL}_2^+(\R)$ and $\tau\in\cH$ we write $\gamma\tau=\z\frac{a\tau+b}{c\tau+d}$, then the map $(\gamma,\tau)\mapsto\gamma\tau$ defines an action ${\rm GL}_2^+(\R)\curvearrowright\cH$.\\

Next, we introduce the factor of automorphy $J_k(\gamma,\tau)=(c\tau+d)^k$.

Let $k\in\M$. For $\gamma=(\BM a&b\\c&d\EM)\in{\rm GL}_2^+(\R)$ and $f:\cH\to\C$ we define $f|_k:\cH\to\C$ by
\[f|_k\gamma:\tau\mapsto \frac{f(\gamma\tau)}{J_k(\gamma,\tau)}\]
then the map $(f,\gamma)\mapsto f|_k\gamma$ defines an action ${\rm Map}(\cH,\C)\curvearrowleft{\rm GL}_2^+(\R)$.

Let $\cO(X)=\{f:X\to\C \,|\, f:{\rm holomorphic}\}$, then the map $(f,\gamma)\mapsto f|_k\gamma$ also defines an action $\cO(\cH)\curvearrowleft{\rm GL}_2^+(\R)$.\\

For a congruence subgroup $\varGamma\subset{\rm SL}_2(\Z)$, we define
\[\cM'(\varGamma)_k=\big\{f\in\cO(\cH) \,\big|\, \forall\gamma\in\varGamma.f|_k\gamma=f\big\}\]
\[\cM(\varGamma)_k=\big\{f\in\cM'(\varGamma)_k \,\big|\, \forall\alpha\in{\rm SL}_2(\Z).f|_k\alpha(\tau)\text{ is bounded as }\Im z\to\infty\big\}\]
\[\cS(\varGamma)_k=\big\{f\in\cM'(\varGamma)_k \,\big|\, \forall\alpha\in{\rm SL}_2(\Z).f|_k\alpha(\tau)\to0\text{ as }\Im z\to\infty\big\}\]

Note $\cM(\varGamma)_0=\C$ and $\cS(\varGamma)_0=\{0\}$.

Since $(\BM 1&N\\0&1\EM)\in\Gamma(N)$, we regard $\cM(\Gamma(N))_k\subset\C[[q^{\frac1N}]]$ via the Fourier expansion, where
\[q^{\frac1N}=e^{\frac{2\pi i\tau}N}\6(\tau\in\cH).\]
Note $\cS(\Gamma(N))_k\subset\C[[q^{\frac1N}]]q^{\frac1N}$.
\\

For $\chi\in{\rm Hom}(\varGamma,\C^\times)$, we define
\[\cM(\varGamma)_{(k,\chi)}=\big\{f\in\cM(\ker\chi)_k \,\big|\, f|_k\gamma=\chi(\gamma)f\text{ for all }\gamma\in\varGamma\big\}\]
then
\[\cM(\varGamma)_{(k,{\tt 1}_\varGamma)}=\cM(\varGamma)_k\]
\[(\BM -1&0\\0&-1\EM)\in\varGamma,\;\chi(\BM -1&0\\0&-1\EM)\neq(-1)^k \Longrightarrow \cM(\varGamma)_{(k,\chi)}=\{0\}\]
\[\sigma_4:\cM(\varGamma)_{(k,\chi)}\simeq\cM(\varGamma)_{(k,\overline{\chi})}\]
\[\cM(\varGamma)_{(k,\chi)}\cM(\varGamma)_{(k',\chi')}\subset\cM(\varGamma)_{(k+k',\chi\chi')}\]\\

We write $\cM(M,N)=\cM(\Gamma_0(M)\cap\Gamma^0(N))$ and $\cM(N)=\cM(N,N)$ (as strings). In particular $\cM(N,1)=\cM(\Gamma_0(N))$ and $\cM(1)=\cM(\Gamma(1))$, $\cM(2)=\cM(\Gamma(2))$.

\newpage

\begin{Prop}Let $\varGamma$ be a congruence subgroup, $k\in\M$ and $\chi\in{\rm Hom}(\varGamma,\C^\times)$.

For $\alpha\in{\rm GL}_2^+(\R)$ if $\varGamma\TL\alpha\subset{\rm SL}_2(\Z)$ then
\[f\mapsto f|_k\alpha\;:\;\cM(\varGamma)_{(k,\chi)}\to\cM(\varGamma\TL\alpha)_{(k,\chi\TL\alpha)}\]
where $(\chi\TL\alpha)(x)=\chi(\alpha x\alpha^{-1})$.
\end{Prop}
\begin{proof}For $\gamma\in\varGamma$
\[(f|_k\alpha)|_k(\alpha^{-1}\gamma\alpha)=\chi(\gamma)f|_k\alpha\]
\end{proof}

If $f\in\cO(\cH)\cap\C[[q^{\frac1N}]]$ then $f^{\ang h}=f|(\BM h&0\\0&1\EM)$. It is important that
\[\cM(\Gamma(N))_k=\cM(\Gamma_{\ang{N+1}}(N^2))_k^{\ang{\frac1N}}\]
Since $\Gamma_0(N)\TL(\BM h&0\\0&1\EM)$ is generated by $\Gamma_0(hN)$ and $(\BM 1&\frac1h\\0&1\EM)$ we see
\[\cM(N,1)_{(k,\chi)}\,\!^{\ang h}=\cM(hN,1)_{(k,\chi)}\cap\C[[q^h]]\]
for $\chi\in{\rm Hom}((\Z/N\Z)^\times,\C^\times)$.

If $f=\VS{n\in\M}a_nq^{\frac nN}\in\cO(\cH)$ then the twisting $f|(\BM 1&1\\0&1\EM)=\VS{n\in\M}1^{\frac nN}a_nq^{\frac nN}$.\\[0.5cm]

Dimension formula is well-known (\cite[\S3]{DS} or \cite[\S6.3]{St}). For $k\in2\N$
\[\z \dim\cM(\varGamma)_k=(k-1)(g(\varGamma)-1)+\big[\frac k4\big]\varepsilon_2(\varGamma)+\big[\frac k3\big]\varepsilon_3(\varGamma)+\frac k2\varepsilon_\infty(\varGamma)\]
\[\dim\cS(\varGamma)_k=\BC \dim\cM(\varGamma)_k-\varepsilon_\infty(\varGamma) & k\geqq4 \\ g(\varGamma) & k=2 \EC\]
where [ ] is the greatest integer function. When $(\BM -1&0\\0&-1\EM)\notin\varGamma$, for $k\in2\N+1$
\[\z \dim\cM(\varGamma)_k=(k-1)(g(\varGamma)-1)+\big[\frac k3\big]\varepsilon_3(\varGamma)+\frac k2\varepsilon_\infty^{\rm reg}(\varGamma)+\frac{k-1}2\varepsilon_\infty^{\rm irr}(\varGamma)\]
\[\dim\cS(\varGamma)_k=\dim\cM(\varGamma)_k-\varepsilon_\infty^{\rm reg}(\varGamma)\]
As for $k=1$
\[\dim\cM(\varGamma)_1\BC =\frac 12\varepsilon_\infty^{\rm reg}(\varGamma) & \text{if }\varepsilon_\infty^{\rm reg}(\varGamma)>2g(\varGamma)-2 \\[0.1cm] \geq\frac 12\varepsilon_\infty^{\rm reg}(\varGamma) & \text{if }\varepsilon_\infty^{\rm reg}(\varGamma)\leq2g(\varGamma)-2 \EC\]
\[\z\dim\cS(\varGamma)_1=\dim\cM(\varGamma)_1-\frac 12\varepsilon_\infty^{\rm reg}(\varGamma)\]

\newpage

\subsection{Eisenstein series of level 1}

For $k\in2\N$, put
\[\iE_{1,k}(\tau)=1+\z\frac1{\zeta(k)}\VS{M\in\N}\VS{N\in\Z}\frac1{(M\tau+N)^k}.\]
where $\zeta(k)=\VS{N\in\Z}\frac1{N^k}$ is the Rieman zeta function
\[\z\zeta(2)=\frac{\pi^2}6,\;\zeta(4)=\frac{\pi^4}{90},\;\zeta(6)=\frac{\pi^6}{945},\;\zeta(8)=\frac{\pi^8}{9450},\cdots\]
The summations are defined by $\VS{N\in\N}a_N=\displaystyle\lim_{N\to\infty}\VS{n=1}^Na_n$ and $\VS{N\in\Z}a_N=\displaystyle\lim_{N\to\infty}\VS{n=-N}^Na_n$.

We compute the Laurent expansion
\[\iE_{1,k}=1-\frac{2k}{B_k}\VS{n\in\N}\VS{d|n}d^{k-1}q^n\]
where $B_k$ is the $k$-th Bernoulli number which satisfies $2\zeta(k)=-\frac{(2\pi i)^k}{k!}B_k$ for $k\ge2$.

For example
\begin{align*}\iE_{1,2}&=1-24\VS{n\in\N}\VS{d|n}dq^n\\
\iE_{1,4}&=1+240\VS{n\in\N}\VS{d|n}d^3q^n\\
\iE_{1,6}&=1-504\VS{n\in\N}\VS{d|n}d^5q^n\end{align*}

We see $\iE_{1,k}\in\cM(1)_k$ for $k\geq4$. As for $k=2$, for $(\BM a&b\\c&d\EM)\in\Gamma(1)$
\[\iE_{1,2}|_2(\BM a&b\\c&d\EM)(z)=\iE_{1,2}(z)-\frac{6ci}{\pi(cz+d)}.\]
$\iE_{1,2}-\frac3\pi{\rm Im}$ is $|_2$ invariant under $\Gamma(1)$, but it is not holomorphic (\cite[p18]{DS}).\\

The dimension formula states for $k\in\M$
\[\dim\cM(1)_{2k}=\big[\z\frac k6\big]+\BC0&\text{if }k\in6\M+1\\1&\text{otherwise}\EC\]

In particular $\dim\cM(1)_8=1$ and $\iE_{1,8}-\iE_{1,4}^2\in\cM(1)_8\cap\C[[q]]q=\{0\}$. That is $\iE_{1,8}=\iE_{1,4}^2$ and a relation between divisor sums is obtained : for $n\in\N$
\[\VS{d|n}d^7=\VS{d|n}d^3+120\VS{i=1}^{n-1}\Big(\VS{d|i}d^3\VS{d|(n-i)}d^3\Big).\]
Similarly $\iE_{1,10}=\iE_{1,4}\iE_{1,6}$ and $\iE_{1,14}=\iE_{1,4}^2\iE_{1,6}$. We will expand these identities for lager weights.

\newpage

Let $\tau\in\cH$ and $\Lambda_\tau=\tau\Z+\Z$. The Weierstrass $\wp$ function (cf \cite[p31]{DS}) is
\[\wp_\tau(z)=\z\frac1{z^2}+\VS{\omega\in\Lambda_\tau\setminus\{0\}}\big(\frac1{(z+\omega)^2}-\frac1{\omega^2}\big).\]
The sum is arranged so that $\frac1{(z+\omega)^2}-\frac1{\omega^2} = O(w^{-3})$ which makes the sum over $\Lambda_\tau\setminus\{0\}$ absolutely convergent. We see
\[\wp_\tau(z)'=\VS{\omega\in\Lambda_\tau}\frac{-2}{(\tau+\omega)^3}\]
and the field of meromorphic functions on $\C/(\tau\Z+\Z)$ is $\C(\wp_\tau,\wp_\tau')$.

For even $k\geq4$, note
\[\iE_{1,k}(\tau)=\z\frac1{2\zeta(k)}\VS{\omega\in\Lambda_\tau\setminus\{0\}}\frac1{\omega^k}\]
and put $\cE_k=(k-1)\zeta(k)\iE_{1,k}$.

\begin{Lem}The Laurent eapansion of $\wp_\tau$ is
\[\wp_\tau(z)=\z\frac1{z^2}+2\VS{n\in2\N}\cE_{n+2}(\tau)z^n\]
for all $z$ such that $0<|z|<\inf\big\{|\omega| \,\big|\, \omega\in\Lambda_\tau\setminus\{0\}\big\}$.
\end{Lem}
\begin{proof}If $|z|<|\omega|$ then
\[\z\frac1{(z+\omega)^2}-\frac1{\omega^2}=\frac1{\omega^2}\big(\frac1{(1-\frac z\omega)^2}-1\big)\]

\end{proof}

\begin{Prop}For $k\in\M$
\[\z\frac{(k+1)(2k+9)}6\cE_{2k+8}=\VS{i=0}^k\cE_{2i+4}\cE_{2k-2i+4}\]
\end{Prop}
\begin{proof}By the above Lemma
\[\wp_\tau(z)=\z\frac1{z^2}+2\cE_4(\tau)z^2+O(z^4),\]
\[\z\frac12\wp_\tau'(z)=-\frac1{z^3}+2\cE_4(\tau)z+O(z^3).\]
We get
\[\z\frac16\wp_\tau''(z)=\wp_\tau(z)^2-\frac{10}3\cE_4(\tau)\]
since both sides work out to
\[\z\frac1{z^4}+\frac23\cE_4(\tau)+O(z^2).\]
The assertion follows from computing the coefficient of $z^{2k+4}$ and
\[\z\frac{(k+3)(2k+5)}6-1=\frac{(k+1)(2k+9)}6.\]
\end{proof}

\newpage

\subsection{Eisenstein series of weight 2 without character}

For $N\geq2$ put
\[\iE_{N,2}=\z\frac1{N-1}\big(N\iE_{1,2}^{\ang N}-\iE_{1,2}\big)=1+\frac{24}{N-1}\VS{n\in\N}\VS{d|n,N\nmid d}dq^n\]
then $\iE_{N,2}\in\cM(N,1)_2$ since for $\gamma=(\BM c&b\\c&d\EM)\in\Gamma_0(N)$
\[\iE_{N,2}|\gamma=\z\frac1{N-1}\big(N\iE_{1,2}|(\BM a&Nb\\c/N&d\EM)(\BM N&0\\0&1\EM)-\iE_{1,2}|\gamma\big).\]

Compute
\[\iE_{1,2}^{\ang N}\big|(\BM 0&-1\\1&0\EM)=\iE_{1,2}\big|(\BM 0&-1\\1&0\EM)(\BM 1&0\\0&N\EM)=\z\frac1{N^2}\big(\iE_{1,2}^{\ang{\frac1N}}-\frac{6Ni}{\pi\tau}\big)\]
thus
\[N\iE_{N,2}\big|(\BM 0&-1\\1&0\EM)=\z\frac 1{N-1}\big(\big(\iE_{1,2}^{\ang{\frac1N}}-\frac{6Ni}{\pi\tau}\big)-N(\iE_{1,2}-\frac{6i}{\pi\tau}\big)\big)=-\iE_{N,2}^{\ang{\frac1N}}\]\\

Put $\iE_{2,2}^{\gets}=\iE_{2,2}^{\ang{\frac12}}\big|(\BM 1&1\\0&1\EM)$ then
\[\z\iE_{2,2}=\frac12\big(\iE_{2,2}^{\ang{\frac12}}+\iE_{2,2}^{\gets}\big)\]
since $\VS{d|2n,2\nmid d}d=\VS{d|n,2\nmid d}d$. Note $\iE_{2,2}^\gets|(\BM 0&-1\\1&0\EM)=\iE_{2,2}^\gets$ since
\[(\BM \frac1n&0\\0&1\EM)(\BM 1&1\\0&1\EM)(\BM 0&-1\\1&0\EM)=(\BM 1&0\\n&1\EM)(\BM \frac1n&0\\0&1\EM)(\BM 1&-1\\0&1\EM).\]
We have $\cM(1)_k=\big\{f\in\cM(2)_k \,\big|\, f|_k(\BM 1&1\\0&1\EM)=f|_k(\BM 0&-1\\1&0\EM)=f\big\}$ and
\[\z\iE_{1,4}=\frac13\big(\iE_{2,2}^{\ang{\frac12}}-1^{\frac13}\iE_{2,2}^{\gets}\big)\big(\iE_{2,2}^{\ang{\frac12}}-1^{\frac23}\iE_{2,2}^{\gets}\big)\]
\[\iE_{1,6}=\iE_{2,2}\iE_{2,2}^{\ang{\frac12}}\iE_{2,2}^{\gets}\]

Put $\iE_{3,2}^{\nwarrow}=\iE_{3,2}^{\ang{\frac13}}\big|(\BM 1&1\\0&1\EM)$ and $\iE_{3,2}^{\swarrow}=\iE_{3,2}^{\ang{\frac13}}\big|(\BM 1&-1\\0&1\EM)$ then
\[\z\iE_{3,2}=\frac13\big(\iE_{3,2}^{\ang{\frac13}}+\iE_{3,2}^{\nwarrow}+\iE_{3,2}^{\swarrow}\big)\]
\[\iE_{1,4}^2=\iE_{3,2}\iE_{3,2}^{\ang{\frac13}}\iE_{3,2}^{\nwarrow}\iE_{3,2}^{\swarrow}\]
\[\z\iE_{1,6}=\frac18\big(\iE_{3,2}^{\ang{\frac13}}+\iE_{3,2}^{\nwarrow}\big)\big(\iE_{3,2}^{\ang{\frac13}}+\iE_{3,2}^{\swarrow}\big)\big(\iE_{3,2}^{\nwarrow}+\iE_{3,2}^{\swarrow}\big)\]

Put
\[\iE_{5,2}^{\nwarrow}=\iE_{5,2}^{\ang{\frac15}}\big|(\BM 1&2\\0&1\EM),\6\iE_{5,2}^{\nearrow}=\iE_{5,2}^{\ang{\frac15}}\big|(\BM 1&1\\0&1\EM),\]
\[\iE_{5,2}^{\swarrow}=\iE_{5,2}^{\ang{\frac15}}\big|(\BM 1&-2\\0&1\EM),\6\iE_{5,2}^{\searrow}=\iE_{5,2}^{\ang{\frac15}}\big|(\BM 1&-1\\0&1\EM),\]
then
\[\z\iE_{5,2}=\frac15\big(\iE_{5,2}^{\ang{\frac15}}+\iE_{5,2}^{\nearrow}+\iE_{5,2}^{\nwarrow}+\iE_{5,2}^{\swarrow}+\iE_{5,2}^{\searrow}\big)\]
\[\iE_{1,6}^2=\iE_{5,2}\iE_{5,2}^{\ang{\frac15}}\iE_{5,2}^{\nearrow}\iE_{5,2}^{\nwarrow}\iE_{5,2}^{\swarrow}\iE_{5,2}^{\searrow}\]

Put
\[\iE_{7,2}^{\nwarrow}=\iE_{7,2}^{\ang{\frac17}}\big|(\BM 1&3\\0&1\EM),\6\iE_{7,2}^{\uparrow}=\iE_{7,2}^{\ang{\frac17}}\big|(\BM 1&2\\0&1\EM),\6\iE_{7,2}^{\nearrow}=\iE_{7,2}^{\ang{\frac17}}\big|(\BM 1&1\\0&1\EM),\]
\[\iE_{7,2}^{\swarrow}=\iE_{7,2}^{\ang{\frac17}}\big|(\BM 1&-3\\0&1\EM),\6\iE_{7,2}^{\downarrow}=\iE_{7,2}^{\ang{\frac17}}\big|(\BM 1&-2\\0&1\EM),\6\iE_{7,2}^{\searrow}=\iE_{7,2}^{\ang{\frac17}}\big|(\BM 1&-1\\0&1\EM),\]
then
\[\z\iE_{7,2}=\frac17\big(\iE_{7,2}^{\ang{\frac17}}+\iE_{7,2}^{\nearrow}+\iE_{7,2}^{\uparrow}+\iE_{7,2}^{\nwarrow}+\iE_{7,2}^{\swarrow}+\iE_{7,2}^{\downarrow}+\iE_{7,2}^{\searrow}\big)\]
\[\iE_{1,4}^4=\iE_{7,2}\iE_{7,2}^{\ang{\frac17}}\iE_{7,2}^{\nearrow}\iE_{7,2}^{\uparrow}\iE_{7,2}^{\nwarrow}\iE_{7,2}^{\swarrow}\iE_{7,2}^{\downarrow}\iE_{7,2}^{\searrow}\]

\newpage

\subsection{Eisenstein series of weight 1}

Let $\chi\in{\rm Hom}((\Z/N\Z)^\times,\C^\times)\setminus\{{\tt 1}_N\}$.

We denote the first generalized Bernoulli number by $B_{\chi}$. If the conductor of $\chi$ is $N$ then $B_{\chi}=\frac1N\VS{a=1}^N\chi(a)a$. Note $B_{\chi}=0$ if $\chi(-1)\neq-1$.

When $\chi(-1)=-1$, put the Eisenstein series
\[\iE_\chi=1-\z\frac2{B_{\chi}}\VS{n\in\N}(\chi*{\tt 1}_\N)(n)q^n\]
then $\iE_\chi\in\cM(N,1)_{(1,\chi)}$.

The dimension formula states $\dim\cM(p,1)_2=1$ for $p=3,5,7,13$ thus
\[\iE_{3,2}=\iE_{\chi_3}^2\]
\[\iE_{5,2}=\iE_{\chi_5}\iE_{\overline{\chi_5}}\]
\[\iE_{7,2}=\iE_{\chi_7^3}^2=\iE_{\chi_7}\iE_{\overline{\chi_7}}\]
\[\iE_{13,2}=\iE_{\chi_{13}}\iE_{\overline{\chi_{13}}}=\iE_{\chi_{13}^3}\iE_{\overline{\chi_{13}^3}}=\iE_{\chi_{13}^5}\iE_{\overline{\chi_{13}^5}}\]\\

Moreover let $\psi\in{\rm Hom}((\Z/M\Z)^\times,\C^\times)\setminus\{{\tt 1}_M\}$. When $\psi\chi(-1)=-1$, put
\[\iG_{\psi,\chi}=\VS{n\in\N}(\psi*\chi)(n)q^n\]
then $\iG_{\psi,\chi}\in\cM(MN,1)_{(1,\psi\chi)}$.

See \cite[\S4.8]{DS} or \cite[\S5.3]{St} for further details.\\[0.5cm]

Put $\iE_{\chi_3}^{\nwarrow}=\iE_{\chi_3}^{\ang{\frac13}}\big|(\BM 1&1\\0&1\EM)$ and $\iE_{\chi_3}^{\swarrow}=\iE_{\chi_3}^{\ang{\frac13}}\big|(\BM 1&-1\\0&1\EM)$ then
\[\z\iE_{\chi_3}=\frac13\big(\iE_{\chi_3}^{\ang{\frac13}}+\iE_{\chi_3}^{\nwarrow}+\iE_{\chi_3}^{\swarrow}\big)\]
\[\z\iE_{\chi_3}^{\ang{\frac13}}+1^{\frac13}\iE_{\chi_3}^{\nwarrow}+1^{\frac23}\iE_{\chi_3}^{\swarrow}=3\VS{n\in3\M+2}(\chi_3*{\tt 1}_\N)(n)q^{\frac n3}=0\]
and
\[\iE_{1,4}=\iE_{\chi_3}\iE_{\chi_3}^{\ang{\frac13}}\iE_{\chi_3}^{\nwarrow}\iE_{\chi_3}^{\swarrow}\]

\newpage

\subsection{Half integer weight theories}

Let $k\in2\M+1$.

For $\gamma=(\BM a&b\\c&d\EM)\in{\rm SL}_2(\Z)$ and $\tau\in\cH$ put $J_{\frac k2}(\gamma,\tau)=(\frac cd)\sqrt{J(\gamma,\tau)}^k$ where $\sqrt z$ is the "principal" determination of the square root of $z$ i.e. the one with argument in $(-\frac\pi2,\frac\pi2]$. Note for $f\in\C[[q]]$, in general $\sqrt f(\tau)\neq\sqrt{f(\tau)}$.

For $f:\cH\to\C$ we define $f|_{\frac k2}:\cH\to\C$ by
\[f|_{\frac k2}\gamma:\tau\mapsto \frac{f(\gamma\tau)}{J_{\frac k2}(\gamma,\tau)}\]
The map $\gamma\mapsto f|_{\frac k2}\gamma$ is not an action : for $\gamma=(\BM a&b\\c&d\EM)$ and $\gamma'=(\BM a'&b'\\c'&d'\EM)$
\[\z f|_{\frac12}\gamma\gamma'=(\frac{ca'+db'}{cb'+dd'})(\frac cd)(\frac{c'}{d'})(f|_{\frac12}\gamma)|_{\frac12}\gamma'\]
However, it satisfies $f^2|_1\gamma=(f|_{\frac12}\gamma)^2$. Indeed, no weight $\frac12$ action with this condition exists: $f|_{\frac12}(\BM -1&0\\0&-1 \EM)$ should be $\pm if$ since $f^2|_1(\BM -1&0\\0&-1 \EM)=-f^2$, contradict to $f|_{\frac12}(\BM -1&0\\0&-1 \EM)|_{\frac12}(\BM -1&0\\0&-1 \EM)=f$.\\

For a congruence subgroup $\varGamma\subset{\rm SL}_2(\Z)$ and $\kappa\in\M+\frac12$, we define $\cM(\varGamma)_\kappa$ similarly in \S 2.1. This natural definition may be different from usual classical one.

In addition, for a map $\chi\in{\rm Map}(\varGamma,\C^\times)$, we define $\cM(\varGamma)_{(\kappa,\chi)}$ then
\[(\BM -1&0\\0&-1\EM)\in\varGamma,\;\chi(\BM -1&0\\0&-1\EM)\neq i^{-2\kappa} \Longrightarrow\cM(\varGamma)_{(\kappa,\chi)}=\{0\}\]
\[\sigma_4:\cM(\varGamma)_{(\kappa,\sqrt{\chi_4}\chi)}\simeq\cM(\varGamma)_{(\kappa,\sqrt{\chi_4}\,\overline{\chi})}.\]\\

If $h|c$ then
\[f^{\ang h}|_\kappa(\BM a&b\\c&d\EM)=(\z\frac hd)(f|_\kappa(\BM a&bh\\c/h&d\EM))^{\ang h}\]
since
\begin{align*}f^{\ang h}|_\kappa(\BM a&b\\c&d\EM)(\tau)&=((\z\frac cd)(c\tau+d)^{1/2})^{-2\kappa}f((\BM h&0\\0&1\EM)(\BM a&b\\c&d\EM)\tau)\\
&=(\z\frac hd)((\frac{c/h}d)(c\tau+d)^{1/2})^{-2\kappa}f((\BM a&bh\\c/h&d\EM)(\BM h&0\\0&1\EM)\tau)\\
&=(\z\frac hd)f|_\kappa(\BM a&bh\\c/h&d\EM)(h\tau).\end{align*}
Similarly if $h|b$ then
\[f^{\ang{\frac1h}}|_\kappa(\BM a&b\\c&d\EM)=(\z\frac hd)(f|_\kappa(\BM a&b/h\\ch&d\EM))^{\ang{\frac1h}}.\]

We see
\[\cM(N,1)_{(\kappa,\chi)}\,\!^{\ang h}=\z\cM(hN,1)_{(\kappa,\chi(\frac h\bullet))}\cap\C[[q^h]].\]

\newpage

\subsection{Theta functions of weight 1/2}

Let
\[\theta=\VS{n\in\Z}q^{n^2}=1+2\VS{n\in\N}q^{n^2}\]
then $\theta\in\cM(4,1)_{(\frac12,\overline{\sqrt{\rho_4}})}$.

The dimension formula states $\dim\cM(4,1)_{(1,\chi_4)}=1$ and we get $\iE_{\chi_4}=\theta^2$.

Similarly $\dim\cM(8,1)_{(1,\chi_4\chi_8)}=1$ and $\iE_{\chi_4\chi_8}=\theta\theta^{\ang2}$.

Note
\begin{align*}\iE_{\chi_4}&=1+4\VS{n\in4\M+1}(\chi_4*{\tt 1}_\N)(n)q^n\\
\iE_{\chi_4\chi_8}&=1+2\VS{n\in8\M+\{1,3\}}(\chi_4\chi_8*{\tt 1}_\N)(n)q^n\end{align*}

As a corollary we see the Jacobi's two-square theorem
\[\#\big\{(a,b)\in\Z\times\Z \,\big|\, a^2+b^2=n\big\}=4(\chi_4*{\tt 1}_\N)(n)\]
\[\#\big\{(a,b)\in\Z\times\Z \,\big|\, a^2+2b^2=n\big\}=2(\chi_4\chi_8*{\tt 1}_\N)(n)\]

It is convenient to write $(\frac k2^*,\chi)=(\frac k2,\xi_8^k\chi)$ and $\frac k2^*=(\frac k2^*,{\tt 1})$.

For example $\iE_{\chi_4}\in\cM(4,1)_{1^*}$ and $\theta\in\cM(2,1)_{(\frac12^*,\chi_8)}$.
\\

For $\chi\in{\rm Hom}((\Z/N\Z)^\times,\C^\times)$ such that $\chi(-1)=1$, put
\[\theta_{\chi}=\VS{n\in\N}\chi(n)q^{n^2}\]
then $\theta_{\chi}\in\cM(4N^2,1)_{(\frac12,\overline{\sqrt{\rho_4}}\chi)}$.

For example $\theta_{{\tt 1}_p}=\frac12(\theta-\theta^{\ang{p^2}})$ for a prime $p$.

In their paper, Serre and Stark prove that $\cM(\Gamma_1(N))_{\frac12}$ is spanned by
\[\theta^{\ang{h}} \text{ with } 4h|N,\6 \theta_{\chi}^{\ang h} \text{ with } 4c_\chi^2h|N\]
where $c_\chi$ is the conductor of $\chi$. For example
\[\{\theta,\theta^{\ang4},\theta^{\ang{16}},\theta^{\ang{64}}\}\]
is a basis of $\cM(256,1)_{(\frac12,\overline{\sqrt{\rho_4}})}$ and
\[\{\theta^{\ang2},\theta^{\ang8},\theta^{\ang{32}},\theta_{\chi_8}\}\]
is a basis of $\cM(256,1)_{(\frac12,\overline{\sqrt{\rho_4}}\chi_8)}$.

If $\chi$ is totally-even, that is, those $\chi$ whose prime-power components $\chi_p$ are all even, or $\chi$ is square of some character of conductor $\gcd(c_\chi,2)c_\chi$, then $\theta_{\chi}$ is (closely related to) a twist of $\theta$, and hence one would not expect it to be cuspidal. However, if $\chi$ is even but not totally even, then $\theta_{\chi}$ turns out to be a cusp form.

\newpage

\section{Eta family}

\subsection{Eta function}

Put the Dedekind eta function
\[\eta=q^{\frac1{24}}\VP{n\in\N}(1-q^n).\]

Note the Jacobi's triple product identity
\[\VP{n\in\N}(1-x^{2n})(1+x^{2n-1}y)(1+x^{2n-1}y^{-1})=\VS{n\in\Z}x^{n^2}y^n\]
for complex numbers $x,y$ such that $|x|<1$ and $y\neq0$.

Putting $x=q^{3/2}$ and $y=-q^{1/2}$ yields Euler's pentagonal numbers theorem
\[\VP{n\in\N}(1-q^n)=\VS{n\in\Z}(-1)^nq^{\frac{3n^2-n}2}\]
thus
\[\eta=\VS{n\in\Z}(-1)^nq^{\frac{(6n-1)^2}{24}}=\theta_{\chi_{12}}^{\ang{\frac1{24}}}\]\\
where $\chi_{12}=\chi_4\chi_3=(\frac3\bullet)$.\\

Note the logarithmic derivative
\[\frac d{dz}\log\eta(z)=\frac{\pi i}{12}+2\pi i\VS{n\in\N}\dfrac{nq^n}{1-q^n}=\dfrac{\pi i}{12}\iE_{1,2}(z).\]

$\eta$ satisfies the functional equations
\[\eta|_{\frac12}(\BM 1&1\\0&1 \EM)=1^{\frac1{24}}\eta,\6\eta|_{\frac12}(\BM 0&-1\\1&0 \EM)=1^{\frac78}\eta.\]
Since ${\rm SL}_2(\Z)=\ang{(\BM 1&1\\0&1 \EM),(\BM 0&-1\\1&0 \EM)}$ we get $\eta^{24}\in\cM(1)_{12}$ and
\[\z\frac1{12^3}(\iE_{1,4}^3-\iE_{1,6}^2)=\eta^{24}.\]

Note $\Gamma(1)$ has only one cusp $i\infty$ and $\cM(1)_k\cap\C[[q]]q=\cS(1)_k$, hence
\[\eta^{24}\in\cS(1)_{12}\]

\vspace{0.5cm}

More precisely, Rademacher showed for $(\BM a&b\\c&d \EM)\in\Gamma(1)$
\[\eta|_{\frac12}(\BM a&b\\c&d \EM)=\BC 1^{\frac{ac+bd-cd-acd^2+3d-3}{24}}\eta & \text{if }c\in2\N\\
\z(\frac dc)(\frac cd)1^{\frac{ac+bd+cd-bc^2d-3c}{24}}\eta & \text{if }c\in2\M+1 \EC\]

\begin{Prop}\label{cuspform}$\eta^3\in\cS(2)_{(\frac32^*,c_8b_8)}$ and $\eta\in\cS(6)_{(\frac12^*,c_{24}b_{24})}$.\end{Prop}
\begin{proof}Note $\xi_8(n)=1^{\frac{n-1}8}$ for odd $n$. For $\gamma=(\BM a&b\\c&d \EM)\in\Gamma_0(2)\cap\Gamma^0(2)$ we see
\[\frac{\eta^3|\gamma}{\eta^3}=1^{\frac{ac+bd-cd-acd^2+3d-3}{8}}=1^{\frac{bd-cd+3d-3}8}=\xi_8^3c_8b_8(\gamma).\]

For $\gamma=(\BM a&b\\c&d \EM)\in\Gamma_0(6)\cap\Gamma^0(6)$ we see
\[\frac{\eta|\gamma}{\eta}=1^{\frac{bd-cd+3d-3}{24}}=\xi_8c_{24}b_{24}(\gamma).\]\end{proof}

\newpage

\subsection{Eta quotients}

Functions of the form $\VP{d|N}\eta^{\ang dr_d}$ ($N\in\N$ and $r_d\in\Z$ for $d|N$) is called eta-quotients. Note $\eta$ is non-vanishing away from the cusps. The order of vanishing of $\VP{d|N}\eta^{\ang dr_d}$ at the cusp $\frac nm$ is $\frac N{24}\VS{d|N}\frac{(d,m)^2r_d}{(m,N/m)dm}$. It only depends on $m$, and not on $n$.\\[0.5cm]

Note $\eta^{\ang h}\big((\BM 0&-1\\1&0\EM)z\big)=\sqrt{\frac zh}e^{-\frac{\pi i}4}\eta^{\ang{\frac1h}}(z)$ and
\[\eta^{"p"}|_{\frac{p-1}2}(\BM 0&-1\\1&0\EM)=\frac{\eta^{\ang p}\big((\BM 0&-1\\1&0\EM)z\big)^p}{\sqrt z\eta\big((\BM 0&-1\\1&0\EM)z\big)}=\frac{e^{-\frac{(p-1)}4\pi i}}{\sqrt p^p}\eta^{/p/}.\]\\

We write $\flat={/2/}$ and $\sharp="2"$ then $\eta^{\sharp}|_{\frac12}(\BM 0&-1\\1&0\EM)=\frac{1^{\frac78}}2\eta^{\flat}$ and $\eta^{\flat}|_{\frac12}(\BM 0&-1\\1&0\EM)=2\cdot1^{\frac78}\eta^{\sharp}$.

\begin{Prop}\label{lev2}$\eta^{\flat}\in\cM(1,2)_{(\frac12^*,c_8)}$ and $\eta^{\sharp}\in\cM(2,1)_{(\frac12^*,b_8)}$.
\end{Prop}
\begin{proof}For $\gamma=(\BM a&b\\c&d \EM)\in\Gamma_0(2)\cap\Gamma^0(2)$, we see
\[\frac{\eta^{\flat}|\gamma}{\eta^{\flat}}=\frac{1^{\frac{2ac+\frac b2d-2cd-2acd^2+3d-3}{12}}}{1^{\frac{ac+bd-cd-acd^2+3d-3}{24}}}=1^{\frac{ac-cd-acd^2+d-1}8}=c_8\xi_8(\gamma).\]

Since $\eta^{\sharp}\in\C[[q]]q^{\frac18}$ we see $\eta^{\sharp}|(\BM 1&1\\0&1\EM)=1^{\frac18}\eta^{\sharp}$ and
\[\eta^{\flat}|(\BM 1&0\\-1&1\EM)=2\cdot1^{\frac18}\eta^{\sharp}|(\BM 0&-1\\1&0\EM)(\BM 1&0\\-1&1\EM)=2\cdot1^{\frac18}\eta^{\sharp}|(\BM 1&1\\0&1\EM)(\BM 0&-1\\1&0\EM)=1^{\frac18}\eta^{\flat}.\]
Then, $(\BM a&b\\c&d\EM)=(\BM a+b&b\\c+d&d\EM)(\BM 1&0\\-1&1\EM)$ and $(\frac{c+d}d)=(\frac cd)$ leads to the former assertion.

For $(\BM a&b\\c&d\EM)\in\Gamma_0(2)$
\[\eta^{\sharp}|(\BM a&b\\c&d\EM)=\eta^{\sharp}|(\BM 0&-1\\1&0\EM)(\BM d&-c\\-b&a\EM)(\BM 0&1\\-1&0\EM)=1^{-\frac{bd}8}\xi_8(d)\eta^{\sharp}.\]
\end{proof}

\begin{Lem}\label{etanf}$\eta^{\flat}=\VS{n\in\Z}(-1)^nq^{\frac{n^2}2}$ and $\eta^{\sharp}=\VS{n\in\M}q^{\frac{(2n+1)^2}8}=\theta_{{\tt 1}_2}^{\ang{\frac18}}$.\end{Lem}
\begin{proof}Note $\eta^{\flat}=\VP{n\in\N}\dfrac{(1-q^{\frac n2})^2}{1-q^n}=\VP{n\in\N}(1-q^n)(1-q^{n-\frac12})^2$. On Jacobi's triple product identity, putting $x=q^{\frac12}$ and $y=-1$ yields the former identity.

Note $\eta^{\sharp}=q^{\frac18}\VP{n\in\N}(1-q^n)(1+q^n)^2$. On Jacobi's triple product identity, putting $x=y=q^{\frac12}$ yields
\[\VP{n\in\N}(1-q^n)(1+q^n)(1+q^{n-1})=\VS{n\in\Z}q^{\frac{n^2+n}2}.\]
\end{proof}

\newpage

We write $\bot={/3/}$ and $\top="3"$.

\begin{Lem}\label{lev3}$\eta^{\bot}\in\cM(1,3)_{(1,\chi_3c_3)}$ and $\eta^{\top}\in\cM(3,1)_{(1,\chi_3b_3)}$.
\end{Lem}
\begin{proof}For $(\BM a&b\\c&d \EM)\in\Gamma^0(3)$ we see
\[\frac{(\frac3d)1^{\frac{3}{24}(3ac+\frac b3d-3cd-3acd^2+3d-3)}}{1^{\frac1{24}(ac+bd-cd-acd^2+3d-3)}}=\z(\frac{-1}d)(\frac{-3}d)1^{\frac{(a-d-ad^2)c}3+\frac{d-1}4}=1^{-\frac{cd}3}\chi_3(d),\]
\[\frac{(\frac3d)(\frac d{3c})(\frac{3c}d)1^{\frac3{24}(3ac+\frac b3d+3cd-3bc^2d-9c)}}{(\frac dc)(\frac cd)1^{\frac3{24}(ac+bd+cd-bc^2d-3c)}}=\z(\frac d3)1^{\frac{(a+d-bcd)c}3}=1^{-\frac{cd}3}\chi_3(d),\]
and $\eta^{\top}|_1(\BM 0&-1\\1&0\EM)=\frac1{\sqrt3^3i}\eta^{\bot}$.
\end{proof}

The above forms have representaions as theta function
\[\eta^{\bot\ang3}=\VS{m,n\in\Z}1^{\frac{m-n}3}q^{m^2+mn+n^2},\]
\[3\eta^{\top}=\VS{m,n\in\Z+\frac13}q^{m^2+mn+n^2}.\]
\\

\begin{Lem}$\eta^{/5/}\in\cM(1,5)_{(2,\chi_5^2)}$ and $\eta^{"5"}\in\cM(5,1)_{(2,\chi_5^2)}$.\end{Lem}
\begin{proof}\[\frac{(\frac5d)1^{\frac5{24}(5ac+\frac b5d-5cd-5acd^2+3d-3)}}{1^{\frac1{24}(ac+bd-cd-acd^2+3d-3)}}=\z(\frac5d)e^{\pi i(d-1)},\]
\[\frac{(\frac5d)(\frac d{5c})(\frac{5c}d)1^{\frac5{24}(5ac+\frac b5d+5cd-5bc^2d-15c)}}{(\frac dc)(\frac cd)1^{\frac1{24}(ac+bd+cd-bc^2d-3c)}}=\z(\frac d5),\]
and $\eta^{"5"}|_2(\BM 0&-1\\1&0\EM)=-\frac1{\sqrt5^5}\eta^{/5/}$.\end{proof}

\vspace{0.5cm}

\begin{Lem}$\eta^{/7/}\in\cM(1,7)_{(3,\chi_7^3)}$ and $\eta^{"7"}\in\cM(7,1)_{(3,\chi_7^3)}$.\end{Lem}
\begin{proof}
\[\frac{(\frac7d)1^{\frac7{24}(7ac+\frac b7d-7cd-7acd^2+3d-3)}}{1^{\frac1{24}(ac+bd-cd-acd^2+3d-3)}}=\z(\frac{-1}d)(\frac{-7}d)e^{\frac{\pi i}2(d-1)},\]
\[\frac{(\frac7d)(\frac d{7c})(\frac{7c}d)e^{\frac7{24}(7ac+\frac b7d+7cd-7bc^2d-21c)}}{(\frac dc)(\frac cd)1^{\frac1{24}(ac+bd+cd-bc^2d-3c)}}=\z(\frac d7),\]
and $\eta^{"7"}|_3(\BM 0&-1\\1&0\EM)=\frac i{\sqrt7^7}\eta^{/7/}$.\end{proof}

\newpage

\subsection{$\natural$ operator}

Let $\natural=\flat\sharp$ i.e. $f^{\natural}=\dfrac{f^5}{f^{\ang{\frac12}2}f^{\ang22}}$, then $f^{\natural}f^{\flat}=f^{\flat\ang22}$, $f^{\natural}f^{\sharp}=f^{\sharp\ang{\frac12}2}$ and $f^{\natural}f^{\flat}f^{\sharp}=f^3$. 
In particular $\eta^{\natural}=\dfrac{\eta^3}{\eta^{\flat}\eta^{\sharp}}\in\cM(2)_{\frac12^*}$.

It is important that $\eta^{\flat}=\VP{n\in\N}\dfrac{1-q^{\frac n2}}{1+q^{\frac n2}}$ and
\[\eta^{\natural}=\VP{n\in\N}\dfrac{1+q^{\frac n2}}{1-q^{\frac n2}}\Big(\dfrac{1-q^n}{1+q^n}\Big)^2=\VP{n\in\N}\dfrac{1-(-q)^{\frac n2}}{1+(-q)^{\frac n2}}=\eta^{\flat}|(\BM 1&1\\0&1\EM).\]

Acting $(\BM1&1\\0&1\\\EM)$ on the former identity of Proposition \ref{etanf} yields $\eta^{\natural}=\theta^{\ang{\frac12}}$.

Remark
\[\eta^{\natural}|_{\frac12}(\BM 0&-1\\1&0\EM)=\frac{\eta\big((\BM 0&-1\\1&0\EM)\tau\big)^5}{\sqrt\tau\eta^{\ang{\frac12}}\big((\BM 0&-1\\1&0\EM)\tau\big)^2\eta^{\ang2}\big((\BM 0&-1\\1&0\EM)\tau\big)^2}=1^{\frac78}\eta^{\natural}.\]\\

\begin{Lem}\label{rel21}$\frac1{16}(\eta^{\natural4}-\eta^{\flat4})=\eta^{\sharp4}$.\end{Lem}
\begin{proof}
Put $f=\eta^{\natural4}-\eta^{\flat4}-(2\eta^{\sharp})^4$ then $f\in\cM(2)_4\cap\C[[q]]q$.

We see $f|(\BM 1&1\\0&1\EM)=f|(\BM 0&-1\\1&0\EM)=-f$ thus $f^2\in\cM(1)_8\cap\C[[q]]q^2=\{0\}$.
\end{proof}

\begin{Lem}\label{rel22}$\frac12(\eta^{\natural2}+\eta^{\flat2})=\eta^{\natural\ang22}$ and $\frac18(\eta^{\natural2}-\eta^{\flat2})=\eta^{\sharp\ang22}$.\end{Lem}
\begin{proof}The product of
\[\z \frac12(\eta^{\natural}+\eta^{\flat})=\eta^{\natural\ang4}\6\frac14(\eta^{\natural}-\eta^{\flat})=\eta^{\sharp\ang4}\]
shows the latter identity. 

Dividing Lemma \ref{rel21} by this one leads to the first one.
\end{proof}

\vspace{0.5cm}

We have $\dim\cM(2)_2=2$ and $\iE_{2,2}=\frac12(\eta^{\natural4}+\eta^{\flat4})$.

Acting $(\BM 0&-1\\1&0\EM)$ yields $\iE_{2,2}^{\ang{\frac12}}=\eta^{\natural4}+16\eta^{\sharp4}$.

Since $\iE_{2,2}=2\iE_{1,2}^{\ang2}-\iE_{1,2}$ and $\iE_{4,2}=\frac13(\iE_{1,2}^{\ang4}-\iE_{1,2})$ we see
\[\z\iE_{4,2}^{\ang{\frac12}}=\frac23\iE_{2,2}+\frac13\iE_{2,2}^{\ang{\frac12}}=\eta^{\natural4}.\]

Remark $\iE_{\chi_4}^{\ang{\frac12}}=\eta^{\natural2}$ and $\iE_{\chi_4\chi_8}^{\ang{\frac14}}=\eta^{\natural\ang{\frac12}}\eta^{\natural}$.\\

\begin{Prop}\label{theta8}
\begin{align*}\theta_{\chi_8}^{\ang{\frac1{16}}}&=\sqrt{\eta^{\flat}\eta^{\sharp}}\\
\theta_{\chi_{12}\chi_8}^{\ang{\frac1{48}}}&=\sqrt{\eta^{\natural}\eta}\end{align*}
\end{Prop}
\begin{proof}
Note $\chi_8(n)=1^{\frac{n^2-1}{16}}$ for $n\in2\M+1$.

Acting $(\BM1&1\\0&1\\\EM)$ on $\VS{n\in2\M+1}q^{\frac{n^2}{16}}=\sqrt{\eta^{\natural}\eta^{\sharp}}$ yields the former identity.

Acting $(\BM1&3\\0&1\\\EM)$ on $\VS{n\in\N}\chi_{12}(n)q^{\frac{n^2}{48}}=\sqrt{\eta^{\flat}\eta}$ yields the latter one.
\end{proof}

\newpage

\subsection{$\nwarrow$ and $\swarrow$ operators}

First note $f^{\bot}f^{\top}f^{\bot\top}=f^8$. 

Let $\eta^{\nwarrow}=\eta^{\bot}|(\BM 1&1\\0&1\EM)$ and $\eta^{\swarrow}=\eta^{\bot}|(\BM 1&-1\\0&1\EM)$ then $\eta^{\nwarrow},\eta^{\swarrow}\in\cM(\Gamma(3))_1=\cM(3)_{(1,\chi_3)}$.

We see
\begin{align*}\eta^{\nwarrow}\eta^{\swarrow}&=\VP{n\in\frac13\N\setminus\N}(1-1^{\frac13}q^n)^3(1-1^{\frac23}q^n)^3\VP{n\in\N}(1-q^n)^4\\
&=\frac{\VP{n\in\N\setminus3\N}(1-q^n)^7\VP{n\in3\N}(1-q^n)^4}{\VP{n\in\frac13\N\setminus\N}(1-q^n)^3}\\
&=\eta^{\bot\top}\end{align*}

We also see $\eta^{\nwarrow}|(\BM 0&-1\\1&0\EM)=1^{\frac23}\eta^{\swarrow}$ since
\[(\BM 1&1\\0&1\EM)(\BM 0&-1\\1&0\EM)=(\BM 0&1\\-1&0\EM)(\BM 1&-1\\0&1\EM)(\BM 0&-1\\1&0\EM)(\BM 1&-1\\0&1\EM).\]

\begin{Prop}
\[\eta^{\bot}+1^{\frac13}\eta^{\nwarrow}+1^{\frac23}\eta^{\swarrow}=0\]
\[\z\sqrt3^3i\eta^{\top}+\eta^{\nwarrow}-\eta^{\swarrow}=0\]
\end{Prop}
\begin{proof}Put the LHS $F,G$ respectively. Then $F|(\BM 1&1\\0&1\EM)=1^{\frac23}F$ and $F|(\BM 0&-1\\1&0\EM)=-G$.

We see $G|(\BM 1&1\\0&1\EM)=\sqrt3^3i1^{\frac13}\eta^{\top}-\eta^{\bot}+\eta^{\swarrow}$ and $G|(\BM 1&2\\0&1\EM)=\sqrt3^3i1^{\frac23}\eta^{\top}+\eta^{\bot}-\eta^{\nwarrow}$ thus
$G|(\BM 1&1\\0&1\EM)|(\BM 0&-1\\1&0\EM)=1^{\frac13}G|(\BM 1&2\\0&1\EM)$.

Therefore $\big(F\cdot G\cdot G|(\BM 1&1\\0&1\EM)\cdot G|(\BM 1&2\\0&1\EM)\big)^3\in\cM(1)_{12}\cap\C[[q]]q^3=\{0\}$.
\end{proof}

\vspace{0.5cm}

The dimension formula states $\dim\cM(3,1)_1=1$ and we see
\[\z\iE_{\chi_3}=\frac13(\eta^{\bot}+\eta^{\nwarrow}+\eta^{\swarrow})=\eta^{\bot}+3\eta^{\top}.\]

We compute $\cM(1,3)_1\ni\iE_{\chi_3}|(\BM 0&-1\\1&0\EM)=\frac{-i}{\sqrt3}(\eta^{\bot}+9\eta^{\top})$ thus $\iE_{\chi_3}^{\ang{\frac13}}=\eta^{\bot}+9\eta^{\top}$.

Next result is well known (cf. \cite[Theorem 4.11.3]{DS})
\begin{align*}\iE_{\chi_3}&=\VS{m,n\in\Z}\frac13(1+1^{\frac{m^2+mn+n^2}3}+1^{-\frac{m^2+mn+n^2}3})1^{\frac{m-n}3}q^{\frac{m^2+mn+n^2}3}\\
&=\VS{m,n\in\Z}q^{m^2+mn+n^2}\end{align*}
\\

Note $\iE_{3,2}=\frac12(3\iE_{1,2}^{\ang3}-\iE_{1,2})$ and
\begin{align*}\iE_{9,2}^{\ang{\frac13}}&=\z\frac18(9\iE_{1,2}^{\ang3}-\iE_{1,2}^{\ang{\frac13}})=\frac34\iE_{3,2}-\frac14\iE_{3,2}^{\ang{\frac13}}\\
&=\z\frac34\big(\frac{1-1^{\frac13}}3\eta^{\nwarrow}+\frac{1-1^{\frac23}}3\eta^{\swarrow}\big)^2-\frac14\big(\frac{1^{\frac13}-1}{\sqrt3i}\eta^{\nwarrow}+\frac{-1^{\frac23}+1}{\sqrt3i}\eta^{\swarrow}\big)^2\\
&=\eta^{\nwarrow}\eta^{\swarrow}.\end{align*}

This result was discovered using symbolic computation by Borwein and Garvan, and it was used to produce a ninth order iteration that converges to $1/\pi$.

\newpage

\section{Rational weight theories}

\subsection{Definitions}

Modular forms of rational weight might not be very popular. See \cite{I1} or \cite{I2}. 
For any discrete subgroup $\varGamma\subset{\rm SL}_2(\Z)$, a nowhere vanishing holomorphic function $J(\gamma,\tau):\varGamma\times\cH\to\C^\times$ is said to be an automorphy factor of $\varGamma$ if
\[J(\gamma_1\gamma_2,\tau) = J(\gamma_1,\gamma_2\tau)J(\gamma_2,\tau)\]
for any $\gamma_1,\gamma_2\in\varGamma$, $\tau\in\cH$. For any automorphy factor $J$, we define an action of $\varGamma$ on holomorphic functions $f$ on $\cH$ by
\[(f|_J\gamma(\tau)=\frac{f(\gamma\tau)}{J(\gamma,\tau)}\]
If $f|_J\gamma=f$ for all $\gamma\in\varGamma$, and if $f$ is holomorphic also at all the cusps of $\varGamma$, we say that $f$ is a modular form of weight $J$.
\\

For rational weight cases, we chose the standard automorphy factor here, however other ones will be made use of later.

Note $\eta$ does not vanish on the upper half plane, so we may define any real power of $\eta$ as an entire function, fixing a natural branch. For $\alpha\in\Gamma(N)$ define
\[f|_{\frac{12}N}\alpha(\tau)=\frac{\eta^{\frac{24}N}(\tau)}{\eta^{\frac{24}N}(\alpha\tau)}f(\alpha\tau)\]
and $(ff')|_{\kappa+\kappa'}\alpha=f|_\kappa\alpha\cdot f'|_{\kappa'}\alpha$. Proposition \ref{cuspform} shows $\eta^{\frac{24}N}\in\cS(\Gamma(N))_{\frac{12}N}$ for each divisor $N$ of 24 and now this result for genral $N$.

A dimension formulas (cf \cite[Lemma 1.7]{I1}) is
\[\z \dim\cM(\Gamma(N))_{\frac{k(N-3)}{2N}}=\Big(\frac{k(N-3)}{48}-\frac{N-6}{24}\Big)N^2\VP{p|N}(1-\frac1{p^2})\]
for odd $N>3$ and $k>4\frac{N-6}{N-3}$.\\

For odd $N,r$ with $N\geq3$ and $1\leq r\leq N-2$, we write
\[f_{N,r}=\eta^{-\frac 3N}q^{\frac{r^2}{8N}}\VS{p\in\Z}(-1)^pq^{\frac{Np^2+rp}2}\]
then $f_{N,r}\in\cM(\Gamma(N))_{\frac{N-3}{2N}}$.

\newpage

\subsection{Integer/4-power weight}

The weight $\frac34$ operator on $\Gamma(16)$ defined in previous section is extended to $\Gamma_0(2)\cap\Gamma^0(2)$ formally by
\[\sqrt{\eta^3}\in\cS(2)_{(\frac34^*,c_{16}b_{16})}.\]
where $(\frac k{2^n}^*,\chi)=(\frac k{2^n},\xi_8^{k/2^{n-1}}\chi)$ formally, then $\sqrt{\eta}=\frac{\eta^2}{\sqrt{\eta^3}}\in\cS(6)_{(\frac14^*,c_{48}b_{48})}$

\begin{Lem}
\[\sqrt{\eta^{\flat}}\in\cM(2)_{(\frac14^*,\chi_8c_{16})} \6 \sqrt{\eta^{\sharp}}\in\cM(2)_{(\frac14^*,\chi_8b_{16})}\]
\[\sqrt{\eta^{\natural}}\in\cM(2)_{\frac14^*}\]
\end{Lem}
\begin{proof}The first assertion follows from $\sqrt{\eta^{\flat}}=\dfrac{\sqrt{\eta^3}}{\eta^{\sharp\ang{\frac12}}}$ and
\[\sqrt{\eta^3}\in\cS(2)_{(\frac34^*,c_{16}b_{16})},\6\eta^{\sharp\ang{\frac12}}\in\cM(1,2)_{(\frac12^*,\chi_8b_{16})}.\]
\end{proof}

The above Lemma and Proposition \ref{theta8} lead to $\theta_{\chi_8}^{\ang{\frac1{16}}}\in\cM(2)_{(\frac12^*,c_{16}b_{16})}$ and $\theta_{\chi_{12}\chi_8}^{\ang{\frac1{48}}}\in\cM(6)_{(\frac12^*,c_{48}b_{48})}$.

\vspace{0.5cm}

The weight $\frac38$ operator is extended by
\[\sqrt[4]{\eta^3}\in\cS(2)_{(\frac38^*,c_{32}b_{32})}\]
then $\sqrt[8]{\eta}=\frac{\eta}{\sqrt[4]{\eta^3}}\in\cS(6)_{(\frac18^*,c_{96}b_{96})}$.

Moreover the weight $\frac3{16}$ operator is extended by
\[\sqrt[8]{\eta^3}\in\cS(2)_{(\frac3{16}^*,c_{64}b_{64})}\]
then $\sqrt[8]{\eta}=\frac{\sqrt{\eta}}{\sqrt[8]{\eta^3}}\in\cS(6)_{(\frac1{16}^*,c_{192}b_{192})}$.

\newpage

\subsection{Integer/3-power weight}

The weight $\frac43$ operator on $\Gamma(9)$ is also extended to $\Gamma_0(3)\cap\Gamma^0(3)$ by
\[\sqrt[3]{\eta^8}\in\cS(3)_{(\frac43,c_9b_9)}\]
then
\[\sqrt[3]{\eta}=\frac{\eta^3}{\sqrt[3]{\eta^8}}\in\cS(6)_{(\frac16^*,c_{72}b_{72})}\]
\[\sqrt[6]{\eta}=\frac{\sqrt{\eta}}{\sqrt[3]{\eta}}\in\cS(6)_{(\frac1{12}^*,c_{144}b_{144})}\]
where $(\frac16^*,\chi)=(\frac16,\xi_8^3\chi)$ and $(\frac1{12}^*,\chi)=(\frac1{12},\sqrt{\xi_8}^3\chi)$.\\

\begin{Lem}
\[\sqrt[3]{\eta^{\bot}}\in\cM(3)_{(\frac13,\chi_3c_9)} \6 \sqrt[3]{\eta^{\top}}\in\cM(3)_{(\frac13,\chi_3b_9)}\]
\[\sqrt[3]{\eta^{\nwarrow}}\in\cM(3)_{(\frac13,\chi_9)} \6 \sqrt[3]{\eta^{\swarrow}}\in\cM(3)_{(\frac13,\overline{\chi_9})}\]
\end{Lem}
\begin{proof}
The first assertion follows from $\sqrt[3]{\eta^{\bot}}=\dfrac{\sqrt[3]{\eta^8}}{\eta^{\top\ang{\frac13}}}$. Note
\[\iE_{\chi_3}=1+6\VS{n\in3\M+1}(\chi_3*{\tt 1}_\N)(n)q^n,\]
\[\iE_{\chi_9}=1+(1-1^{\frac13})\VS{n\in\N}(\chi_9*{\tt 1}_\N)(n)q^n.\]

We see
\[\VS{n\in3\M+1}(\chi_3*{\tt 1}_\N)(n)q^{\frac n9}=\frac16(\iE_{\chi_3}^{\ang{\frac19}}-\iE_{\chi_3}^{\ang{\frac13}})=\eta^{\top\ang{\frac13}}=\sqrt[3]{\eta^{\nwarrow}\eta^{\swarrow}\eta^{\top}}\]
and $\chi_9^2(n)=1^{-\frac{n-1}9}$ for $n\in3\M+1$, hence
\[\z\iG_{\chi_9^2,\overline{\chi_9}}^{\ang{\frac19}}=\sqrt[3]{\eta^{\bot}\eta^{\nwarrow}\eta^{\top}}\]
in particular $\sqrt[3]{\eta^{\nwarrow}}\in\cM(9)_{(\frac13,\chi_9)}$.

The dimension formula states $\dim\cM(\Gamma(9))_1=18$ and $\dim\cM(9)_{(1,\chi_3)}=6$. Hence we get $\dim\cM(9)_{(1,\chi_9)}=6$ and
\[\iE_{\chi_9}^{\ang{\frac13}}=\sqrt[3]{\eta^{\nwarrow2}\eta^{\swarrow}}\]
\end{proof}

Since $\dim\cM(9)_{\frac13}=4$ we easily see
\[\z f_{9,1}=\frac13(\sqrt[3]{\eta^{\bot}}+\sqrt[3]{\eta^{\nwarrow}}+\sqrt[3]{\eta^{\swarrow}})\]
\[f_{9,3}=\sqrt[3]{\eta^{\top}}\]
\[\z f_{9,5}=-\frac13(\sqrt[3]{\eta^{\bot}}+1^{\frac23}\sqrt[3]{\eta^{\nwarrow}}+1^{\frac13}\sqrt[3]{\eta^{\swarrow}})\]
\[\z f_{9,7}=-\frac13(\sqrt[3]{\eta^{\bot}}+1^{\frac13}\sqrt[3]{\eta^{\nwarrow}}+1^{\frac23}\sqrt[3]{\eta^{\swarrow}})\]\

The weight $\frac49$ operator on $\Gamma(27)$ is also extended to $\Gamma_0(3)\cap\Gamma^0(3)$ by
\[\sqrt[9]{\eta^8}\in\cS(3)_{(\frac49,c_{27}b_{27})}\]

\newpage

\subsection{Integer/5 weight}

Note $\eta^{\frac{24}5}=\sqrt[5]{\eta^{/5/}}\eta^{"5"\ang{\frac15}}$ thus $\sqrt[5]{\eta^{/5/}}\in\cM(\Gamma(5))_{\frac25}$. We extend the weight $\frac25$ operator on $\Gamma(5)$ to $\Gamma_0(5)\cap\Gamma^0(5)$ by
\[\sqrt[5]{\eta^{/5/}}\in\cM(5)_{(\frac25,\chi_5^2)}.\]

We see $\sqrt[5]{\eta^{/5/4}\eta^{"5"}}=\dfrac{\eta^{\ang{\frac15}4}\eta^{\ang5}}{\eta}\in\cM(5)_{(2,,\chi_5^2)}$ since
\[\frac{(\frac5d)e^{\frac{4\pi i}{12}(5ac+\frac b5d-5cd-5acd^2+3d-3)}e^{\frac{\pi i}{12}(\frac{ac}5+5bd-\frac{cd}5-\frac{acd^2}5+3d-3)}}{e^{\frac{\pi i}{12}(ac+bd-cd-acd^2+3d-3)}}=\z(\frac5d)e^{\pi i(d-1)},\]
\[\frac{(\frac5d)(\frac d{c/5})(\frac{c/5}d)e^{\frac{4\pi i}{12}(5ac+\frac b5d+5cd-5bc^2d-15c)e^{\frac{\pi i}{12}(ac+bd-cd-acd^2+3d-3)}}}{(\frac dc)(\frac cd)e^{\frac{\pi i}{12}(ac+bd+cd-bc^2d-3c)}}=\z(\frac d5).\]
In particular
\[\sqrt[5]{\eta^{"5"}}\in\cM(5)_{(\frac25,\chi_5^2)}.\]\\

Now, put
\[\eta_{\chi_5}=\sqrt{\sqrt[5]{\eta^{/5/}}+(1-2i)\sqrt[5]{\eta^{"5"}}}\]
and $\eta_{\overline{\chi_5}}=\eta_{\chi_5}^{\sigma_4}$.

\begin{Lem}$\eta_{\chi_5}\in\cM(5)_{(\frac15,\chi_5)}$.\end{Lem}
\begin{proof}First we see $\eta_{\chi_5}^2\in\cM(5)_{(\frac25,\chi_5^2)}$. Note
\[\iE_{\chi_5}=1+(3-i)\VS{n\in\N}\VS{d|n}\chi_5(d)q^n.\]

The dimension formula states $\dim\cM(5)_{(2,\chi_5^2)}=6$ and
\[\eta_{\chi_5}^5=\z\frac{-3-i}4\iE_{\chi_5}^{\ang{\frac15}}+\frac{7+i}4\iE_{\chi_5}+(\frac52-5i)\iG_{\chi_5^2|\overline{\chi_5}}^{\ang{\frac15}}\]
\end{proof}

We also see $\eta_{\chi_5}\cdot\eta_{\chi_5}|(\BM 1&1\\0&1\EM)\cdot\eta_{\chi_5}|(\BM 1&2\\0&1\EM)\cdot\eta_{\chi_5}|(\BM 1&-2\\0&1\EM)\cdot\eta_{\chi_5}|(\BM 1&-1\\0&1\EM)=\iE_{\chi_5}$.

At last, remark
\[f_{5,1}=\z\frac12\big(\eta_{\chi_5}+\eta_{\overline{\chi_5}}\big)=\VP{n\in\N}(1-q^n)^{\frac25}\VP{n\in\M}\dfrac1{(1-q^{5n+1})(1-q^{5n+4})},\]
\[f_{5,3}=\z\frac i2\big(\eta_{\chi_5}-\eta_{\overline{\chi_5}}\big)=q^{\frac15}\VP{n\in\N}(1-q^n)^{\frac25}\VP{n\in\M}\dfrac1{(1-q^{5n+2})(1-q^{5n+3})},\]
and $f_{5,1}f_{5,3}=\sqrt[5]{\eta^{"5"}}$.  These are essentially the famous Rogers-Ramanujan functions.

\newpage

\section{Graded rings}

\subsection{First results}

Let $\varGamma$ be a congruence subgroup.

For a set $X$ (usually a semigroup), let
\[\cM(\varGamma)_X=\VT{\chi\in X}\cM(\varGamma)_x\]
\[\cS(\varGamma)_X=\VT{\chi\in X}\cS(\varGamma)_x\]

Generally, for a commutative semigroup $S$, a subset $X\subset S$ and a $S$-graded ring $R=\VT{s\in S}R_s$ let $R|_X=\VT{x\in X}R_x$. For example $\cM(\varGamma)_{2\M}=\cM(\varGamma)_\M\big|_{2\M}$.

We will study larger ring $\cM(\varGamma)_{\M\times{\rm Hom}(\varGamma,\C^\times)}$ or $\cM(\varGamma)_{\frac12\M\times{\rm Map}(\varGamma,\C^\times)}$. In many cases, these rings are finitely generated and we give explicit structures.\\

When $f$ is a weight $n$ form, the Hibert function
\[\VS{k\in\M}\dim\C[f]_k\cdot t^k=\VS{k\in\M}t^{nk}=\dfrac1{1-t^n}.\]

When $f,g$ are weight $m,n$ forms respectively, the Hibert function
\[\VS{k\in\M}\dim\C[f,g]_k\cdot t^k=\VS{k\in\M}t^{mk}\VS{k\in\M}t^{nk}=\dfrac1{(1-t^m)(1-t^n)}\]

Note in particular
\[\dfrac1{(1-t^m)(1-t^n)}=\VS{k\in\M}\big([\frac kn]+1\big)t^k.\]\\

Next result is well-known (cf. \cite[Theorem 3.5.2]{DS} or \cite[Theorem 2.17]{St}).

\begin{Thm}The ring of modular forms $\cM(\Gamma(1))_\M$ is a polynomial ring in two variables
\[\cM(\Gamma(1))_\M=\C\big[\iE_{1,4},\iE_{1,6}\big].\]
\end{Thm}
\begin{proof}Put $R=\C\big[\iE_{1,4},\iE_{1,4}^3-\iE_{1,6}^2\big]$ then the natural map $R\to\cM(1)_{4\M}$ is injective. The natural map ${\rm RHS}=R\oplus R\iE_{1,6}\to{\rm LHS}$ is injective as well. Compute the Hilbert fuction
\begin{align*}\VS{k\in\M}\dim{\rm RHS}_{2k}\cdot t^k
&=\frac1{(1-t^2)(1-t^3)}\\&=\frac1{t^2}\Big(\frac1{(1-t)(1-t^2)}-\frac1{(1-t)(1-t^3)}\Big)\\
&=\VS{k\in\M}\big([\frac k2]-[\frac k3]\big)t^{k-2}\\
&=\VS{k\in\M}\dim\cM(1)_{2k}\cdot t^k\end{align*}
\end{proof}

\newpage

\begin{Thm}\label{23}
\[\cM(\Gamma(2))_\M=\C\big[\iE_{2,2}^{\ang{\frac12}},\iE_{2,2}\big]=\C\big[\eta^{\flat4},\eta^{\sharp4}\big]\]
\[\cM(\Gamma(3))_\M=\C\big[\iE_{\chi_3}^{\ang{\frac13}},\iE_{\chi_3}\big]=\C\big[\eta^{\bot},\eta^{\top}\big]\]
\end{Thm}
\begin{proof}For the former assertion, since
\[\C\big[\iE_{2,2}^{\ang{\frac12}},\iE_{2,2}\big]=\C\big[\iE_{2,2}^{\ang{\frac12}},\iE_{2,2}-\iE_{2,2}^{\ang{\frac12}}\big]\]
the natural map ${\rm RHS}\to{\rm LHS}$ is injective. 

The dimension formula states $\dim\cM(\Gamma(2))_{2k}=k+1$ for $k\in\M$.

Also $\dim\cM(\Gamma(3))_k=k+1$ for $k\in\M$.
\end{proof}

\vspace{0.5cm}

\begin{Thm}\label{cuspideal}For $N=1,2,3,4,6,12$ the ideal of cusp forms $\cS(\Gamma(N))_\M$ are principal ideal
\[\cS(\Gamma(N))_\M=\cM(\Gamma(N))_\M\eta^{\frac{24}N}.\]
\end{Thm}
\begin{proof}The dimension formula states for $k\in\M$
\[\dim\cS(\Gamma(1))_k=(\dim\cM(\Gamma(1))_k-1)^+=\dim\cM(\Gamma(1))_{k-12}\]
\[\dim\cS(\Gamma(2))_k=(\dim\cM(\Gamma(2))_k-3)^+=\dim\cM(\Gamma(2))_{k-6}\]
\[\dim\cS(\Gamma(3))_k=(\dim\cM(\Gamma(3))_k-4)^+=\dim\cM(\Gamma(3))_{k-4}\]
\[\dim\cS(\Gamma(4))_k=(\dim\cM(\Gamma(4))_k-6)^+=\dim\cM(\Gamma(4))_{k-3}\]
\[\dim\cS(\Gamma(6))_k=(\dim\cM(\Gamma(6))_k-12)^+=\dim\cM(\Gamma(6))_{k-2}\]
\[\dim\cS(\Gamma(12))_k=(\dim\cM(\Gamma(12))_k-48)^+=\dim\cM(\Gamma(12))_{k-1}\]
where $a^+=\BC a &\text{if }0<a\\0 &\text{oth}\EC$

\end{proof}

\begin{Thm}\label{4}
\[\cM(\Gamma(4))_{\frac12\M}=\C\big[\theta^{\ang{\frac14}},\theta\big]=\C\big[\eta^{\flat\ang{\frac12}},\eta^{\sharp\ang2}\big]\]
\[\cS(\Gamma(4))_{\frac12\M}=\cM(\Gamma(4))_{\frac12\M}\eta^6\]
\end{Thm}
\begin{proof}The dimension formula states $\dim\cS(\Gamma(4))_{\frac k2}=(k-5)^+$ only for even $k$. However the formula after Lemma \ref{sumPM} shows
\[\dim\cS(\Gamma(4))_{\frac k2}\leq\dim\cS(\Gamma(4))_{\frac{k+1}2}-1=(k-5)^+\]
for odd $k$. Two $\subset$ in $\C\big[\eta^{\flat\ang{\frac12}},\eta^{\sharp\ang2}\big]\eta^6\subset\cM(\Gamma(4))_{\frac12\M}\eta^6\subset\cS(\Gamma(4))_{\frac12\M}$ are $=$.\end{proof}

\begin{Thm}\label{5}
\[\cM(\Gamma(5))_{\frac15\M}=\C\big[f_{5,1},f_{5,3}\big]=\C\big[\eta_{\chi_5},\eta_{\overline{\chi_5}}\big]\]
\[\cS(\Gamma(5))_{\frac15\M}=\cM(\Gamma(5))_{\frac15\M}\eta^{\frac{24}5}\]
\end{Thm}
\begin{proof}The dimension formuls states $\dim\cS(\Gamma(5))_{\frac k5}=(k-11)^+$ for $k\in5\M$ and
\[\dim\cS(\Gamma(5))_{\frac k5}\leq\dim\cS(\Gamma(5))_{\frac{k+1}5}-1=(k-11)^+\]
for $k\in5\M+4$ and also for $k\in5\M+3$ ...
\end{proof}

\newpage

\subsection{Decompositions}

Generally, the spaces of modular forms may be decomposed by some characters. For example Dirichlet characters decompose the vector space $\cM(\Gamma_1(N))_k$ into a direct sum of subspaces that we can analyze indepecdently.

\begin{Lem}\label{decom}Let $\varGamma$ be a congruence group and $\chi\in{\rm Hom}(\varGamma,\C^\times)$.

Also let $\varPhi\subset{\rm Hom}(\varGamma,\C^\times)$ be a finite subgroup and put $\varGamma'=\cap_{\phi\in\varPhi}\ker\phi$.

Suppose $|\varGamma/\varGamma'|=|\varPhi|=n$.

We decompose
\[\cM(\varGamma')_{(k,\chi)}=\VT{\phi\in\varPhi}\cM(\varGamma)_{(k,\chi\phi)}\]
if there exists a representation $\varDelta\subset\varGamma$ of $\varGamma/\varGamma'$ such that $1\in\varDelta\subset\ker\chi$,  and the orthogonal relations
\[\z \frac1n\VS{\alpha\in\varDelta}\phi(\alpha)=\delta_{\phi,{\tt 1}}\6\text{for }\phi\in\varPhi\]
\[\z \frac1n\VS{\phi\in\varPhi}\phi(\alpha)=\delta_{\alpha,1}\6\text{for }\alpha\in\varDelta\]
are satisfied.
\end{Lem}
\begin{proof}The natural decomposition is given by
\[\z f\mapsto(\pi_\phi(f))_{\phi\in\varPhi}\6\text{where }\pi_\phi(f)=\frac1n\VS{\alpha\in\varDelta}\overline{\phi(\alpha)}f|\alpha\]

This map is bijective by the orthogonal relations. Indeed, injectivity is followed from
\[\VS{\phi\in\varPhi}\pi_\phi(f)=\VS{\alpha\in\varDelta}\frac1n\VS{\phi\in\varPhi}\overline{\phi(\alpha)}f|\alpha=f.\]
For surjectivity, let $(g_\phi)_{\phi\in\varPhi}\in{\rm RHS}$ and put $f=\VS{\psi\in\varPhi}g_\psi$ then $f\in{\rm LHS}$ and
\[\z \pi_\phi(f)=\VS{\psi\in\varPhi}\frac1n\VS{\alpha\in\varDelta}\overline{\phi(\alpha)}g_\psi|\alpha=g_\phi\]

For $\gamma\in\varGamma$, taking $\beta\in\varDelta$ sach that $\phi(\gamma)=\phi(\beta)$ we get
\begin{align*}\pi_\phi(f)|\gamma&=\z\frac1n\VS{\alpha\in\varDelta}\overline{\phi(\alpha)}f|\alpha\gamma\beta^{-1}\alpha^{-1}\alpha\beta\\
&=\z\frac1n\VS{\alpha\in\varDelta}\chi(\alpha\gamma\beta^{-1}\alpha^{-1})\overline{\phi(\alpha)}f|\alpha\beta\\
&=\z\frac1n\VS{\alpha\in\varDelta}\chi(\gamma)\overline{\phi(\alpha\beta^{-1})}f|\alpha\\
&=\chi\phi(\gamma)\pi_\phi(f)\end{align*}

\end{proof}

Now, for $N\in\N$, put $\varGamma=\Gamma_0(N)$, $\varPhi={\rm Hom}((\Z/N\Z)^\times,\C^\times)$ and $\chi={\tt 1}$ in Lemma \ref{decom}. Then $\varGamma'=\Gamma_1(N)$ and we can chose $\varDelta\subset\varGamma$ such that $d_N:\varDelta\to(\Z/N\Z)^\times$ is bijective, thus the Nebentype decomposition
\[\cM(\Gamma_1(N))_k=\VT{\phi\in{\rm Hom}((\Z/N\Z)^\times,\C^\times)}\cM(\Gamma_0(N))_{(k,\phi)}.\]
is obtained (cf. \cite[\S4.3]{DS} or \cite[Proposition 9.2]{St}). More generally

\begin{Prop}\label{dec1}For $k\in\M$ and a subgroup $G\subset(\Z/N\Z)^\times$, we have
\[\cM(\Gamma_G(N))_k=\VT{\phi\in\varPhi}\cM(\Gamma_0(N))_{(k,\phi)}\]
where $\varPhi=\big\{f\in{\rm Hom}((\Z/N\Z)^\times,\C^\times) \,\big|\, f(G)=\{1\}\big\}$.
\end{Prop}

\newpage

It is important that
\[\cM(\Gamma(N))_k=\VT{\chi\in{\rm Hom}((\Z/N\Z)^\times,\C^\times)}\cM(N)_{(k,\chi)}.\]

This result has a more sophiscated representation
\[\cM(\Gamma(N))_\M=\cM(N)_{\M\times{\rm Hom}((\Z/N\Z)^\times,\C^\times)}.\]\\

We naturally regard $\M\subset\M\times{\rm Hom}(\varGamma,\C^\times)$ by $k=(k,{\tt 1}_\varGamma)$.

Also regard ${\rm Hom}(\varGamma,\C^\times)\subset\M\times{\rm Hom}(\varGamma,\C^\times)$ by $\chi=(0,\chi)$.

The subgroup generated by $x_1,\cdots,x_n$ is also denoted $\ang{x_1,\cdots,x_n}$. For example $\M=\ang1$ (remark $\Z=\ang1$ when we consider on groups) and
\[\cM(\Gamma(4))_\M=\cM(4)_{\M\times{\rm Hom}((\Z/4\Z)^\times,\C^\times)}=\cM(4)_{\ang{1,\chi_4}}.\]
Moreover, since $\cM(4)_{(k,\chi_4^\ell)}=\{0\}$ if $k+\ell$ is odd, we can representate
\[\cM(\Gamma(4))_\M=\cM(4)_{\ang{1^*}}.\]
Note $1^*=(1,\chi_4)$. Similarly we see
\[\cM(\Gamma(8))_\M=\cM(8)_{\ang{1^*,\chi_8}}.\]\\

Since the operators of half-integer weight are not actions, Lemma \ref{decom} can't be applied simply. However, as for $\kappa\in\M+\frac12$ we also have the Nebentype decomposition
\[\cM(\Gamma_1(N))_\kappa=\VT{\chi\in{\rm Hom}(\Z/N^\times,\C^\times)}\cM(\Gamma_0(N))_{(\kappa,\sqrt{\chi_4}\chi)}.\]

\begin{Prop}\label{dec2}For $\kappa\in\M+\frac12$ and a subgroup $G\subset\Z/N^\times$, if $\chi_4(G)=\{1\}$ then $\Gamma_G(N)\subset\Gamma_1(4)$ and
\[\cM(\Gamma_G(N))_{\kappa}=\VT{\phi\in\varPhi}\cM(\Gamma_0(N))_{(\kappa,\sqrt{\chi_4}\phi)}\]
where $\varPhi=\big\{f\in{\rm Hom}(\Z/N^\times,\C^\times) \,\big|\, f(G)=\{1\}\big\}$.
\end{Prop}

\vspace{0.5cm}

For example, for even $N$
\[\cM(\Gamma(N))_{\M+\frac12}=\cM(N)_{(\M+\frac12)\times\sqrt{\rho_4}{\rm Hom}((\Z/N\Z)^\times,\C^\times)}.\]

Let $w_2=(\frac12^*,\chi_8)=(\frac12,\chi_4\chi_8)$ then we can representate
\[\cM(\Gamma(4))_{\frac12\M}=\cM(4)_{\ang{w_2}}.\]
This decomposition is obtained also from structure Theorem \ref{4}.

Similarly Theorem \ref{5} induces
\[\cM(\Gamma(5))_{\frac15\M}=\cM(5)_{\ang{\frac15,\chi_5}}.\]

\newpage

We give other type decompositions.

Put $\varGamma=\Gamma_0(2)\cap\Gamma^0(2)$ and $\varPhi=\ang{c_4,b_4}$. Note $\ell_4:\varGamma\to\{\pm1\}$ are characters by Lemma \ref{b8c8} and $\varGamma'=\Gamma_0(4)\cap\Gamma^0(4)$. We can chose $\varDelta=\{1,(\BM 1&0\\2&1 \EM),(\BM 1&2\\0&1 \EM),(\BM 1&2\\2&5 \EM)\}$ thus
\[\cM(4)_{(k,\chi_4^k)}=\VT{\phi\in\varPhi}\cM(2)_{(k,\chi_4^k\phi)}\]
\[=\cM(2)_{(k,\chi_4^k)}\oplus\cM(2)_{(k,\chi_4^kc_4)}\oplus\cM(2)_{(k,\chi_4^kb_4)}\oplus\cM(2)_{(k,\chi_4^kc_4b_4)}.\]

This result has a more sophiscated representation
\[\cM(4)_{\ang{1^*}}=\cM(2)_{\ang{1^*,\ell_4}}\]
where the sequence $\ell_n=(c_n,b_n)$.

Change $\varPhi=\ang{c_8,b_8}$ then we see
\[\cM(8)_{\ang{1^*}}=\cM(2)_{\ang{1^*,\ell_8}}.\]

Change $\varGamma=\Gamma_0(4)\cap\Gamma^0(4)$ then we see
\[\cM(8)_{\ang{1^*,\chi_8}}=\cM(4)_{\ang{1^*,\chi_8,\ell_8}}\]\\

Next Lemma will be used many times.

\begin{Lem}\label{EX}
Let $S$ be a commutative semigroup and $T\subset S$ be a sub-semigroup. Suppose $2\leq n\in\N$ and $s\in S$ satisfies $S=\VT{i=0,1,2,\cdots,n-1}(T+is)$ and $ns\in T$.

Also let $R$ be a $S$-graded ring and $x\in R_s$. If there is an element $y\in R|_T$ and a homogeneous ideal $I\subset R$ such that $xy\in I$ and $R|_T\cap I=(R|_T)x^ny$ then
\[R=\VT{i=0,1,2,\cdots,n-1}(R|_T)x^i\]
\[I=Rxy\]
\end{Lem}
\begin{proof}For $t\in T$
\[R_{t+s}\subset \frac1{x^{n-1}y}(R|_T\cap I)\subset(R|_T)x\]
\[R_{t+2s}\subset \frac1{x^{n-2}y}(R|_T\cap I)\subset(R|_T)x^2\]
etc and
\[I_{t+s}\subset \frac1{x^{n-1}}(R|_T\cap I)\subset(R|_T)xy\]
etc.
\end{proof}

\newpage

\subsection{Second results}

For $A\subset\M$ we denote $f^A=\{f^a \,|\, a\in A\}$.

Start with Lemma \ref{4}
\[\cM(4)_{\ang{w_2}}=\C\big[\eta^{\natural\ang{\frac12}},\eta^{\natural\ang2}\big]=\C\big[\eta^{\flat\ang{\frac12}},\eta^{\sharp\ang2}\big]\]
\[\cS(4)_{\ang{w_2}}=\cM(4)_{\ang{w_2}}\eta^6\]

\begin{Thm}\label{8}
\[\cM(\Gamma(8))_{\frac12\M}=\cM(8)_{\ang{\frac12^*,\chi_8}}=\bigoplus\cM(4)_{\ang{w_2}}\eta^{\natural\{0,1\}}\eta^{\flat\{0,1\}}\eta^{\sharp\{0,1\}}\]
\[\cS(\Gamma(8))_{\frac12\M}=\cM(\Gamma(8))_{\frac12\M}\eta^3\]
\end{Thm}
\begin{proof}Note $\eta^6=\eta^{\natural2}\eta^{\flat2}\eta^{\sharp2}$. Lemma \ref{EX} shows
\[\cM(4)_{\ang{\frac12^*,\chi_8}}=\bigoplus\cM(4)_{\ang{w_2}}\eta^{\natural\{0,1\}}\]
\[\cS(4)_{\ang{\frac12^*,\chi_8}}=\cM(4)_{\ang{\frac12^*,\chi_8}}\eta^{\natural}\eta^{\flat2}\eta^{\sharp2}\]

Again Lemma \ref{EX} shows
\[\cM(4)_{\ang{\frac12^*,\chi_8,c_8}}=\bigoplus\cM(4)_{\ang{\frac12^*,\chi_8}}\eta^{\flat\{0,1\}}\]
\[\cS(4)_{\ang{\frac12^*,\chi_8,c_8}}=\cM(4)_{\ang{\frac12^*,\chi_8,c_8}}\eta^{\natural}\eta^{\flat}\eta^{\sharp2}\]

Once more Lemma \ref{EX} shows
\[\cM(4)_{\ang{\frac12^*,\chi_8,\ell_8}}=\bigoplus\cM(4)_{\ang{\frac12^*,\chi_8,c_8}}\eta^{\sharp\{0,1\}}\]
\[\cS(4)_{\ang{\frac12^*,\chi_8,\ell_8}}=\cM(4)_{\ang{\frac12^*,\chi_8,\ell_8}}\eta^3\]

We have seen $\cM(\Gamma(8))_{\frac k2}=\cM(4)_{\frac k2^*+\ang{\chi_8,\ell_8}}$ after Lemma \ref{decom} for even $k$.

This assertion stays hold for odd $k$ since
\begin{align*}\cM(\Gamma(8))_{\frac k2}\eta^3&\subset\cS(\Gamma(8))_{\frac{k+3}2}\\
&=\cS(4)_{\frac{k+3}2^*+\ang{\chi_8,\ell_8}}=\cM(4)_{\frac k2^*+\ang{\chi_8,\ell_8}}\eta^3\end{align*}

\end{proof}

The above Theorem induces a new dimension formula
\[\VS{k\in\M}\dim\cM(\Gamma(8))_{\frac k2}\cdot t^{\frac k2}=\dfrac{(1+t^{\frac12})^3}{(1-t^{\frac12})^2}=\VS{k\in\M}\big(8k-4+\delta_{k,0}+\delta_{k,1}\big)t^{\frac k2}\]

\begin{Lem}\label{16a}
\[\cM(16)_{\ang{\frac14^*,\chi_8}}=\bigoplus\cM(8)_{\ang{\frac12^*,\chi_8}}\sqrt{\eta^{\natural}}\,\!^{\{0,1\}}\sqrt{\eta^{\flat}}\,\!^{\{0,1\}}\sqrt{\eta^{\sharp}}\,\!^{\{0,1\}}\]
\[\cS(16)_{\ang{\frac14^*,\chi_8}}=\cM(16)_{\ang{\frac14^*,\chi_8}}\sqrt{\eta}^3\]
\end{Lem}
\begin{proof}First we see $\cM(8)_{\ang{\frac14^*,\chi_8,\ell_{16}}}={\rm RHS}$ and
\[\cS(8)_{\ang{\frac14^*,\chi_8,\ell_{16}}}=\cM(8)_{\ang{\frac14^*,\chi_8,\ell_{16}}}\sqrt{\eta}^3.\]

Lemma \ref{b8c8} and Lemma \ref{decom} lead to $\cM(16)_{\frac k4^*+\ang{\chi_8}}=\cM(4)_{\frac 4k^*+\ang{\chi_8,\ell_{16}}}$ for $k\in4\N$.

This assertion stays hold for $k\in4\M+1$ since
\begin{align*}\cM(16)_{\frac k4^*+\ang{\chi_8}}\sqrt{\eta}^3&\subset\cS(16)_{\frac{k+3}4^*+\ang{\chi_8}}\\
&=\cS(4)_{\frac{k+3}4^*+\ang{\chi_8,\ell_{16}}}\\
&=\cM(4)_{\frac k4^*+\ang{\chi_8,\ell_{16}}}\sqrt{\eta}^3\end{align*}

It stays also hold for $k\in4\M+2$ and then for $k\in4\M+3$ finally.
\end{proof}

The above Theorem induces a new dimension formula
\[\VS{k\in\M}\dim\cM(16)_{\frac k4^*+\ang{\chi_8}}t^{\frac k4}=\dfrac{(1+t^{\frac12})^3(1+t^{\frac14})^3}{(1-t^{\frac12})^2}=\dfrac{(1+t^{\frac12})^3(1+t^{\frac14})}{(1-t^{\frac14})^2}\]\\

\begin{Lem}
\[\cM(3)_{\ang{\frac13,\chi_9}}=\C\big[\sqrt[3]{\eta^{\nwarrow}},\sqrt[3]{\eta^{\swarrow}}\big]\]
\[\cS(3)_{\ang{\frac13,\chi_9}}=\cM(3)_{\ang{\frac13,\chi_9}}\eta^{\bot}\eta^{\top}\sqrt[3]{\eta^{\nwarrow}\eta^{\swarrow}}\]
\end{Lem}
\begin{proof}The dimension formuls states $\dim\cS(3)_{\frac k3+\ang{\chi_9}}=(k-7)^+$ for $k\in3\M$ and
\[\dim\cS(3)_{\frac k3+\ang{\chi_9}}\leq\dim\cS(3)_{\frac{k+1}3+\ang{\chi_9}}-1=(k-7)^+\]
for $k\in3\M+2$ and also for $k\in3\M+1$.
\end{proof}

\begin{Thm}
\[\cM(\Gamma(9))_{\frac13\M}=\cM(9)_{\ang{\frac13,\chi_9}}=\bigoplus\cM(3)_{\ang{\frac13,\chi_9}}\sqrt[3]{\eta^{\bot}}^{\{0,1,2\}}\sqrt[3]{\eta^{\top}}^{\{0,1,2\}}\]
\[\cS(\Gamma(9))_{\frac13\M}=\cM(\Gamma(9))_{\frac13\M}\sqrt[3]{\eta^8}\]
\end{Thm}
\begin{proof}Lemma \ref{EX} shows
\[\cM(3)_{\ang{\frac13,\chi_9,\ell_9}}=\bigoplus\cM(3)_{\ang{\frac13,\chi_9}}\sqrt[3]{\eta^{\bot}}^{\{0,1,2\}}\sqrt[3]{\eta^{\top}}^{\{0,1,2\}}\]
\[\cS(3)_{\ang{\frac13,\chi_9,\ell_9}}=\cM(3)_{\ang{\frac13,\chi_9,\ell_9}}\sqrt[3]{\eta^8}\]

Lemma \ref{b3c3} and Lemma \ref{decom} lead to $\cM(\Gamma(9))_{\frac k3}=\cM(3)_{\frac k3+\ang{\chi_9,\ell_9}}$ for $k\in3\N$.

This assertion stays hold for $k\in3\M+2$ since
\begin{align*}\cM(\Gamma(9))_{\frac k3}\sqrt[3]{\eta^8}&\subset\cS(\Gamma(9))_{\frac{k+4}3}\\
&=\cS(3)_{\frac{k+4}3+\ang{\chi_9,\ell_9}}\\
&=\cM(3)_{\frac k3+\ang{\chi_9,\ell_9}}\sqrt[3]{\eta^8}\end{align*}
It stays hold also for $k\in3\M+1$.
\end{proof}

The above Theorem induces a new dimension formula
\[\VS{k\in\M}\dim\cM(\Gamma(9))_{\frac k3}t^{\frac k3}=\dfrac{(1+t^{\frac13}+t^{\frac23})^2}{(1-t^{\frac13})^2}=\VS{k\in\M}(9k-9+\delta_{k,0}+4\delta_{k,1}+\delta_{k,2})t^{\frac k3}.\]

\newpage

\section{level 2-power cases}

\subsection{$\eta$-$\frac12,\frac14$ form}

For odd $c$, let $[\eta,\frac c2]=\eta^{\natural\ang{\frac12}}|(\BM 1&c\\0&1\EM)$. Concretely
\[\z [\eta,\frac12]=\eta^{\natural\ang2}+2i\eta^{\sharp\ang2}=\VP{n\in\N}\dfrac{1-(-iq)^{\frac n4}}{1+(-iq)^{\frac n4}}\]
\[\z [\eta,\frac32]=\eta^{\natural\ang2}-2i\eta^{\sharp\ang2}=\VP{n\in\N}\dfrac{1-(iq)^{\frac n4}}{1+(iq)^{\frac n4}}\]
thus
\[\z \sqrt{[\eta,\frac12][\eta,\frac32]}=\eta^{\natural}\]

\begin{Prop}\label{theta16}
\begin{align*}\theta_{\chi_{16}}^{\ang{\frac1{32}}}&=\z\sqrt[4]{\eta^{\flat}\eta^{\sharp}}\sqrt{[\eta,\frac32]}\\
\theta_{\chi_{12}\chi_{16}}^{\ang{\frac1{96}}}&=\z\sqrt[4]{\eta^{\natural}\eta}\sqrt{[\eta,\frac32]}\end{align*}
\end{Prop}
\begin{proof}Note $\chi_{16}(n)=1^{\frac{3(n^2-1)}{32}}$ for $n\in2\M+1$.

Acting $(\BM1&3\\0&1\\\EM)$ on $\VS{n\in2\M+1}q^{\frac{n^2}{32}}=\sqrt[4]{\eta^{\natural}\eta^{\sharp}}\sqrt{\eta^{\natural\ang{\frac12}}}$ yields the former identity.

Acting $(\BM1&9\\0&1\\\EM)$ on $\VS{n\in\N}q^{\frac{n^2}{96}}=\sqrt[4]{\eta^{\flat}\eta}\sqrt{\eta^{\flat\ang{\frac12}}}$ yields the latter identity.
\end{proof}

\vspace{0.5cm}

Next, for odd $c$, let
\[\z [\eta,\frac c4]=\eta^{\natural\ang{\frac14}}|(\BM 1&c\\0&1\EM)=\eta^{\flat}+2\cdot1^{\frac c8}\eta^{\sharp}=\VP{n\in\N}\dfrac{1-(-1^{\frac c8}q)^{\frac n4}}{1+(-1^{\frac c8}q)^{\frac n4}}\]
We see $\sqrt{[\eta,\frac c4][\eta,\frac c4+1]}=[\eta,\frac c2+1]$

\begin{Prop}\label{theta32}
\begin{align*}
\theta_{\chi_{32}}^{\ang{\frac1{64}}}&=\z\sqrt[8]{\eta^{\flat}\eta^{\sharp}}\sqrt[4]{[\eta,\frac32]}\sqrt{[\eta,\frac34]}\\
\theta_{\chi_{12}\chi_{32}}^{\ang{\frac1{192}}}&=\z\sqrt[8]{\eta^{\natural}\eta}\sqrt[4]{[\eta,\frac32]}\sqrt{[\eta,\frac54]}\end{align*}
\end{Prop}
\begin{proof}
Note $\chi_{32}(n)=1^{\frac{3(n^2-1)}{64}}$ for $n\in2\M+1$.

Acting $(\BM1&3\\0&1\\\EM)$ on $\VS{n\in2\M+1}q^{\frac{n^2}{64}}=\sqrt[8]{\eta^{\natural}\eta^{\sharp}}\sqrt[4]{\eta^{\natural\ang{\frac12}}}\sqrt{\eta^{\natural\ang{\frac14}}}$ yields the former identity.

Acting $(\BM1&9\\0&1\\\EM)$ on $\VS{n\in2\M+1}q^{\frac{n^2}{192}}=\sqrt[8]{\eta^{\flat}\eta}\sqrt[4]{\eta^{\flat\ang{\frac12}}}\sqrt{\eta^{\flat\ang{\frac14}}}$ yields the latter identity.
\end{proof}

Acting $\sigma_8:1^{\frac18}\mapsto1^{\frac38}$ yields
\[\z\theta_{\chi_{32}^3}^{\ang{\frac1{64}}}=\sqrt[8]{\eta^{\flat}\eta^{\sharp}}\sqrt[4]{[\eta,\frac12]}\sqrt{[\eta,\frac14]}\]
\[\z\theta_{\chi_{12}\chi_{32}^3}^{\ang{\frac1{192}}}=\sqrt[8]{\eta^{\natural}\eta}\sqrt[4]{[\eta,\frac12]}\sqrt{[\eta,\frac74]}\]

\newpage

\subsection{o-form, d-form, u-form}

Put
\[\z [d,\frac k4]=\sqrt{\eta^{\natural\ang2}}+1^{\frac k8}\sqrt{2\eta^{\sharp\ang2}}\]
\[\z [u,\frac k4]=\frac1{1+1^{\frac k8}}\Big(\sqrt{\eta^{\natural\ang{\frac12}}}+1^{\frac k8}\sqrt{\eta^{\flat\ang{\frac12}}}\Big)\6[u,1]=\frac12\Big(\sqrt{\eta^{\natural\ang{\frac12}}}-\sqrt{\eta^{\flat\ang{\frac12}}}\Big)\]
and $[o,\frac k4]=[u,\frac k4]|(\BM 1&1\\0&1\EM)$.

\begin{Lem}\label{oud}
\[\z [o,\frac34][o,\frac74]=[d,\frac12][d,\frac32]=\eta^{\natural\ang{\frac12}}\]
\[\z [o,\frac14][o,\frac54]=[d,0][d,1]=\eta^{\flat\ang{\frac12}}\]
\[\z [o,\frac12][o,\frac32]=[u,\frac12][u,\frac32]=\eta^{\natural\ang2}\]
\[[o,0][o,1]=[u,0][u,1]=\eta^{\sharp\ang2}\]
\[\z [d,\frac14][d,\frac54]=[u,\frac34][u,\frac74]=[\eta,\frac32]\]
\[\z [d,\frac34][d,\frac74]=[u,\frac14][u,\frac54]=[\eta,\frac12]\]
\end{Lem}
\begin{proof}LHS of the first identity is $\frac1{1+i}[\eta,\frac12]+\frac1{1-i}[\eta,\frac32]$.

LHS of the third identity is $\frac12[\eta,\frac12]+\frac12[\eta,\frac32]$.\end{proof}

\begin{Lem}\label{ud}
\[\z \sqrt{[o,0][o,\frac12]}=[o,0]^{\ang{\frac12}} \6 \sqrt{[o,\frac74][o,\frac14]}=[o,\frac14]^{\ang2}\]
\[\z \sqrt{[d,\frac74][d,\frac14]}=[o,\frac12]^{\ang{\frac12}} \6 \sqrt{[u,\frac34][u,\frac54]}=[o,\frac34]^{\ang2}\]
\[\z \sqrt{[o,1][o,\frac32]}=[o,1]^{\ang{\frac12}} \6 \sqrt{[o,\frac34][o,\frac54]}=[o,\frac54]^{\ang2}\]
\[\z \sqrt{[d,\frac34][d,\frac54]}=[o,\frac32]^{\ang{\frac12}} \6 \sqrt{[u,\frac74][u,\frac14]}=[o,\frac74]^{\ang2}\]
and
\[\z \sqrt{[o,0][o,\frac32]}=[u,0]^{\ang{\frac12}} \6 \sqrt{[o,\frac14][o,\frac34]}=[d,0]^{\ang2}\]
\end{Lem}
\begin{proof}On the first identity, LHS is
\[\z\sqrt{\frac1{2(1+i)}\big([\eta,\frac12]+(1+i)\eta^{\natural}+i[\eta,\frac32]\big)}=\sqrt{\frac12(\eta^{\natural\ang2}+2\eta^{\sharp\ang2}+\eta^{\natural})},\]
and RHS is
\[\z \frac12\sqrt{[\eta,\frac12]+2\eta^{\natural}+[\eta,\frac32]}^{\ang{\frac12}}=\sqrt{\frac12(\eta^{\natural\ang2}+\eta^{\natural})}^{\ang{\frac12}}.\]
\end{proof}

We also see
\[\z \sqrt{[u,0][u,\frac12]}=\sqrt{[o,0][o,\frac32]}\big|(\BM 1&1\\0&1\EM)=\frac12\Big(\sqrt{[\eta,\frac14]}+\sqrt{[\eta,\frac54]}\Big)\]
\[\z \sqrt{[u,1][u,\frac32]}=\sqrt{[o,1][o,\frac12]}\big|(\BM 1&1\\0&1\EM)=\frac{1^{\frac78}}2\Big(\sqrt{[\eta,\frac14]}-\sqrt{[\eta,\frac54]}\Big)\]
\[\z \sqrt{[d,0][d,\frac32]}=\sqrt{[d,\frac74][d,\frac14]}\big|(\BM 1&1\\0&1\EM)=\frac1{1+i}\sqrt{[\eta,\frac14]}+\frac1{1-i}\sqrt{[\eta,\frac54]}\]

\newpage

\begin{Prop}
\[\z \theta_{\chi_{64}}^{\ang{\frac1{128}}}=\sqrt[16]{\eta^{\flat}\eta^{\sharp}}\sqrt[8]{[\eta,\frac32]}\sqrt[4]{[\eta,\frac74]}\cdot[o,\frac12]^{\ang{\frac14}}|(\BM1&3\\0&1\\\EM)\]
\[\z \theta_{\chi_{12}\chi_{64}}^{\ang{\frac1{392}}}=\sqrt[16]{\eta^{\natural}\eta}\sqrt[8]{[\eta,\frac32]}\sqrt[4]{[\eta,\frac14]}\cdot[o,\frac12]^{\ang{\frac14}}|(\BM1&9\\0&1\\\EM)\]\end{Prop}
\begin{proof}Compute
\[\VS{n\in16\M+\{1,15\}}q^{\frac{n^2}{32}}-\VS{n\in16\M+\{7,9\}}q^{\frac{n^2}{32}}=\frac12\sqrt[4]{\eta^{\flat}\eta^{\sharp}}\big(\sqrt{[\eta,\frac12]}+\sqrt{[\eta,\frac32]}\big)\]
\[-\VS{n\in16\M+\{3,13\}}q^{\frac{n^2}{32}}+\VS{n\in16\M+\{5,11\}}q^{\frac{n^2}{32}}=\frac i2\sqrt[4]{\eta^{\flat}\eta^{\sharp}}\big(\sqrt{[\eta,\frac12]}-\sqrt{[\eta,\frac32]}\big)\]
Acting $\ang{\frac14}$ and $(\BM1&11\\0&1\\\EM)$ yields
\[\VS{n\in8\M+\{1,7\}}\chi_{64}(n)q^{\frac{n^2}{128}}=\frac12\sqrt[16]{\eta^{\flat}\eta^{\sharp}}\sqrt[8]{[\eta,\frac32]}\sqrt[4]{[\eta,\frac74]}\big(\sqrt{[\eta,\frac{15}8]}+\sqrt{[\eta,\frac78]}\big)\]
\[\VS{n\in8\M+\{3,5\}}\chi_{64}(n)q^{\frac{n^2}{128}}=\frac i2\sqrt[16]{\eta^{\flat}\eta^{\sharp}}\sqrt[8]{[\eta,\frac32]}\sqrt[4]{[\eta,\frac74]}\big(\sqrt{[\eta,\frac{15}8]}-\sqrt{[\eta,\frac78]}\big)\]
where $[\eta,\frac c8]=\eta^{\natural\ang{\frac18}}|(\BM1& c\\0&1\\\EM)$ for odd $c$.
\end{proof}

Acting $\sigma_{16}:1^{\frac1{16}}\mapsto1^{\frac3{16}}$ on the last identity yields
\[\z \theta_{\chi_{64}^3}^{\ang{\frac1{128}}}=\sqrt[16]{\eta^{\flat}\eta^{\sharp}}\sqrt[8]{[\eta,\frac12]}\sqrt[4]{[\eta,\frac54]}\cdot[o,\frac12]^{\ang{\frac14}}|(\BM1&9\\0&1\\\EM)\]

Similarly
\[\z \theta_{\chi_{64}^5}^{\ang{\frac1{128}}}=\sqrt[16]{\eta^{\flat}\eta^{\sharp}}\sqrt[8]{[\eta,\frac32]}\sqrt[4]{[\eta,\frac34]}\cdot[o,\frac12]^{\ang{\frac14}}|(\BM1&-1\\0&1\\\EM)\]
\[\z \theta_{\chi_{64}^7}^{\ang{\frac1{128}}}=\sqrt[16]{\eta^{\flat}\eta^{\sharp}}\sqrt[8]{[\eta,\frac12]}\sqrt[4]{[\eta,\frac14]}\cdot[o,\frac12]^{\ang{\frac14}}|(\BM1&5\\0&1\\\EM)\]\\

I conjecture
\[\z[o,0]=\VP{n\in\N}(1+q^{2n})\sqrt{\dfrac{1-q^{2n}}{(1-q^{\frac n2})^{\chi_8(n)}}}\]
\[\z[o,1]=q^{\frac14}\VP{n\in\N}(1+q^{2n})\sqrt{(1-q^{2n})(1-q^{\frac n2})^{\chi_8(n)}}\]
hence
\[{\z [o,\frac12]}=\sqrt{\eta^{\natural\ang2}}\frac{[o,0]^{\ang{\frac12}}}{[o,0]^{\ang{\frac12}}\big|(\BM 1&2\\0&1\EM)}=\VP{n\in\N}\sqrt{\dfrac{1-(-q)^n}{1+(-q)^n}\dfrac{1+\chi_8(n)q^{\frac n4}}{1-\chi_8(n)q^{\frac n4}}}\]
\[{\z [o,\frac32]}=\sqrt{\eta^{\natural\ang2}}\frac{[o,0]^{\ang{\frac12}}\big|(\BM 1&2\\0&1\EM)}{[o,0]^{\ang{\frac12}}}=\VP{n\in\N}\sqrt{\dfrac{1-(-q)^n}{1+(-q)^n}\dfrac{1-\chi_8(n)q^{\frac n4}}{1+\chi_8(n)q^{\frac n4}}}\]

\[\z [o,\frac14][o,\frac34]=\eta^{\flat}\VP{n\in\N}\dfrac{1+\chi_8(n)(1^{\frac78}q^{\frac14})^n}{1-\chi_8(n)(1^{\frac78}q^{\frac14})^n}\cdot\dfrac{1+\chi_8(n)(1^{\frac18}q^{\frac14})^n}{1-\chi_8(n)(1^{\frac18}q^{\frac14})^n}\]

\newpage

\subsection{Formal weights a,A}

For a formal symbol $x$ we may make a group
\[(\Q/\Z)x=\{ax \,|\, a\in\Q/\Z\}.\]

For $h=\frac{h_2}{h_1}$ such that $h_1,h_2\in\N$ and $\gcd(h_1,h_2)=1$, it is convenient to make new operators of weight in $\Q+(\Q/\Z)\ang h$ by
\[\z \sqrt[N]{\eta^{\ang h}}^{24}\in\cS(\Gamma(Nh_1h_2))_{\frac{24}N+\frac{24}N\ang h}.\]\\

Note $\eta^{\ang{\frac12}3}\in\cS(2)_{(\frac32^*,\chi_8c_4b_{16})}$.

Write $\ia=\ang{\frac12}$ and $\iA=\ang2$. Weight $\frac{\ia}2$ and $\frac{\iA}2$ operators on $\Gamma(32)$ defined above are exanpded to $\Gamma_0(2)\cap\Gamma^0(2)$ by
\[\sqrt{\eta^{\ang{\frac12}3}}\in\cS(2)_{(\frac34^*+\frac\ia2,\overline{{\chi_{16}}}c_8b_{32})}\6\sqrt{\eta^{\ang23}}\in\cS(2)_{(\frac34^*+\frac\iA2,\chi_{16}c_{32}b_8)}\]
where  $(\frac k{2^n}^*+a,\chi)=(\frac k{2^n}+a,\xi_8^{k/2^{n-1}}\chi)$ formally.

Then $\sqrt{\eta^{\ang{\frac12}}}=\frac{\eta^{\ang{\frac12}2}}{\sqrt{\eta^{\ang{\frac12}3}}}\in\cS(6)_{(\frac14^*+\frac\ia2,\chi_{16}c_{24}b_{96})}$ and $\sqrt{\eta^{\ang2}}\in\cS(6)_{(\frac14^*+\frac\iA2,\overline{\chi_{16}}c_{96}b_{24})}$.\\

\begin{Lem}
\[\sqrt{\eta^{\flat\ang{\frac12}}}\in\cM(2,4)_{(\frac14^*+\frac\ia2,\overline{\chi_{16}}c_8)}\z\6\sqrt{\eta^{\sharp\ang2}}\in\cM(4,2)_{(\frac14^*+\frac\iA2,\chi_{16}b_8)}\]
\[\sqrt{\eta^{\natural\ang{\frac12}}}\in\cM(2,4)_{(\frac14^*+\frac\ia2,\chi_{16})}\6\sqrt{\eta^{\natural\ang2}}\in\cM(4,2)_{(\frac14^*+\frac\iA2,\overline{\chi_{16}})}\]
\end{Lem}
\begin{proof}$\sqrt{\eta^{\flat\ang{\frac12}}}=\dfrac{\sqrt{\eta^{\ang{\frac12}3}}}{\eta^{\sharp\ang{\frac14}}}$.
\end{proof}

Note
\[\sqrt{\eta^{\flat\ang{\frac12}}}\big|_{\frac14+\frac{\ia+\iA}2}(\BM 1&2\\0&1\EM)=\dfrac{\sqrt{\eta^{\ang{\frac12}3}}\big|_{\frac34+\frac{\ia+\iA}2}(\BM 1&2\\0&1\EM)}{\eta^{\sharp\ang{\frac14}}\big|_{\frac12}(\BM 1&2\\0&1\EM)}=\frac{1^{\frac1{16}}\sqrt{\eta^{\ang{\frac12}3}}}{1^{\frac1{16}}\sqrt{\eta^{\flat\ang{\frac12}}\eta^{\sharp\ang{\frac12}}}}=\sqrt{\eta^{\natural\ang{\frac12}}}\]\\

\begin{Lem}
\[\sqrt[4]{\eta^{\flat}}\in\cM(2)_{(\frac18^*+\frac\ia2,\chi_{16}c_{32})}\6\sqrt[4]{\eta^{\sharp}}\in\cM(2)_{(\frac18^*+\frac\iA2,\overline{\chi_{16}}b_{32})}\]
\[\sqrt[4]{\eta^{\natural}}\in\cM(2)_{\frac18^*-\frac{\ia+\iA}2}\]
\end{Lem}
\begin{proof}First $\sqrt{\eta^{\sharp\ang{\frac12}}}=\dfrac{\sqrt{\eta^{\ang{\frac12}3}}}{\eta^{\flat}}\in\cM(2)_{(\frac14^*+\frac\ia2,\overline{\chi_{16}}b_{32})}$ and\\
$\sqrt[4]{\eta^{\flat}}=\dfrac{\sqrt[4]{\eta^3}}{\sqrt{\eta^{\sharp\ang{\frac12}}}}\in\cM(2)_{(\frac18^*+\frac\ia2,\chi_{16}c_{32})}$.
\end{proof}

\newpage

\begin{Lem}
\[\z \sqrt{[\eta,\frac12]}\in\cM(4)_{(\frac14^*+\frac{\ia+\iA}2,\overline{\chi_{16}})}\6\sqrt{[\eta,\frac32]}\in\cM(4)_{(\frac14^*+\frac{\ia+\iA}2,\chi_{16})}\]\end{Lem}
\begin{proof}Put $f=\sqrt[4]{\eta^{\flat}\eta^{\sharp}}\sqrt{[\eta,\frac32]}$. Then $f\in\cM(32)_{(\frac12^*,\chi_{16})}$ by Proposition \ref{theta16}.

In addtion $f|(\BM 1&4\\0&1\EM)=1^{\frac18}f$.

Note if $\tau\in\cH$ and $|\tau|=1$ then $-\frac1{\tau}=-\overline\tau$ and
\[\z\big(f|(\BM 0&-1\\1&0\EM)\big)(\tau)=\frac1{\sqrt{\tau}}f(-\frac1{\tau})=\frac1{\sqrt{\tau}}\overline{f^{\sigma_4}(\tau)}\]
and
\[\z \big(f|(\BM 0&-1\\1&0\EM)\big)|(\BM 1&-4\\0&1\EM)=1^{\frac18} f|(\BM 0&-1\\1&0\EM).\]

We see $(\BM 1&0\\4&1\EM)=(\BM 0&-1\\1&0\EM)(\BM 1&-4\\0&1\EM)(\BM 0&1\\-1&0\EM)$ and
\[f|(\BM 1&0\\4&1\EM)=-\big((f|(\BM 0&-1\\1&0\EM)\big)\big|(\BM 1&-4\\0&1\EM)\big)\big|(\BM 0&1\\-1&0\EM)=1^{\frac18}f\]

In particular $f\in\cM(4)_{(\frac12^*,\chi_{16}c_{32}b_{32})}$.
\end{proof}

Note \[\iE_{\chi_4\chi_{16}}=1+(1-i)\VS{n\in\N}(\chi_4\chi_{16}*{\tt 1}_\N)(n)q^n\]

The dimension formula states $\dim\cM(4)_{(1^*,\chi_{16})}=1$ and we see
\[\z \iE_{\chi_4\chi_{16}}^{\ang{\frac14}}=\sqrt{\eta^{\natural\ang{\frac12}}\eta^{\natural}\eta^{\natural\ang2}[\eta,\frac32]}\]\\

\begin{Prop}\label{aA}We have an identity in $\cM(4)_{\frac12^*+\frac{\ia+\iA}2}$
\[\z\sqrt{[\eta,\frac74][\eta,\frac14]}=\sqrt{\eta^{\natural\ang{\frac12}}\eta^{\natural\ang2}}+\sqrt2\sqrt{\eta^{\flat\ang{\frac12}}\eta^{\sharp\ang2}}\]
\end{Prop}
\begin{proof}Act $(\BM 1&1\\0&1\EM)$ on
\begin{align*}\z \sqrt{\eta^{\natural\ang{\frac14}}[\eta,\frac32]^{\ang{\frac12}}}&\z =\big(\sqrt{[o,\frac32][o,0]}+\sqrt{[o,\frac12][o,1]}\big)\big(\sqrt{[o,0][o,\frac12]}-i\sqrt{[o,\frac32][o,1]}\big)\\
&\z =\big([o,0]-i[o,1]\big)\sqrt{\eta^{\natural\ang2}}+\big([o,\frac12]-i[o,\frac32]\big)\sqrt{\eta^{\sharp\ang2}}\\
&\z =\sqrt{[\eta,\frac32]\eta^{\natural\ang2}}+(1-i)\sqrt{[\eta,\frac12]\eta^{\sharp\ang2}}\end{align*}
\end{proof}

\newpage

\subsection{Structure theorems with $\chi_{16}$}

Let $w_4=(\frac14^*+\frac{\ia+\iA}2,\chi_{16})$.

\begin{Lem}\label{K4}
\[\z \cM(4|\chi_8)_{\ang{w_4}}=\bigoplus\C\Big[\sqrt{[\eta,\frac12]},\sqrt{[\eta,\frac32]}\Big]\]
\[\cS(4|\chi_8)_{\ang{w_4}}=\cM(4|\chi_8)_{\ang{w_4}}\eta^{\natural}\eta^{\flat2}\eta^{\sharp2}\]
\end{Lem}
\begin{proof}Start with $\cM(4|\chi_8)_{\ang{\frac12^*}}=\bigoplus\cM(4)_{\ang{w_2}}\eta^{\natural\{0,1\}}$ and
\[\dim\cS(4|\chi_8)_{\frac k2^*}=(2k-9)^+\]
by Theorem \ref{8}. It induces
\[\z\dim\cS(4|\chi_8)_{\ang{w_4}}|_{\frac k4}\leq(k-9)^+\]
\end{proof}

Let the sequence
\[\z \alpha_2=(\frac\ia2,\frac\iA2)\]

We see some expansions
\[\z \cM(4|\chi_8)_{\ang{w_4,\alpha_2}}=\bigoplus\cM(4|\chi_8)_{\ang{w_4}}\sqrt{\eta^{\natural\ang{\frac12}}}^{\{0,1\}}\sqrt{\eta^{\natural\ang2}}^{\{0,1\}}\]
\[\cS(4|\chi_8)_{\ang{w_4,\alpha_2}}=\cM(4|\chi_8)_{\ang{w_4,\alpha_2}}\eta^{\natural}\sqrt{\eta^{\natural\ang{\frac12}}\eta^{\natural\ang2}}\eta^{\flat\ang{\frac12}}\eta^{\sharp\ang2}\]
and
\[\cM(4|\chi_8)_{\ang{w_4,\alpha_2,\ell_8}}=\bigoplus\cM(4|\chi_8)_{\ang{w_4,\alpha_2}}\sqrt{\eta^{\flat\ang{\frac12}}}^{\{0,1\}}\sqrt{\eta^{\sharp\ang2}}^{\{0,1\}}\]
\[\cS(4|\chi_8)_{\ang{w_4,\alpha_2,\ell_8}}=\cM(4|\chi_8)_{\ang{w_4,\alpha_2,\ell_8}}\eta^3\]
and
\[\cM(4|\chi_8)_{\ang{\frac14^*,\alpha_2,\chi_{16},\ell_{16}}}=\bigoplus\cM(4|\chi_8)_{\ang{w_4,\alpha_2,\ell_8}}\sqrt{\eta^{\natural}}\,\!^{\{0,1\}}\sqrt{\eta^{\flat}}\,\!^{\{0,1\}}\sqrt{\eta^{\sharp}}\,\!^{\{0,1\}}\]
\[\cS(4|\chi_8)_{\ang{\frac14^*,\alpha_2,\chi_{16},\ell_{16}}}=\cM(4|\chi_8)_{\ang{\frac14^*,\alpha_2,\chi_{16},\ell_{16}}}\sqrt{\eta}^3\]

Note $\cM(4)_{\ang{\frac14^*,\alpha_2,\chi_{16},\ell_{16}}}=\cM(4|\chi_8)_{\ang{\frac14^*,\alpha_2,\chi_{16},\ell_{16}}}$. Indeed, $\subset$ follows from \[\Gamma_0(4)\cap\Gamma^0(4)\supset\Gamma_0(4)\cap\Gamma^0(4)\cap\ker(\chi_8)\]
and $\supset$ does from the above structure assertion.
\\

\begin{Thm}
\[\cM(\Gamma(16))_{\frac14\M}=\cM(16)_{\ang{\frac14^*,\chi_{16}}}=\cM(4)_{\ang{\frac14^*,\chi_{16},\ell_{16}}}\]
\[\cS(\Gamma(16))_{\frac14\M}=\cM(\Gamma(16))_{\frac14\M}\sqrt{\eta}^3\]
\end{Thm}
\begin{proof}On the first assertion, both $\supset$ are trivial.

Restriction of the above result lead to
\[\cS(4)_{\ang{\frac14^*,\chi_{16},\ell_{16}}}=\cM(4)_{\ang{\frac14^*,\chi_{16},\ell_{16}}}\sqrt{\eta}^3\]
Similarly as Lemma \ref{16a} we see $\cM(\Gamma(16))_{\frac k4}=\cM(4)_{\frac k4^*+\ang{\chi_{16},\ell_{16}}}$.
\end{proof}

\newpage

The Hilbert function of $\cM(4)_{\ang{\frac14^*,\alpha_2,\chi_{16},\ell_{16}}}$ is
\[\frac{(1+t^{\frac14+\frac{\ia+\iA}2})^2(1+t^{\frac14+\frac\ia2})(1+t^{\frac14+\frac\iA2})(1+t^{\frac14})^3}{(1-t^{\frac14+\frac\ia2})(1-t^{\frac14+\frac\iA2})}=\frac{(1+t^{\frac14+\frac{\ia+\iA}2})^2h_{41}(1+t^{\frac14})}{(1-t^{\frac14})^2}\]
where
\begin{align*}h_{41}&=(1+t^{\frac14+\frac\ia2})^2(1+t^{\frac14+\frac\iA2})^2\\
&=(1+t^{\frac12})^2+2(t^{\frac14+\frac\ia2}+t^{\frac14+\frac\iA2})(1+t^{\frac12})+4t^{\frac12+\frac{\ia+\iA}2}\end{align*}
and a new dimension formula is
\[\VS{k\in\M}\dim\cM(\Gamma(16))_{\frac k4}\cdot t^{\frac k4}=\dfrac{((1+t^{\frac12})^3+8t^{\frac34})(1+t^{\frac14})}{(1-t^{\frac14})^2}\]\\

A more expansion is
\[\cM(4)_{\ang{\frac18^*,\alpha_2,\chi_{16},\ell_{32}}}=\bigoplus\cM(4)_{\ang{\frac14^*,\alpha_2,\chi_{16},\ell_{16}}}\sqrt[4]{\eta^{\natural}}\,\!^{\{0,1\}}\sqrt[4]{\eta^{\flat}}\,\!^{\{0,1\}}\sqrt[4]{\eta^{\sharp}}\,\!^{\{0,1\}}\]
\[\cS(4)_{\ang{\frac18^*,\alpha_2,\chi_{16},\ell_{32}}}=\cM(4)_{\ang{\frac18^*,\alpha_2,\chi_{16},\ell_{32}}}\sqrt[4]{\eta}^3\]
\begin{Lem}
\[\cM(32)_{\ang{\frac18^*,\chi_{16}}}=\cM(4)_{\ang{\frac18^*,\chi_{16},\ell_{32}}}\]
\end{Lem}
\begin{proof}Lemma \ref{b8c8} and Lemma \ref{decom} lead to $\cM(32)_{\frac k8^*+\ang{\chi_{16}}}=\cM(4)_{\frac k8^*+\ang{\chi_{16},\ell_{32}}}$ for $k\in8\N$. This assertion stays hold for $k\in8\M+5$ since
\begin{align*}\cM(32)_{\frac k8^*+\ang{\chi_{16}}}\sqrt[4]{\eta}^3&\subset\cS(32)_{\frac{k+3}8^*+\ang{\chi_{16}}}\\
&=\cS(4)_{\frac{k+3}8^*+\ang{\chi_{16},\ell_{32}}}\\
&=\cM(4)_{\frac k8^*+\ang{\chi_{16},\ell_{32}}}\sqrt{\eta}^3\end{align*}
It stays hold also for $k\in8\M+2$ and then for $k\in8\M+7$, ....
\end{proof}

\vspace{0.5cm}

Next, to study finelier put
\[\z W_4=((\frac14^*+\frac\ia2,\chi_{16}),(\frac14^*+\frac\iA2,\chi_{16}))\]
\begin{Lem}\label{K41}
\[\cM(4|\chi_8)_{\ang{W_4}}=\bigoplus\C\Big[\sqrt{\eta^{\natural\ang{\frac12}}},\sqrt{\eta^{\natural\ang2}}\Big]\eta^{\natural\{0,1\}}\]
\[\cS(4|\chi_8)_{\ang{W_4}}=\cM(4|\chi_8)_{\ang{W_4}}\eta^{\natural}\sqrt{\eta^{\natural\ang{\frac12}}\eta^{\natural\ang2}}\eta^{\flat\ang{\frac12}}\eta^{\sharp\ang2}\]
\end{Lem}
\begin{proof}The asserion easily follows from another representation
\[\z \cM(4|\chi_8)_{\ang{w_4,\alpha_2}}=\bigoplus\C\Big[\eta^{\natural\ang{\frac12}},\sqrt{\eta^{\natural\ang2}}\Big]\sqrt{[\eta,\frac12]}^{\{0,1\}}\sqrt{[\eta,\frac32]}^{\{0,1\}}\]

\end{proof}

\newpage

Let the sequence
\[\z \alpha_2'=((\frac\ia2,\chi_{16}),(\frac\iA2,\chi_{16}))\]

We see some expansions
\[\cM(4|\chi_8)_{\ang{W_4,\ell_8}}=\bigoplus\cM(4|\chi_8)_{\ang{W_4}}\sqrt{\eta^{\flat\ang{\frac12}}}^{\{0,1\}}\sqrt{\eta^{\sharp\ang2}}^{\{0,1\}}\]
\[\cS(4|\chi_8)_{\ang{W_4,\ell_8}}=\cM(4|\chi_8)_{\ang{W_4,\ell_8}}\eta^3\]
and
\[\cM(4|\chi_8)_{\ang{\frac14^*,\alpha_2',\ell_{16}}}=\bigoplus\cM(4|\chi_8)_{\ang{W_4,\ell_8}}\sqrt{\eta^{\natural}}\,\!^{\{0,1\}}\sqrt{\eta^{\flat}}\,\!^{\{0,1\}}\sqrt{\eta^{\sharp}}\,\!^{\{0,1\}}\]
\[\cS(4|\chi_8)_{\ang{\frac14^*,\alpha_2',\ell_{16}}}=\cM(4|\chi_8)_{\ang{\frac14^*,\alpha_2',\ell_{16}}}\sqrt{\eta}^3\]
and
\[\cM(4|\chi_8)_{\ang{\frac18^*,\alpha_2',\ell_{32}}}=\bigoplus\cM(4|\chi_8)_{\ang{\frac14^*,\alpha_2',\ell_{16}}}\sqrt[4]{\eta^{\natural}}\,\!^{\{0,1\}}\sqrt[4]{\eta^{\flat}}\,\!^{\{0,1\}}\sqrt[4]{\eta^{\sharp}}\,\!^{\{0,1\}}\]

Its Hilbert function is
\[\frac{(1+t^{\frac12})(1+t^{\frac14+\frac\ia2})(1+t^{\frac14+\frac\iA2})(1+t^{\frac14})^3h_{42}}{(1-t^{\frac14+\frac\ia2})(1-t^{\frac14+\frac\iA2})}=\frac{h_{41}h_{42}(1+t^{\frac14})(1+t^{\frac12})}{(1-t^{\frac14})^2}\]
where
\begin{align*}h_{42}&=(1+t^{\frac18+\frac{\ia+\iA}2})(1+t^{\frac18+\frac\ia2})(1+t^{\frac18+\frac\iA2})\\
&=1+t^{\frac38}+(t^{\frac18}+t^{\frac14})(t^{\frac\ia2}+t^{\frac\iA2}+t^{\frac{\ia+\iA}2})\end{align*}

We compute a new dimension formula
\[\VS{k\in\M}\dim\cM(32)_{\frac k8^*+\ang{\chi_8}}\cdot t^{\frac k4}\]
\[=\frac{(1-2t^{\frac18}+3t^{\frac14}+t^{\frac38}+t^{\frac12}+t^{\frac58}+t^{\frac34}+3t^{\frac78}-2t+t^{\frac98})(1+t^{\frac14})(1+t^{\frac12})}{(1-t^{\frac18})^2}\]\\
\[\VS{k\in\M}\dim\cM(32)_{\frac k8^*+\ang{\chi_{16}}}\cdot t^{\frac k4}=\VS{k\in\M}\dim\cM(32)_{\frac k8^*+\ang{\chi_8}}\cdot t^{\frac k4}\,+\]
\[\frac{(t^{\frac18}-t^{\frac14}+5t^{\frac38}-t^{\frac12}-t^{\frac58}+5t^{\frac34}-t^{\frac78}+t)(1+t^{\frac14})\cdot2t^{\frac14}}{(1-t^{\frac18})^2}\]

\newpage

\subsection{Formal weights b,B}

Write $\ib=\ang{\frac14}$ and $\iB=\ang4$.

Weight $\frac{\ib}2$ and $\frac{\iB}2$ operators are expanded by
\[\sqrt{\eta^{\ang{\frac14}3}}\in\cS(2,4)_{(\frac34^*+\frac{\ib}2,c_4b_{64})}\6\sqrt{\eta^{\ang43}}\in\cS(4,2)_{(\frac34^*+\frac{\iB}2,c_{64}b_4)}\]
then
\[\sqrt{\eta^{\ang{\frac14}}}\in\cS(6,12)_{(\frac14^*+\frac\ib2,c_{12}b_{192})}\6\sqrt{\eta^{\ang4}}\in\cS(12,6)_{(\frac14^*+\frac\iB2,c_{192}b_{12})}\]

\begin{Lem}
\[\sqrt{\eta^{\natural\ang{\frac14}}}\in\cM(2,8)_{\frac14^*+\frac\ib2} \6 \sqrt{\eta^{\natural\ang4}}\in\cM(8,2)_{\frac14^*+\frac\iB2}\]
\[\sqrt{\eta^{\flat\ang{\frac14}}}\in\cM(2,8)_{(\frac14^*+\frac\ib2,\chi_8c_4)} \6 \sqrt{\eta^{\sharp\ang4}}\in\cM(8,2)_{(\frac14^*+\frac\iB2,\chi_8b_4)}\]
\end{Lem}
\begin{proof}
$\sqrt{\eta^{\flat\ang{\frac14}}}=\dfrac{\sqrt{\eta^{\ang{\frac14}3}}}{\eta^{\sharp\ang{\frac18}}}$.
\end{proof}

\vspace{0.5cm}

The weight $\frac{\ia}4$ and $\frac{\iA}4$ operators are exanpded by
\[\sqrt[4]{\eta^{\ang{\frac12}3}}\in\cS(2)_{(\frac38^*-\frac{\ia}4,{\chi_{32}}^3c_{16}b_{64})}\6\sqrt[4]{\eta^{\ang23}}\in\cS(2)_{(\frac38^*-\frac{\iA}4,\overline{\chi_{32}}^3c_{64}b_{16})}\]

\begin{Lem}
\[\sqrt[8]{\eta^{\flat}}\in\cM(2)_{(\frac1{16}^*+\frac\ia4,\chi_{16}c_{64})}\6\sqrt[8]{\eta^{\sharp}}\in\cM(2)_{(\frac1{16}^*+\frac\iA4,\overline{\chi_{16}}b_{64})}\]
\[\sqrt[8]{\eta^{\natural}}\in\cM(2)_{\frac1{16}^*-\frac{\ia+\iA}4}\]
\end{Lem}

\begin{Lem}
\[\sqrt[4]{\eta^{\flat\ang{\frac12}}}\in\cM(2,4)_{(\frac18^*-\frac\ia4+\frac\ib2,\overline{\chi_{32}}c_{16})}\6\sqrt[4]{\eta^{\sharp\ang2}}\in\cM(4,2)_{(\frac18^*-\frac\iA4+\frac\iB2,\chi_{32}b_{16})}\]
\[\sqrt[4]{\eta^{\natural\ang{\frac12}}}\in\cM(2,4)_{(\frac18^*+\frac\ia4+\frac\ib2,\overline{\chi_{32}}^3)}\6\sqrt[4]{\eta^{\natural\ang2}}\in\cM(4,2)_{(\frac18^*+\frac\iA4+\frac\iB2,\chi_{32}^3)}\]
\end{Lem}
\begin{proof}First $\sqrt{\eta^{\sharp\ang{\frac14}}}=\dfrac{\sqrt{\eta^{\ang{\frac14}3}}}{\eta^{\flat\ang{\frac12}}}\in\cM(2,4)_{(\frac14^*+\frac\ib2,\chi_8b_{64})}$ and\\
$\sqrt[4]{\eta^{\flat\ang{\frac12}}}=\dfrac{\sqrt[4]{\eta^{\ang{\frac12}3}}}{\sqrt{\eta^{\sharp\ang{\frac14}}}}\in\cM(2,4)_{(\frac18^*+\frac\ia4+\frac\ib2,\overline{\chi_{32}}^3)}$
\end{proof}

\newpage

\begin{Lem}
\[\z \sqrt{[\eta,\frac32]^{\ang{\frac12}}}\in\cM(4,8)_{\frac14^*+\frac{\ia+\ib}2} \6 \sqrt{[\eta,\frac12]^{\ang2}}\in\cM(8,4)_{\frac14^*+\frac{\iA+\iB}2}\]
\[\z \sqrt{[\eta,\frac12]^{\ang{\frac12}}}\in\cM(4,8)_{(\frac14^*+\frac{\ia+\ib}2,\chi_8)} \6 \sqrt{[\eta,\frac32]^{\ang2}}\in\cM(8,4)_{(\frac14^*+\frac{\iA+\iB}2,\chi_8)}\]
\end{Lem}
\begin{proof}The first assertion follows from the identity in the proof of Proposition \ref{aA}.\end{proof}

\begin{Prop}\label{Ab} We have an identity in $\cM(8)_{(\frac12^*+\frac{\iA+\ib}2,\overline{\chi_{16}})}$
\[\z\sqrt{[\eta,\frac32]^{\ang{\frac12}}[\eta,\frac12]}=\sqrt{\eta^{\natural\ang{\frac14}}\eta^{\natural\ang2}}-(1+i)\sqrt{\eta^{\flat\ang{\frac14}}\eta^{\sharp\ang2}}\]
\end{Prop}
\begin{proof}By acting $(\BM1&4\\0&1\\\EM)$ on
\begin{align*}\z \sqrt{\eta^{\natural\ang{\frac14}}}\sqrt{[\eta,\frac12]^{\ang{\frac12}}[\eta,\frac12]}&\z =\big(\sqrt{[\eta,\frac12]\eta^{\natural\ang2}}+(1+i)\sqrt{[\eta,\frac32]\eta^{\sharp\ang2}}\big)\sqrt{[\eta,\frac12]}\\
&=\big(\eta^{\flat\ang{\frac12}}+2(1+i)\eta^{\sharp\ang2}\big)\sqrt{\eta^{\natural\ang2}}+(1+i)\big(\eta^{\natural{\ang{\frac14}}}-2\eta^{\sharp}\big)\sqrt{\eta^{\sharp\ang2}}\\
&\z =\sqrt{\eta^{\natural\ang{\frac14}}}\big(\sqrt{\eta^{\flat\ang{\frac14}}\eta^{\natural\ang2}}+(1+i)\sqrt{\eta^{\natural\ang{\frac14}}\eta^{\sharp\ang2}}\big)\end{align*}
\end{proof}

\vspace{0.5cm}

For $\psi:\Gamma_0(N)\cap\Gamma^0(N)\to\C^\times$ write $\cM(N|\psi)=\cM(\ker(\psi))$.

It is important that
\[\z [u,0]^{\ang{\frac12}}\in\cM(4|\chi_8)_{\frac14^*+\frac\ib2}\]
\[\z [o,0]^{\ang{\frac12}}\in\cM(4|\chi_8)_{\frac14^*+\frac{\ia+\ib}2}\]\\

Make formal symbols $\downarrow$, $\uparrow$ and let $\frac{\downarrow}2=\frac{\uparrow}2=\frac{\ia+\iA}2$.

Then, make operators of weight $\frac\downarrow4$ and $\frac\uparrow4$ by
\[\z \sqrt{[o,\frac74]}\in\cM(4|\chi_8)_{(\frac18^*+\frac{\downarrow}4,\chi_{32})} \6 \sqrt{[o,\frac12]}\in\cM(4|\chi_8)_{(\frac18^*+\frac{\uparrow}4,\overline{\chi_{32}})}\]
then Lemma \ref{oud} leads to
\[\z \sqrt{[o,\frac34]}\in\cM(4|\chi_8)_{(\frac18^*+\frac\iA2+\frac\downarrow4,\chi_{32})} \6 \sqrt{[o,\frac32]}\in\cM(4|\chi_8)_{(\frac18^*+\frac\ia2+\frac\uparrow4,\overline{\chi_{32}})}\]
In addition, lemma \ref{ud} leads to
\[\z \sqrt{[o,\frac14]}\in\cM(4|\chi_8)_{(\frac18^*+\frac{\ia+\iB}2+\frac\downarrow4,\overline{\chi_{32}})} \6 \sqrt{[o,0]}\in\cM(4|\chi_8)_{(\frac18^*+\frac{\iA+\ib}2+\frac\uparrow4,\chi_{32})}\]
\[\z \sqrt{[o,\frac54]}\in\cM(4|\chi_8)_{(\frac18^*+\frac\iB2-\frac{\downarrow}4,\overline{\chi_{32}}c_8)} \6 \sqrt{[o,1]}\in\cM(4|\chi_8)_{(\frac18^*+\frac\ib2-\frac{\uparrow}4,\chi_{32}b_8)}\]\\

\newpage

\subsection{s-form}

Put
\[\z[s,0]=\sqrt{\eta^{\natural\ang{\frac12}}\eta^{\natural{\ang2}}}+(1^{\frac1{16}}+1^{\frac{15}{16}})^2\sqrt{\eta^{\flat\ang{\frac12}}\eta^{\sharp\ang2}}\]
\[\z[-s,0]=[s,0]^{\sigma_8}=\sqrt{\eta^{\natural\ang{\frac12}}\eta^{\natural{\ang2}}}-(1^{\frac1{16}}-1^{\frac{15}{16}})^2\sqrt{\eta^{\flat\ang{\frac12}}\eta^{\sharp\ang2}}\]
and
\[\z [s,\frac k4]=[s,0]\big|(\BM 1&k\\0&1\EM)\]
\[\z [-s,\frac k4]=[-s,0]\big|(\BM 1&4k\\0&1\EM)\]
then $[s,\frac k4+2]=[s,\frac k4]$ and $[-s,\frac k4+2]=[-s,\frac k4]$. For example
\[\z[s,1]=\sqrt{\eta^{\natural\ang{\frac12}}\eta^{\natural{\ang2}}}-(1^{\frac1{16}}+1^{\frac{15}{16}})^2\sqrt{\eta^{\flat\ang{\frac12}}\eta^{\sharp\ang2}}\]
\[\z[-s,1]=\sqrt{\eta^{\natural\ang{\frac12}}\eta^{\natural{\ang2}}}+(1^{\frac1{16}}-1^{\frac{15}{16}})^2\sqrt{\eta^{\flat\ang{\frac12}}\eta^{\sharp\ang2}}\]
\[\z [s,\frac12]=\sqrt{\eta^{\flat\ang{\frac12}}\eta^{\natural{\ang2}}}+(1^{\frac1{16}}+1^{\frac3{16}})^2\sqrt{\eta^{\natural\ang{\frac12}}\eta^{\sharp\ang2}}\]
\[\z [s,\frac14]=\sqrt{[\eta,\frac12]\eta^{\natural\ang2}}+(1+1^{\frac18})^2\sqrt{[\eta,\frac32]\eta^{\sharp\ang2}}\]\\

\begin{Lem}\label{uds}
\[\z \sqrt{[d,0][u,0]}=\frac12\Big(\sqrt{[s,0]}+\sqrt{[-s,1]}\Big)\]
\[\z \sqrt{[d,1][u,1]}=\frac12\Big(\sqrt{[s,0]}-\sqrt{[-s,1]}\Big)\]
\[\z \sqrt{[d,\frac32][u,\frac32]}=\frac{1-1^{\frac18}}2\sqrt{[s,0]}+\frac{1+1^{\frac18}}2\sqrt{[-s,1]}\]
and
\[\z \sqrt{[o,\frac14][u,\frac74]}=(1-\frac1{\sqrt2})\sqrt{[s,\frac14]^{\ang2}}+\frac1{\sqrt2}\sqrt{[-s,\frac54]^{\ang2}}\]
\[\z \sqrt{[o,\frac34][u,\frac14]}=\frac1{1+i}\sqrt{[s,\frac14]^{\ang2}}+\frac1{1-i}\sqrt{[-s,\frac54]^{\ang2}}\]
\[\z \sqrt{[o,\frac74][u,\frac54]}=\frac1{1-i}\sqrt{[s,\frac14]^{\ang2}}+\frac1{1+i}\sqrt{[-s,\frac54]^{\ang2}}\]
\[\z \sqrt{[o,\frac54][u,\frac34]}=-\frac1{\sqrt2}\sqrt{[s,\frac14]^{\ang2}}+(1+\frac1{\sqrt2})\sqrt{[-s,\frac54]^{\ang2}}\]
\end{Lem}

\begin{Prop}\label{now}We have an identity in $\cM(8)_{\frac12^*}$
\[\z\sqrt{[s,0][-s,0]}=\eta^{\flat}+2\eta^{\sharp}\]
and in $\cM(8)_{(\frac12^*,\chi_8)}$
\[\z\sqrt{[s,0][s,1]}=\eta^{\flat\ang{\frac12}}-2\sqrt2\eta^{\sharp\ang2}\]
\end{Prop}
\begin{proof}
LHS of the former identity is
\begin{align*}&\big(\sqrt{[d,0][u,0]}+\sqrt{[d,1][u,1]}\big)\big(\sqrt{[d,1][u,0]}+\sqrt{[d,0][u,1]}\big)\\
&=\sqrt{[d,0][d,1]}\big([u,0]+[u,1]\big)+\big([d,0]+[d,1]\big)\sqrt{[u,0][u,1]}\end{align*}

LHS of the latter identity is
\begin{align*}&\big(\sqrt{[d,0][u,0]}+\sqrt{[d,1][u,1]}\big)\big(\sqrt{[d,1][u,0]}-\sqrt{[d,0][u,1]}\big)\\
&=\sqrt{[d,0][d,1]}\big([u,0]-[u,1]\big)-\big([d,0]-[d,1]\big)\sqrt{[u,0][u,1]}\end{align*}
\end{proof}
Acting $\sigma_8$ on the latter identity yields
\[\z\sqrt{[-s,0][-s,1]}=\eta^{\flat\ang{\frac12}}+2\sqrt2\eta^{\sharp\ang2}\]

\begin{Prop}We have identities in $\cM(4)_{\frac12^*+\frac{\ia+\iA}2}$ (cf. Proposition \ref{aA})
\[\z\sqrt{[s,\frac12][-s,\frac32]}=\sqrt{\eta^{\natural\ang{\frac12}}\eta^{\natural{\ang2}}}+\sqrt2i\sqrt{\eta^{\flat\ang{\frac12}}\eta^{\sharp\ang2}}\]
\end{Prop}
\begin{proof}
Act $(\BM 1&2\\0&1\EM)$ on
\begin{align*}\sqrt{[s,0][-s,1]}&=\big(\sqrt{[d,0][u,0]}+\sqrt{[d,1][u,1]}\big)\big(\sqrt{[d,0][u,0]}-\sqrt{[d,1][u,1]}\big)\\
&=[d,0][u,0]-[d,1][u,1]\\
&=\sqrt{\eta^{\flat\ang{\frac12}}\eta^{\natural{\ang2}}}+\sqrt2\sqrt{\eta^{\natural\ang{\frac12}}\eta^{\sharp\ang2}}\end{align*}
\end{proof}

\begin{Prop}\label{bB}We have an identity in $\cM(8)_{\frac12^*+\frac{\ib+\iB}2}$
\[\z\sqrt[4]{[\eta,\frac34][\eta,\frac54]}\sqrt{[s,0]}=\sqrt{\eta^{\natural\ang{\frac14}}\eta^{\natural\ang4}}-\sqrt2\sqrt{\eta^{\flat\ang{\frac14}}\eta^{\sharp\ang4}}\]
and in $\cM(8)_{\frac12^*+\frac{\ia+\iA+\ib+\iB}2}$
\[\z(1-1^{\frac18})\sqrt[4]{[\eta,\frac74][\eta,\frac14]}\sqrt{[s,0]}=\sqrt{[\eta,\frac32]^{\ang{\frac12}}[\eta,\frac12]^{\ang2}}-1^{\frac18}\sqrt{[\eta,\frac12]^{\ang{\frac12}}[\eta,\frac32]^{\ang2}}\]
\end{Prop}
\begin{proof}LHS of the former identity is
\[\z\sqrt{\eta^{\natural2}+2\sqrt{[\eta,\frac34][\eta,\frac54]}\sqrt{\eta^{\flat\ang{\frac12}}\eta^{\sharp\ang2}}}=\sqrt{\eta^{\natural2}+2\big(\eta^{\flat}\eta^{\sharp}-\sqrt2\eta^{\flat\ang{\frac12}}\eta^{\sharp\ang2}\big)}\]
\end{proof}

\begin{Prop}We have identities in $\cM(8)_{(\frac12^*+\frac{\iA+\ib}2,\overline{\chi_{16}})}$ (cf. Proposition \ref{Ab})
\[\z\sqrt{[s,\frac74][-s,\frac54]}=\sqrt{\eta^{\natural\ang{\frac14}}\eta^{\natural\ang2}}-\sqrt2i\sqrt{\eta^{\flat\ang{\frac14}}\eta^{\sharp\ang2}}\]
\[\z\sqrt{[s,\frac34][s,\frac54]}=\sqrt{\eta^{\natural\ang{\frac14}}\eta^{\natural\ang2}}-(1^{\frac1{16}}+1^{\frac{15}{16}})^2\sqrt{\eta^{\flat\ang{\frac14}}\eta^{\sharp\ang2}}\]
\end{Prop}
\begin{proof}

\end{proof}

Acting $\sigma_8$ on the last identity yields
\[\z\sqrt{[-s,\frac74][-s,\frac14]}=\sqrt{\eta^{\natural\ang{\frac14}}\eta^{\natural\ang2}}+(1^{\frac1{16}}-1^{\frac{15}{16}})^2\sqrt{\eta^{\flat\ang{\frac14}}\eta^{\sharp\ang2}}\]\\

Note
\[\iE_{\chi_4\chi_{32}}=1+(i-1^{\frac38})\VS{n\in\N}(\chi_4\chi_{32}*{\tt 1}_\N)(n)q^n\]
The dimension formula states $\dim\cM(4,8)_{(2,\chi_{16})}=6$ and we see
\[\z \iE_{\chi_4\chi_{32}}^{\ang{\frac18}}=\sqrt[4]{\eta^{\natural\ang{\frac14}}\eta^{\natural\ang{\frac12}}\eta^{\natural}\eta^{\natural\ang2}[\eta,\frac12]^{\ang{\frac12}}[\eta,\frac12]}\sqrt{[-s,\frac74]}\]\\

\newpage

\subsection{Structure theorems with $\chi_{32}$}

Let the sequence
\[\z W_8=\ang{(\frac18^*+\frac{\downarrow}4,\chi_{32}),(\frac18^*+\frac{\uparrow}4,\overline{\chi_{32}})}\]

\begin{Lem}
\[\cM(4|\chi_8)_{\ang{W_8}}=\bigoplus \z \C\Big[\sqrt{[o,\frac74]},\sqrt{[o,\frac12]}\Big]\]
\[\z \cS(4|\chi_8)_{\ang{W_8}}=\cM(4|\chi_8)_{\ang{W_8}}\eta^{\natural}\sqrt{[o,\frac74][o,\frac12]}[o,\frac34][o,\frac32]\eta^{\flat\ang{\frac12}}\eta^{\sharp\ang2}\]
\end{Lem}
\begin{proof}Put $w_0=(\frac{\ia+\iA}2+\frac{\downarrow+\uparrow}4,\chi_{16})$ and start with
\[\z\dim\cS(4|\chi_8)_{\ang{w_4}}|_{\frac k4}=(k-9)^+.\]
by the proof of Lemma \ref{K4}. Since $\sqrt{[o,\frac74][o,\frac12]}\in\cM(4|\chi_8)_{w_4+w_0}$ we see
\[\dim\cS(4|\chi_8)_{\ang{w_4}+w_0}|_{\frac k4}\leq(k-8)^+\]
and
\[\dim\cS(4|\chi_8)_{\ang{W_8}}|_{\frac k8}\leq(k-9)^++(k-8)^+.\]
\end{proof}

Let the sequence
\[\z \beta_2=(\frac\ib2,\frac\iB2)\]

We see some expansions
\[\cM(4|\chi_8)_{\ang{W_8,\alpha_2}}=\bigoplus \z \cM(4|\chi_8)_{\ang{W_8}}\sqrt{[o,\frac34]}^{\{0,1\}}\sqrt{[o,\frac32]}^{\{0,1\}}\]
\[\cS(4|\chi_8)_{\ang{W_8,\alpha_2}}=\cM(4|\chi_8)_{\ang{W_8,\alpha_2}}\eta^{\natural}\sqrt{\eta^{\natural\ang{\frac12}}\eta^{\natural\ang2}}\eta^{\flat\ang{\frac12}}\eta^{\sharp\ang2}\]
and
\[\cM(4|\chi_8)_{\ang{W_8,\alpha_2,\beta_2}}=\bigoplus \z \cM(4|\chi_8)_{\ang{W_8,\alpha_2}}\sqrt{[o,\frac14]}^{\{0,1\}}\sqrt{[o,0]}^{\{0,1\}}\]
\[\z \cS(4|\chi_8)_{\ang{W_8,\alpha_2,\beta_2}}=\cM(4|\chi_8)_{\ang{W_8,\alpha_2,\beta_2}}\eta^{\natural}\sqrt{\eta^{\natural\ang{\frac12}}\eta^{\natural\ang2}[o,\frac14][o,0]}[o,\frac54][o,1]\]
and
\[\cM(4|\chi_8)_{\ang{W_8,\alpha_2,\beta_2,\ell_8}}=\bigoplus \z \cM(4|\chi_8)_{\ang{W_8,\alpha_2,\beta_2}}\sqrt{[o,\frac54]}^{\{0,1\}}\sqrt{[o,1]}^{\{0,1\}}\]
\[\cS(4|\chi_8)_{\ang{W_8,\alpha_2,\beta_2,\ell_8}}=\cM(4|\chi_8)_{\ang{W_8,\alpha_2,\beta_2,\ell_8}}\eta^3\]

\newpage

\begin{Lem}
\[\cM(4|\chi_8)_{\ang{w_4,\alpha_2,\beta_2}}=\bigoplus\cM(4|\chi_8)_{\ang{w_4}}F_4G_4\]
where
\[\z F_4=\big\{1,\sqrt{\eta^{\natural\ang{\frac12}}},[d,0]^{\ang2},[o,0]^{\ang2}\big\}\]
\[\z G_4=\big\{1,\sqrt{\eta^{\natural\ang2}},[u,0]^{\ang{\frac12}},[o,0]^{\ang{\frac12}}\big\}\]
The ideal $\cS(4|\chi_8)_{\ang{w_4,\alpha_2,\beta_2}}$ is not principal.
\end{Lem}

\begin{proof}Put
The Hibert function of $\cM(4|\chi_8)_{\ang{W_8,\alpha_2,\beta_2}}$ is
\[\frac{(1+t^{\frac18+\frac\iA2+\frac{\downarrow}4})(1+t^{\frac18+\frac\ia2+\frac{\uparrow}4})(1+t^{\frac18+\frac{\ia+\iB}2+\frac{\downarrow}4})(1+t^{\frac18+\frac{\iA+\ib}2+\frac{\uparrow}4})}{(1-t^{\frac18+\frac{\downarrow}4})(1-t^{\frac18+\frac{\uparrow}4})}\]
and its $\ang{\frac18,\alpha_2,\beta_2}$-part is
\[\frac{\big(1+(t^{\frac\ia2}+t^{\frac\iB2}+t^{\frac{\iA+\iB}2})t^{\frac14}\big)\big(1+(t^{\frac\iA2}+t^{\frac\ib2}+t^{\frac{\ia+\ib}2})t^{\frac14}\big)}{(1-t^{\frac14+\frac{\ia+\iA}2})^2}\]

To prove the former assertion, it is enoght to show the natural map ${\rm RHS}\to{\rm LHS}$ is injection. Decompose
\[{\rm RHS}={\rm RHS}|_{\ang{w_4,\alpha_2}}\oplus{\rm RHS}|_{\ang{w_4}+\frac\ib2}\oplus{\rm RHS}|_{\ang{w_4}+\frac\iB2}\oplus{\rm RHS}|_{\ang{w_4}+\frac{\ib+\iB}2}\oplus{\rm RHS}|_{\ang{w_4}+\frac{\ia+\ib}2}\oplus\cdots\]

Then, for example
\begin{align*}{\rm RHS}|_{\ang{w_4}+\frac\ib2}&=\bigoplus\cM(4|\chi_8)_{\ang{w_4}}\Big\{[u,0]^{\ang{\frac12}},\sqrt{\eta^{\natural\ang{\frac12}}}[o,0]^{\ang{\frac12}}\Big\}\\
&=\bigoplus\C\big[[o,0],[o,1]\big]\Big\{[u,0]^{\ang{\frac12}},x\Big\}\end{align*}
where $x=\sqrt{\eta^{\natural\ang{\frac12}}}[o,0]^{\ang{\frac12}}-\eta^{\natural}$. The natural map ${\rm RHS}|_{\ang{w_4}+\frac\ib2}\to{\rm LHS}$ is injective since $[o,0]=1+\cdots$, $[o,1]=q^{\frac14}+\cdots$, $[u,0]^{\ang{\frac12}}=1+\cdots$, $x=q^{\frac18}+\cdots$

\end{proof}

\begin{Lem}
\[\cM(4|\chi_8)_{\ang{w_4,\alpha_2,\beta_2,\ell_8}}=\bigoplus\cM(4|\chi_8)_{\ang{w_4}}F_8G_8\]
\[\cS(4|\chi_8)_{\ang{w_4,\alpha_2,\beta_2,\ell_8}}=\cM(4|\chi_8)_{\ang{w_4,\alpha_2,\beta_2,\ell_8}}\eta^3\]
where
\[\z F_8=F_4\cup\big\{\sqrt{\eta^{\flat\ang{\frac12}}},\eta^{\flat},[d,1]^{\ang2},[o,1]^{\ang2}\big\}\]
\[\z G_8=G_4\cup\big\{\sqrt{\eta^{\sharp\ang2}},\eta^{\sharp},[u,1]^{\ang{\frac12}},[o,1]^{\ang{\frac12}}\big\}\]
\end{Lem}

\begin{proof}
The Hibert function of $\cM(4|\chi_8)_{\ang{W_8,\alpha_2,\beta_2,\ell_8}}$ is
\[\frac{(1+t^{\frac18+\frac\iA2+\frac{\downarrow}4})(1+t^{\frac18+\frac\ia2+\frac{\uparrow}4})(1+t^{\frac18+\frac{\ia+\iB}2+\frac{\downarrow}4})(1+t^{\frac18+\frac{\iA+\ib}2+\frac{\uparrow}4})(1+t^{\frac18+\frac\iB2-\frac{\downarrow}4})(1+t^{\frac18+\frac\ib2-\frac{\uparrow}4})}{(1-t^{\frac18+\frac{\downarrow}4})(1-t^{\frac18+\frac{\uparrow}4})}\]
and its $\ang{\frac18,\alpha_2,\beta_2}$-part is
\[\frac{\big(1+2(t^{\frac\ia2}+t^{\frac\iB2}+t^{\frac{\iA+\iB}2})t^{\frac14}+t^{\frac12}\big)\big(1+2(t^{\frac\iA2}+t^{\frac\ib2}+t^{\frac{\ia+\ib}2})t^{\frac14}+t^{\frac12}\big)}{(1-t^{\frac14+\frac{\ia+\iA}2})^2}\]

To prove the former assertion, it is enoght to show the natural map ${\rm RHS}\to{\rm LHS}$ is injection. For example
\[{\rm RHS}|_{\ang{w_4}+(\frac\ib2,c_8)}=\bigoplus\cM(4|\chi_8)_{\ang{w_4}}\Big\{\eta^{\flat}[u,0]^{\ang{\frac12}},\sqrt{\eta^{\flat\ang{\frac12}}}[o,0]^{\ang{\frac12}}\Big\}\]
\end{proof}

Let the sequence
\[\z \alpha_4'=((\frac\ia4,\chi_{32}),(\frac\iA4,\chi_{32}))\]

We see some more expansions
\[\cM(4|\chi_8)_{\ang{\frac18^*,\alpha_2,\beta_2,\chi_{16},\ell_8}}=\bigoplus \z \cM(4|\chi_8)_{\ang{w_4,\alpha_2,\beta_2,\ell_8}}\sqrt[4]{\eta^{\natural}}^{\{0,1,2,3\}}\]
\[\z \cS(4|\chi_8)_{\ang{\frac14^*,\alpha_2,\beta_2,\chi_{16},\ell_8}}=\cM(4|\chi_8)_{\ang{\frac14^*,\alpha_2,\beta_2,\chi_{16},\ell_8}}\sqrt[4]{\eta^{\natural}}\eta^{\flat}\eta^{\sharp}\]
and
\[\cM(4|\chi_8)_{\ang{\frac18^*,\alpha_4',\beta_2,\chi_{16},\ell_8}}=\bigoplus \z \cM(4|\chi_8)_{\ang{\frac18^*,\alpha_2,\beta_2,\chi_{16},\ell_8}}\sqrt[4]{\eta^{\natural\ang{\frac12}}}^{\{0,1\}}\sqrt[4]{\eta^{\natural\ang2}}^{\{0,1\}}\]
\[\z \cS(4|\chi_8)_{\ang{\frac18^*,\alpha_4',\beta_2,\chi_{16},\ell_8}}=\cM(4|\chi_8)_{\ang{\frac18^*,\alpha_4',\beta_2,\chi_{16},\ell_8}}\sqrt[4]{\eta^{\natural}\eta^{\natural\ang{\frac12}}\eta^{\natural\ang2}}\sqrt{\eta^{\flat\ang{\frac12}}\eta^{\sharp\ang2}}\]
and
\[\cM(4|\chi_8)_{\ang{\frac18^*,\alpha_4',\beta_2,\chi_{16},\ell_{16}}}=\bigoplus\cM(4|\chi_8)_{\ang{w_8,\alpha_4',\beta_2,\chi_{16},\ell_8}}\sqrt[4]{\eta^{\flat\ang{\frac12}}}^{\{0,1\}}\sqrt[4]{\eta^{\sharp\ang2}}^{\{0,1\}}\]
\[\cS(4|\chi_8)_{\ang{\frac18^*,\alpha_4',\beta_2,\chi_{16},\ell_{16}}}=\cM(4|\chi_8)_{\ang{\frac18^*,\alpha_4',\beta_2,\chi_{16},\ell_{16}}}\sqrt[4]{\eta^{\natural}}\sqrt{\eta^{\flat}\eta^{\sharp}}\]
and
\[\cM(4|\chi_8)_{\ang{\frac18^*,\alpha_4',\beta_2,\chi_{16},\ell_{32}}}=\bigoplus\cM(4|\chi_8)_{\ang{\frac18^*,\alpha_4',\beta_2,\chi_{16},\ell_{16}}}\sqrt[4]{\eta^{\natural}}\,\!^{\{0,1\}}\sqrt[4]{\eta^{\flat}}\,\!^{\{0,1\}}\sqrt[8]{\eta^{\sharp}}\,\!^{\{0,1\}}\]
\[\cS(4|\chi_8)_{\ang{\frac18^*,\alpha_4',\beta_2,\chi_{16},\ell_{32}}}=\cM(4|\chi_8)_{\ang{\frac18^*,\alpha_4',\beta_2,\chi_{16},\ell_{32}}}\sqrt[4]{\eta^3}\]
and
\[\cM(4|\chi_8)_{\ang{\frac1{16}^*,\alpha_4',\beta_2,\chi_{16},\ell_{64}}}=\bigoplus\cM(4|\chi_8)_{\ang{\frac18^*,\alpha_4',\beta_2,\chi_{16},\ell_{32}}}\sqrt[8]{\eta^{\natural}}\,\!^{\{0,1\}}\sqrt[8]{\eta^{\flat}}\,\!^{\{0,1\}}\sqrt[8]{\eta^{\sharp}}\,\!^{\{0,1\}}\]
\[\cS(4|\chi_8)]_{\ang{\frac1{16}^*,\alpha_4',\beta_2,\chi_{16},\ell_{64}}}=\cM(4|\chi_8)_{\ang{\frac1{16}^*,\alpha_4',\beta_2,\chi_{16},\ell_{64}}}\sqrt[8]{\eta^3}\]\\

\begin{Lem}
\[\cM(64)_{\ang{\frac1{16}^*,\chi_{16}}}=\cM(4|\chi_8)_{\ang{\frac1{16}^*,\chi_{16},\ell_{64}}}\]
\end{Lem}
\begin{proof}First note
\[c_{64},b_{64}:\Gamma_0(4)\cap\Gamma^0(4)\cap\ker(\chi_8)\to\C^\times\]
are homomorphisms and $\cM(64)_{\frac k{16}^*+\ang{\chi_{16}}}=\cM(4|\chi_8)_{\frac k{16}^*+\ang{\chi_{16},\ell_{64}}}$ for $k\in16\N$.

This assertion stays hold for $k\in16\M+13$ since
\begin{align*}\cM(64)_{\frac k{16}^*+\ang{\chi_{16}}}\sqrt[8]{\eta}^3&\subset\cS(64)_{\frac{k+3}{16}^*+\ang{\chi_{16}}}\\
&=\cS(4|\chi_8)_{\frac{k+3}{16}^*+\ang{\chi_{16},\ell_{64}}}\\
&=\cM(4|\chi_8)_{\frac k8^*+\ang{\chi_{16},\ell_{64}}}\sqrt{\eta}^3\end{align*}

It stays hold also for $k\in16\M+10$ and then for $k\in16\M+7$, ....
\end{proof}

\newpage

\section{Level 3-power cases}

\subsection{Formal weights $\bot,\top$}

Write $\bot=\ang{\frac13}$ and $\top=\ang3$.

Weight $\frac\bot3$ and $\frac\top3$ operators are exanpded by
\[\sqrt[3]{\eta^{\ang{\frac13}8}}\in\cS(1,9)_{(\frac43-\frac\bot3,c_3b_{27})}\6\sqrt[3]{\eta^{\ang38}}\in\cS(9,1)_{(\frac43-\frac\top3,c_{27}b_3)}\]

\begin{Lem}
\[\sqrt[3]{\eta^{\bot\ang{\frac13}}}\in\cM(1,9)_{(\frac13-\frac\bot3,\chi_3c_3)} \6 \sqrt[3]{\eta^{\top\ang3}}\in\cM(9,1)_{(\frac13-\frac\top3,\chi_3b_3)}\]
\[\sqrt[3]{\eta^{\nwarrow\ang{\frac13}}}\in\cM(1,9)_{(\frac13-\frac\bot3,\chi_9)} \6 \sqrt[3]{\eta^{\swarrow\ang3}}\in\cM(9,1)_{(\frac13-\frac\top3,\overline{\chi_9})}\]
\[\sqrt[3]{\eta^{\swarrow\ang{\frac13}}}\in\cM(1,9)_{(\frac13-\frac\bot3,\overline{\chi_9})} \6 \sqrt[3]{\eta^{\nwarrow\ang3}}\in\cM(9,1)_{(\frac13-\frac\top3,\chi_9)}\]
\end{Lem}
\begin{proof}
The first assertion follows from $\sqrt[3]{\eta^{\bot}}=\dfrac{\eta^{\bot\ang3}}{\sqrt[3]{\eta^{\bot\top}}}$ and
\[\sqrt[3]{\eta^{\bot\top}}^{\ang{\frac13}}=\dfrac{\eta^{\bot}\eta^{\top\ang{\frac19}}}{\sqrt[3]{\eta^{\ang{\frac13}8}}}\in\cM(1,9)_{\frac23+\frac\bot3}\]

The second assertion follows from $\iE_{\chi_9}^{\ang{\frac19}}=\sqrt[3]{\eta^{\nwarrow2}\eta^{\swarrow}}$.
\end{proof}

\begin{Lem}
\[\sqrt[9]{\eta^{\bot}}\in\cM(9)_{(\frac19+\frac\bot3,\chi_3c_{27})} \6 \sqrt[9]{\eta^{\top}}\in\cM(9)_{(\frac19+\frac\top3,\chi_3b_{27})}\]
\end{Lem}
\begin{proof}$\sqrt[9]{\eta^{\bot}}=\dfrac{\sqrt[9]{\eta^8}}{\sqrt[3]{\eta^{\top\ang{\frac13}}}}$ and $\sqrt[3]{\eta^{\top\ang{\frac13}}}=\dfrac{\sqrt[3]{\eta^{\ang{\frac13}8}}}{\eta^{\bot}}$.
\end{proof}

\vspace{0.5cm}

Formally let $\frac{\bot+\top+\nwarrow+\swarrow}3=0$ and make new operators of weight $\frac\nwarrow3$ and $\frac\swarrow3$ by
\[\sqrt[9]{\eta^{\nwarrow}}\in\cM(9)_{(\frac19+\frac\nwarrow3,\chi_{27})} \6 \sqrt[9]{\eta^{\swarrow}}\in\cM(9)_{(\frac19+\frac\swarrow3,\overline{\chi_{27}})}\]\\

Next, put $[\eta,\frac k9]=\eta^{\bot\ang{\frac13}}|(\BM 1&k\\0&1\EM)$ then
\begin{align*}\z [\eta,\frac19]&=(\iE_{\chi_3}-3\eta^{\top})^{\ang{\frac13}}|(\BM 1&1\\0&1\EM)\\
&=(\eta^{\bot}+9\eta^{\top}-3\sqrt[3]{\eta^{\nwarrow}\eta^{\swarrow}\eta^{\top}})|(\BM 1&1\\0&1\EM)\\
&=\eta^{\nwarrow}+9\cdot1^{\frac13}\eta^{\top}-3\cdot1^{\frac19}\sqrt[3]{\eta^{\swarrow}\eta^{\nwarrow}\eta^{\top}}\in\cM(\Gamma(9))_1\end{align*}
and
\[\z \sqrt[3]{[\eta,\frac19][\eta,\frac49][\eta,\frac79]}=\eta^{\nwarrow}.\]

\newpage

Moreover put
\[\z [E,\frac k9]=\BC \frac{1^{\frac23}-1^{\frac13}}3\sqrt[3]{[\eta,\frac k9]}+\frac{1-1^{\frac23}}3\big(\sqrt[3]{[\eta,\frac{k+3}9]}+\sqrt[3]{[\eta,\frac{k-3}9]}\big)&\text{for }k=1,4,7\\
\frac{1^{\frac13}-1^{\frac23}}3\sqrt[3]{[\eta,\frac k9]}+\frac{1-1^{\frac13}}3\big(\sqrt[3]{[\eta,\frac{k+3}9]}+\sqrt[3]{[\eta,\frac{k-3}9]}\big)&\text{for }k=8,5,2\EC\]
\begin{Lem}
\[\z [E,\frac19]\in\cM(9)_{(\frac13-\frac\nwarrow3,\chi_3)}\6[E,\frac89]\in\cM(9)_{(\frac13-\frac\swarrow3,\chi_3)}\]
\[\z [E,\frac49]\in\cM(9)_{(\frac13-\frac\nwarrow3,\chi_9)}\6[E,\frac59]\in\cM(9)_{(\frac13-\frac\swarrow3,\overline{\chi_9})}\]
\[\z [E,\frac79]\in\cM(9)_{(\frac13-\frac\nwarrow3,\overline{\chi_9})}\6[E,\frac29]\in\cM(9)_{(\frac13-\frac\swarrow3,\chi_9)}\]
\end{Lem}

\vspace{0.5cm}

At last, remark
\[\z f_{27,3}=\sqrt[9]{\eta^{\top}}f_{9,1}^{\ang3}\]
\[f_{27,9}=\sqrt[9]{\eta^{\top}}\sqrt[3]{\eta^{\top\ang3}}\]
\[\z f_{27,15}=\sqrt[9]{\eta^{\top}}f_{9,5}^{\ang3}\]
\[\z f_{27,21}=\sqrt[9]{\eta^{\top}}f_{9,7}^{\ang3}\]
\[\z f_{27,1}-f_{27,17}-f_{27,19}=\sqrt[9]{\eta^{\bot}}f_{9,1}^{\ang{\frac13}}\]
\[\z f_{27,5}-f_{27,13}-f_{27,23}=\sqrt[9]{\eta^{\bot}}f_{9,5}^{\ang{\frac13}}\]
\[\z f_{27,7}-f_{27,11}-f_{27,25}=\sqrt[9]{\eta^{\bot}}f_{9,7}^{\ang{\frac13}}\]\\
Acting $(\BM 1&1\\0&1\EM)$ on the fifth identity yields
\[\z f_{27,1}-1^{\frac13}f_{27,17}-1^{\frac23}f_{27,19}=\sqrt[9]{\eta^{\nwarrow}}\frac13\big(\sqrt[3]{[\eta,\frac19]}+\sqrt[3]{[\eta,\frac49]}+\sqrt[3]{[\eta,\frac79]}\big)\]

\newpage

\subsection{s-form}

Put
\[\z [s,\frac03]=\eta^{\bot\ang3}+3\eta^{\top}\]
and $[s,\frac k3]=[s,\frac03]\big|(\BM 1&k\\0&1\EM)$ then
\[\z \iE_{\chi_3}=\sqrt[3]{[s,\frac03][s,\frac13][s,\frac23]}\]
since $\eta^{\bot2}+9\eta^{\bot}\eta^{\top}+27\eta^{\top2}=\eta^{\nwarrow}\eta^{\swarrow}$.

\begin{Prop}\label{lev9}We have identities in $\cM(3)_{\frac23}$

\begin{align*}\sqrt[3]{\eta^{\nwarrow}\eta^{\swarrow}}&\z =\sqrt[3]{\eta^{\bot2}}+3\sqrt[3]{[s,\frac03]\eta^{\top\ang3}}\\
&\z =\sqrt[3]{[s,\frac13]\eta^{\nwarrow\ang3}}+(1-1^{\frac13})\sqrt[3]{[s,\frac03]\eta^{\top\ang3}}\\
&\z =\sqrt[3]{\eta^{\bot}[s,\frac03]}^{\ang{\frac13}}+3\sqrt[3]{\eta^{\top2}}\\
&\z =\sqrt[3]{\eta^{\nwarrow}[s,\frac13]}^{\ang{\frac13}}+3\cdot1^{\frac23}\sqrt[3]{\eta^{\top2}}\end{align*}
\end{Prop}

\begin{Lem}
\[\z \sqrt[3]{[s,\frac03]^{\ang{\frac13}}}\in\cM(3,9)_{(\frac13+\frac\bot3,\chi_3)} \6 \sqrt[3]{[s,\frac03]}\in\cM(9,3)_{(\frac13+\frac\top3,\chi_3)}\]
\[\z \sqrt[3]{[s,\frac13]^{\ang{\frac13}}}\in\cM(3,9)_{(\frac13+\frac\bot3,\overline{\chi_9})} \6 \sqrt[3]{[s,\frac23]}\in\cM(9,3)_{(\frac13+\frac\top3,\chi_9)}\]
\[\z \sqrt[3]{[s,\frac23]^{\ang{\frac13}}}\in\cM(3,9)_{(\frac13+\frac\bot3,\chi_9)} \6 \sqrt[3]{[s,\frac13]}\in\cM(9,3)_{(\frac13+\frac\top3,\overline{\chi_9})}\]
\end{Lem}

\begin{Prop}We have identities in $\cM(9)_{\frac23-\frac\top3}$
\[\z \sqrt[3]{[s,\frac13][s,\frac23]}+1^{\frac13}\sqrt[3]{\eta^{\nwarrow}\eta^{\swarrow\ang3}}+1^{\frac23}\sqrt[3]{\eta^{\swarrow}\eta^{\nwarrow\ang3}}=0\]
\[\z \sqrt[3]{[s,\frac03]^2}+1^{\frac23}\sqrt[3]{\eta^{\nwarrow}\eta^{\swarrow\ang3}}+1^{\frac13}\sqrt[3]{\eta^{\swarrow}\eta^{\nwarrow\ang3}}=0\]
\[\z \sqrt[3]{[s,\frac03]^2}-\sqrt[3]{[s,\frac13][s,\frac23]}=3\sqrt[3]{\eta^{\bot}\eta^{\top\ang3}}\]
\end{Prop}

\vspace{0.5cm}

Moreover put $[s,\frac k9]=[s,\frac03]^{\ang{\frac13}}\big|(\BM 1&k\\0&1\EM)$ for $k$ prime to 3 and
\[\z [S,\frac k9]=\BC \frac{1^{\frac13}-1^{\frac23}}3\sqrt[3]{[s,\frac k9]}+\frac{1-1^{\frac13}}3\big(\sqrt[3]{[s,\frac{k+3}9]}+\sqrt[3]{[s,\frac{k-3}9]}\big)&\text{for }k=1,4,7\\
\frac{1^{\frac23}-1^{\frac13}}3\sqrt[3]{[s,\frac k9]}+\frac{1-1^{\frac23}}3\big(\sqrt[3]{[s,\frac{k+3}9]}+\sqrt[3]{[s,\frac{k-3}9]}\big)&\text{for }k=8,5,2\EC\]
\begin{Lem}
\[\z [S,\frac79]\in\cM(9)_{(\frac13+\frac\nwarrow3,\chi_3)}\6[S,\frac29]\in\cM(9)_{(\frac13+\frac\swarrow3,\chi_3)}\]
\[\z [S,\frac49]\in\cM(9)_{(\frac13+\frac\nwarrow3,\chi_9)}\6[S,\frac59]\in\cM(9)_{(\frac13+\frac\swarrow3,\overline{\chi_9})}\]
\[\z [S,\frac19]\in\cM(9)_{(\frac13+\frac\nwarrow3,\overline{\chi_9})}\6[S,\frac89]\in\cM(9)_{(\frac13+\frac\swarrow3,\chi_9)}\]
\end{Lem}

\newpage

Put
\[\z [u,\frac03]=\frac13\big(\sqrt[3]{[s,\frac03]}+\sqrt[3]{[s,\frac13]}+\sqrt[3]{[s,\frac23]}\big)\]
\[\z [u,\frac13]=\frac13\big(\sqrt[3]{[s,\frac03]}+1^{\frac23}\sqrt[3]{[s,\frac13]}+1^{\frac13}\sqrt[3]{[s,\frac23]}\big)\]
\[\z [u,\frac23]=-\frac13\big(\sqrt[3]{[s,\frac03]}+1^{\frac13}\sqrt[3]{[s,\frac13]}+1^{\frac23}\sqrt[3]{[s,\frac23]}\big)\]
then
\[\z [u,\frac03][u,\frac13][u,\frac23]=\eta^{\top\ang3}\]

\[\z \sqrt[3]{[u,\frac03]^2}\big(\sqrt[3]{[s,\frac13]}-\sqrt[3]{[s,\frac23]}\big)=\sqrt[3]{[u,\frac13][u,\frac23]}\big(1^{\frac13}\sqrt[3]{\eta^{\nwarrow}}-1^{\frac23}\sqrt[3]{\eta^{\swarrow}}\big)\]\\

We see
\[\iE_{\chi_{27}}=1+(1^{\frac19}-1^{\frac29})\VS{n\in\N}(\chi_{27}*{\tt 1}_\N)(n)q^n\]
and
\[\iE_{\chi_{27}}^{\ang{\frac19}}=\sqrt[9]{\eta^{\nwarrow2}\eta^{\swarrow}}E\]
where
\[\z E=\frac{1^{\frac23}-1^{\frac13}}3\sqrt[3]{[E,\frac29][S,\frac19]}+\frac{1-1^{\frac23}}3\big(\sqrt[3]{[E,\frac59][S,\frac49]}+\sqrt[3]{[E,\frac89][S,\frac79]}\big)\]
\[=\sqrt[9]{\eta^{\nwarrow\ang{\frac13}2}\eta^{\swarrow\ang{\frac13}}}\sqrt[3]{\sqrt[3]{\eta^{\nwarrow}\eta^{\swarrow2}}+(1^{\frac19}-1)^3\sqrt[3]{\eta^{\nwarrow2}\eta^{\top}}-3\cdot1^{\frac19}(1^{\frac29}-1^{\frac19}+1)(1^{\frac29}+1)\sqrt[3]{\eta^{\swarrow}\eta^{\top2}}}\]\\

Put
\[\z E_\bot=\big\{\sqrt[3]{\eta^{\bot\ang{\frac13}}},\sqrt[3]{\eta^{\nwarrow\ang{\frac13}}},\sqrt[3]{\eta^{\swarrow\ang{\frac13}}}\big\} \6 S_\bot=\big\{\sqrt[3]{[s,\frac03]^{\ang{\frac13}}},\sqrt[3]{[s,\frac13]^{\ang{\frac13}}},\sqrt[3]{[s,\frac23]^{\ang{\frac13}}}\big\}\]
\[\z E_\top=\big\{\sqrt[3]{\eta^{\top\ang3}},\sqrt[3]{\eta^{\swarrow\ang3}},\sqrt[3]{\eta^{\nwarrow\ang3}}\big\} \6 S_\top=\big\{\sqrt[3]{[s,\frac03]},\sqrt[3]{[s,\frac13]},\sqrt[3]{[s,\frac23]}\big\}\]
\[\z E_\nwarrow=\big\{\sqrt[3]{[E,\frac19]},\sqrt[3]{[E,\frac49]},\sqrt[3]{[E,\frac79]}\big\} \6 S_\nwarrow=\big\{\sqrt[3]{[S,\frac19]},\sqrt[3]{[S,\frac49]},\sqrt[3]{[S,\frac79]}\big\}\]
\[\z E_\swarrow=\big\{\sqrt[3]{[E,\frac89]},\sqrt[3]{[E,\frac59]},\sqrt[3]{[E,\frac29]}\big\} \6 S_\swarrow=\big\{\sqrt[3]{[S,\frac89]},\sqrt[3]{[S,\frac59]},\sqrt[3]{[S,\frac29]}\big\}\]

\newpage

\subsection{Structure theorems with $\chi_{27}$}

\begin{Lem}Let $R$ be direct sum of
\[\cM(\Gamma(9))_{\frac13\M}\]
\[\bigoplus\C\big[\sqrt[3]{\eta^{\bot}},\sqrt[3]{\eta^{\top}}\big]\Big(\sqrt[9]{\eta^{\bot}\eta^{\top}\eta^{\nwarrow2}\eta^{\swarrow2}}E_\nwarrow E_\swarrow\cup\sqrt[9]{\eta^{\bot2}\eta^{\top2}\eta^{\nwarrow}\eta^{\swarrow}}E_\bot E_\top\;\cup\]
\[\sqrt[9]{\eta^{\bot2}\eta^{\top}}S_\bot E_\top\cup\sqrt[9]{\eta^{\top2}\eta^{\bot}}S_\top E_\bot\Big)\]
\[\bigoplus\C\big[\sqrt[3]{\eta^{\nwarrow}},\sqrt[3]{\eta^{\swarrow}}\big]\sqrt[3]{\eta^{\bot}}^{\{0,1,2\}}\Big(\sqrt[9]{\eta^{\top}\eta^{\nwarrow}\eta^{\swarrow}}S_\bot\cup\sqrt[9]{\eta^{\top2}\eta^{\nwarrow2}\eta^{\swarrow2}}E_\bot\Big)\]
\[\bigoplus\C\big[\sqrt[3]{\eta^{\nwarrow}},\sqrt[3]{\eta^{\swarrow}}\big]\sqrt[3]{\eta^{\top}}^{\{0,1,2\}}\Big(\sqrt[9]{\eta^{\bot}\eta^{\nwarrow}\eta^{\swarrow}}S_\top\cup\sqrt[9]{\eta^{\bot2}\eta^{\nwarrow2}\eta^{\swarrow2}}E_\top\Big)\]Then $\cM(27)_{\ang{1,\chi_9}}=R|_\M$.
\end{Lem}
\begin{proof}Note $\sqrt[3]{\eta^{\bot}}=1-q^{\frac13}+\cdots$ and $\sqrt[3]{\eta^{\top}}=q^{\frac19}+\cdots$. Compute
\[\z [E,\frac19][E,\frac89]=1-(1^{\frac29}+1^{\frac79})q^{\frac1{27}}+(1-1^{\frac19}-1^{\frac89})q^{\frac2{27}}+\cdots\]
\[\z [E,\frac49][E,\frac59]=1-(1^{\frac19}+1^{\frac89})q^{\frac1{27}}+(1-1^{\frac49}-1^{\frac59})q^{\frac2{27}}+\cdots\]
\[\z [E,\frac79][E,\frac29]=1-(1^{\frac49}+1^{\frac59})q^{\frac1{27}}+(1-1^{\frac29}-1^{\frac79})q^{\frac2{27}}+\cdots\]
Hence the natural map
\[\bigoplus\z \C\big[\sqrt[3]{\eta^{\bot}},\sqrt[3]{\eta^{\top}}\big]\big\{[E,\frac19][E,\frac89],[E,\frac49][E,\frac59],[E,\frac79][E,\frac29]\big\}\to\cM(9)_{\ang{\frac13,\chi_3}-\frac{\nwarrow+\swarrow}3}\]
is injective. So is
\[\bigoplus\z \C\big[\sqrt[3]{\eta^{\bot}},\sqrt[3]{\eta^{\top}}\big]\big\{[E,\frac19][E,\frac29],[E,\frac49][E,\frac89],[E,\frac79][E,\frac59]\big\}\to\cM(9)_{\ang{\frac13,\chi_3}+\chi_9-\frac{\nwarrow+\swarrow}3}\]
Gathering three parts, so is
\[\bigoplus\C\big[\sqrt[3]{\eta^{\bot}},\sqrt[3]{\eta^{\top}}\big]E_\nwarrow E_\swarrow\to\cM(9)_{\ang{\frac13,\chi_9}-\frac{\nwarrow+\swarrow}3}.\]

Similarly the natural map
\[\bigoplus\C\big[\sqrt[3]{\eta^{\nwarrow}},\sqrt[3]{\eta^{\swarrow}}\big]S_\bot\to\cM(3,9)_{\ang{\frac13,\chi_9}+\frac\bot3}\]
is injective. So is
\[\bigoplus\C\big[\sqrt[3]{\eta^{\nwarrow}},\sqrt[3]{\eta^{\swarrow}}\big]\sqrt[3]{\eta^{\bot}}S_\bot\to\cM(3,9)_{\ang{\frac13,\chi_9}+c_9+\frac\bot3}\]
Gathering three parts, so is
\[\bigoplus\C\big[\sqrt[3]{\eta^{\nwarrow}},\sqrt[3]{\eta^{\swarrow}}\big]\sqrt[3]{\eta^{\bot}}^{\{0,1,2\}}S_\bot\to\cM(9)_{\ang{\frac13,\chi_9}+\frac\bot3}.\]

Gathering all parts, we obtain injectivety of the natural map $R\to\cM(27)_{\ang{\frac13,\chi_9}}$.

The dimension formula states for $k\ge2$
\[\dim\cM(27)_{k+\ang{\chi_9}}=243k-189\]

Hilbert function is
\[\frac{(1+t^{\frac13}+t^{\frac23})^2+9t^{\frac43}\cdot2+9t\cdot2+(1+t^{\frac13}+t^{\frac23})(3t^{\frac23}+3t)\cdot2}{(1-t^{\frac13})^2}\]
\[=\frac{1+2t^{\frac13}+9t^{\frac23}+32t+31t^{\frac43}+6t^{\frac53}}{(1-t^{\frac13})^2}\]
\end{proof}

\newpage

\begin{Lem}Let $X$ be cup of
\[\sqrt[9]{\eta^{\nwarrow}}\Big(\sqrt[9]{\eta^{\nwarrow}\eta^{\swarrow}}S_{\nwarrow}E_{\swarrow}\cup\sqrt[9]{\eta^{\bot2}}S_{\bot}E_{\nwarrow}\cup\sqrt[9]{\eta^{\top2}}S_{\top}E_{\swarrow}\;\cup\]
\[\sqrt[9]{\eta^{\bot2}\eta^{\top}\eta^{\nwarrow}\eta^{\swarrow}}E_\bot E_\nwarrow\cup\sqrt[9]{\eta^{\bot}\eta^{\top2}\eta^{\nwarrow}\eta^{\swarrow}}E_\top E_\nwarrow\Big)\]
\[\sqrt[9]{\eta^{\swarrow2}}\Big(\sqrt[9]{\eta^{\bot}}E_\bot S_\swarrow\cup\sqrt[9]{\eta^{\top}}E_\top S_\swarrow\Big)\]\[\sqrt[9]{\eta^{\bot}\eta^{\top}}S_{\swarrow}\sqrt[3]{\eta^{\swarrow}}^{\{0,1,2\}}\cup\sqrt[9]{\eta^{\bot2}\eta^{\top2}}E_{\nwarrow}\sqrt[3]{\eta^{\nwarrow}}^{\{0,1,2\}}\]
Then $\cM(27)_{\ang{1,\chi_9}+\chi_{27}}=\bigoplus\C\big[\sqrt[3]{\eta^{\bot}},\sqrt[3]{\eta^{\top}}\big]X\big|_\M$.
\end{Lem}
\begin{proof}The dimension formula states for $k\ge2$
\[\dim\cM(27)_{k+\ang{\chi_9}+\chi_{27}}=243k-189\]

Hilbert function is
\[\frac{t^{\frac19}\big(9t^{\frac89}\cdot3+3t^{\frac59}(1+t^{\frac13}+t^{\frac23})+9t^{\frac{11}9}\cdot2\big)+t^{\frac29}\big(9t^{\frac79}\cdot2+3t^{\frac79}(1+t^{\frac13}+t^{\frac23})\big)}{(1-t^{\frac13})^2}\]
\[=\frac{3t^{\frac23}+51t+24t^{\frac43}+3t^{\frac53}}{(1-t^{\frac13})^2}\]

\end{proof}

By the above two Lemmas, we get a generator of $\cM(\Gamma(27))_\M$.\\

\end{document}